\documentclass[11pt]{amsart}

\usepackage[utf8]{inputenc}

\usepackage{color}
\usepackage{enumerate}

\usepackage[all]{xy}

\usepackage{amsaddr}
\usepackage{amssymb}
\usepackage{a4wide}
\usepackage{fixme}

\usepackage{marginnote}

\DeclareMathOperator{\Ann}{Ann}
\DeclareMathOperator{\identity}{id}
\newcommand{\Z}{\mathbb{Z}}
\newcommand{\N}{\mathbb{N}}
\newcommand{\F}{\mathbb{F}}
\newcommand{\R}{\mathbb{R}}
\newcommand{\C}{\mathbb{C}}

\usepackage{xargs}                      
\usepackage[pdftex,dvipsnames]{xcolor}  

\usepackage{hyperref}
\hypersetup{final} 

\title[Prime group graded rings]{Prime group graded rings with applications to \\ partial crossed products and Leavitt path algebras}

\newtheorem{thm}{Theorem}[section]
\newtheorem{prop}[thm]{Proposition}
\newtheorem{lem}[thm]{Lemma}
\newtheorem{cor}[thm]{Corollary}
\theoremstyle{definition}
\newtheorem{defi}[thm]{Definition}
\newtheorem{exa}[thm]{Example}
\newtheorem{rem}[thm]{Remark}

\DeclareMathOperator{\Supp}{Supp}
\DeclareMathOperator{\Aut}{Aut}
\DeclareMathOperator{\Stab}{Stab}

\author{Daniel L\"{a}nnstr\"{o}m}
\address{Department of Mathematics and Natural Sciences,
Blekinge Institute of Technology,
SE-37179 Karlskrona, Sweden}
\author{Patrik Lundstr\"{o}m}
\address{Department of Engineering Science,
University West, SE-46186 Trollh\"{a}ttan, Sweden}
\author{Johan \"{O}inert}
\address{Department of Mathematics and Natural Sciences,
Blekinge Institute of Technology,
SE-37179 Karlskrona, Sweden}
\author{Stefan Wagner}
\address{Department of Mathematics and Natural Sciences,
Blekinge Institute of Technology,
SE-37179 Karlskrona, Sweden}

\email{daniel.lannstrom@bth.se}
\email{patrik.lundstrom@hv.se}
\email{johan.oinert@bth.se}
\email{stefan.wagner@bth.se}

\makeatletter
\@namedef{subjclassname@2020}{%
  \textup{2020} Mathematics Subject Classification}
\makeatother



\begin{document}

\begin{abstract}
In this article we generalize a classical
result by Passman on primeness of unital strongly group graded rings to the class of nearly epsilon-strongly group graded rings
which are not necessarily unital.
Using this result, we obtain 
(i) 
a characterization of 
prime $s$-unital strongly group graded rings, and,
in particular, of infinite matrix rings and of group rings
over $s$-unital rings, thereby generalizing a well-known result by Connell;
(ii) characterizations of prime $s$-unital partial skew group rings and of prime unital partial
crossed products;
(iii) a generali\-zation of the well-known characterizations of prime Leavitt path algebras, by Larki and by Abrams-Bell-Rangaswamy.
\end{abstract}

\maketitle


\tableofcontents

\section{Introduction}
\label{Sec:Intro}

Let $S$ be a ring. By this we mean that $S$ is associative 
but not necessarily unital. 
Unless otherwise stated, ideals of $S$ are assumed to
be two-sided.
Recall that a proper ideal $P$ of $S$ is called {\it prime}
if for all ideals $A$ and $B$ of $S$,
$A \subseteq P$ or $B \subseteq P$ holds
whenever $AB \subseteq P$.
The ring $S$ is called \emph{prime} if $\{ 0 \}$ is a prime ideal of $S$. 
The class of prime rings contains many well-known 
constructions, for instance left or right primitive rings,
simple rings and matrix rings over integral domains.
Prime rings also generalize integral domains to a non-commutative setting. Indeed, a commutative ring is prime if and only if it is an integral domain. 

Throughout this article, $G$ denotes a multiplicatively written group with 
neutral element $e$.
Recall that $S$ is called {\it $G$-graded},
if for each $x \in G$ there is an additive subgroup $S_x$
of $S$ such that $S = \bigoplus_{x \in G} S_x$, as
additive groups, and for all $x,y \in G$, the inclusion
$S_x S_y \subseteq S_{xy}$ holds.
If in addition, $S_x S_y = S_{xy}$ holds for all $x,y\in G$,
then $S$ is said to be {\it strongly $G$-graded}.
An interesting problem, studied for the past 50 years,
concerns finding necessary and sufficient conditions
for different classes of group graded rings to be prime,
see
\cite{AbHa93,
cohen1983group,
connell1963group,
nastasescu1982graded,
nastasescu2004methods,
passman1962nil,
passman1970radicals,
passman2011algebraic,
passman1983semiprime,
PassmanCancellative,
passman1984infinite}.
In the case when $S$ is unital and strongly $G$-graded,
Passman has
completely solved this problem by proving the following
rather involved result:
\begin{thm}[{Passman \cite[Thm.~1.3]{passman1984infinite}}]\label{maintheorem}
Suppose that $S$ is a unital and strongly $G$-graded ring.
Then $S$ is not prime if and only if there exist:
\begin{enumerate}[{\rm (i)}]
    \item subgroups $N \lhd H \subseteq G$ with $N$ finite,
    \item an $H$-invariant ideal $I$ of $S_e$ such that
    $I^x I = \{ 0 \}$ for every $x \in G \setminus H$, and
    \item nonzero $H$-invariant 
    ideals $\tilde{A}, \tilde{B}$ of $S_N$ such that
    $\tilde{A},\tilde{B} \subseteq I S_N$ and
    $\tilde{A}\tilde{B} = \{ 0 \}$.
\end{enumerate}
\end{thm}
Let us briefly explain the notation used in the formulation of
this result as well as some technical aspects of
Passman's proof of it.
Suppose that $I$ is an ideal of the subring $S_e$.
If $x \in G$, then $I^x$ denotes the $S_e$-ideal
$S_{x^{-1}} I S_x$.
Let $H,N$ be subgroups of $G$.
The ideal $I$ is called \emph{$H$-invariant} if
$I^x \subseteq I$
holds
for every
$x \in H$;
$S_N$ denotes $\bigoplus_{x \in N} S_x$, which is clearly a subring of $S$.
In \cite{passman1984infinite} Passman provided a
``combinatorial'' proof of Theorem~\ref{maintheorem}
by combining two main ideas. 
First, a coset counting
method, also known as the ``$\Delta$-method'', 
developed by Passman~\cite{passman1962nil} and Connell~\cite{connell1963group}, secondly, the
``bookkeeping procedure'' introduced by 
Passman in \cite{passman1983semiprime} which involves a careful
study of the action of the group $G$ on the 
lattice of ideals of $S_e$.
In \cite{passman1984infinite} 
Passman also showed that analogous criteria
exist for {\it semiprimeness} of strongly group graded rings.
In this article, however, only the concept of primeness
will be studied.

In a subsequent article \cite{PassmanCancellative} Passman
obtained an analogue of Theorem~\ref{maintheorem}  
for the slightly larger class of unital $G$-graded rings which are
{\it cancellative}, that is, rings $S$ having the property that 
for all $x,y \in G$ and all homogeneous subsets 
$U,V \subseteq S$, the implication 
$U S_x S_y V = \{ 0 \} \Rightarrow U S_{xy} V = \{ 0 \}$ holds.
It is clear that strongly $G$-graded rings are cancellative.
However, not all cancellative $G$-graded rings are 
strongly graded. For instance, {\it the first Weyl algebra} 
$A_1(\F)$, over a field $\F$ of characteristic zero,
is a $\mathbb{Z}$-graded ring which is not strongly graded
but still cancellative, since it is a {\it domain}
(see e.\,g. \cite[Chap.~2]{goodearl2004}).

The motivation for the present article is the observation
that many important examples of group graded rings are {\it not}
cancellative, but may still be prime. 
Indeed, suppose that $R$ is a unital ring
and let $S := M_n(R)$ denote the ring of $n \times n$-matrices
with entries in $R$.
Then it is easy to see that $S$ is prime if and 
only if $R$ is prime. On the other hand, one can
construct group gradings on $S$ that are not 
cancellative. Consider the case $n=2$. 
Let $e_{ij}$ denote the matrix with $1$ in position $ij$
and zeros elsewhere.
If $G=\mathbb{Z}$ and we put
$S_0 := R e_{11} + R e_{22}$, $S_1 := R e_{12}$,
$S_{-1} := R e_{21}$, and $S_x := \{ 0 \}$, for
$x \in \mathbb{Z} \setminus \{ 0,1,-1 \}$, then
this defines a $G$-grading on $S$ satisfying
$S_1 \cdot S_1 S_{-1} \cdot S_0 = \{ 0 \}$ but 
$S_1 \cdot S_0 \cdot S_0 = R e_{12} \neq \{ 0 \}$.
This shows that the grading is not cancellative.
In a similar fashion, one may define non-cancellative
$\mathbb{Z}$-gradings on $M_n(R)$, for every $n \geq 2$.

This phenomenon is not confined to rings of matrices.
In fact, these structures can be considered as special
cases of so-called {\it Leavitt path algebras} $L_R(E)$,
over a unital ring $R$, defined by directed graphs $E$ (for the details, 
see Section~\ref{Sec:LPA}).
All Leavitt path algebras carry a canonical 
$\mathbb{Z}$-grading. 
One can show that with this grading, a Leavitt path
algebra defined by a finite graph $E$, is cancellative 
if and only if it is strongly graded
(see Proposition~\ref{prop:EpsCancellativeStrong} and
Proposition~\ref{prop:LPAepsilon}).
However, it is easy to give examples of Leavitt path algebras
which are not strongly $\mathbb{Z}$-graded.
Nevertheless, the question of primeness
of such structures has been completely resolved
in the case when $R$ is commutative,
for any directed graph $E$, by Larki~\cite{larki2015ideal} 
building upon previous work by Abrams, Bell and Rangaswamy 
\cite[Thm. 1.4]{abrams2014prime}.
Their results involve a certain ``connectedness''
property on the set $E^0$ of vertices of $E$. 
Namely, a directed graph $E$
is said to satisfy  \emph{condition (MT-3)} if for all $u,v \in E^0$,
there exist $w \in E^0$ and paths from $u$ to $w$
and from $v$ to $w$.

\begin{thm}[{Larki \cite[Prop. 4.5]{larki2015ideal}}]\label{larki}
Suppose that $E$ is a directed graph and that $R$ is a unital commutative ring.
Then $L_R(E)$ is prime if and only if 
$R$ is an integral domain and $E$ satisfies condition (MT-3).
\end{thm}

The purpose of the present article is to prove a 
primeness result (see Theorem~\ref{thm:mainNew}) 
that holds for a class of group graded rings that contains all unital strongly group graded rings
as well as many types of group graded rings that are not
cancellative.
The rings that we consider are the 
nearly epsilon-strongly group graded rings
introduced by Nystedt and \"{O}inert in
\cite{nystedt2017epsilon}.
Recall that a, not necessarily unital, 
$G$-graded ring $S$ is called 
\emph{nearly epsilon-strongly $G$-graded}
if for every $x \in G$ and every $s \in S_x$,
there exist $\epsilon_x(s) \in S_x S_{x^{-1}}$ and
$\epsilon_x'(s) \in S_{x^{-1}} S_x$
such that the equalities 
$\epsilon_x(s) s = s = s \epsilon_{x}'(s)$ hold.
Note that
every nearly epsilon-strongly $G$-graded ring $S$ is necessarily \emph{non-degenerately $G$-graded}, i.\,e.
for every $x \in G$ and every 
nonzero $s \in S_x$,
we have $s S_{x^{-1}} \ne \{0\}$ and $S_{x^{-1}} s \ne \{0\}$.
In loc. cit. it is shown 
that every Leavitt path algebra,
equipped with its canonical $\mathbb{Z}$-grading,
is nearly epsilon-strongly graded.
In addition, $s$-unital partial skew group rings and unital partial crossed products
are nearly epsilon-strongly graded
(see Section~\ref{Sec:Partial}).

Here is the main result of this article:

\begin{thm}\label{thm:mainNew}
Suppose that $G$ is a group and that
$S$ is a 
$G$-graded ring.
Consider the following five assertions:
\begin{enumerate}[{\rm (a)}]
	\item $S$ is not prime.
	
	\item There exist:
	\begin{enumerate}[{\rm (i)}]
\item subgroups $N \lhd H \subseteq G$,
\item an
$H$-invariant ideal $I$ of $S_e$ such that 
$I^x I = \{0\}$ for every $x \in G \setminus H$, and
\item nonzero ideals $\tilde{A}, \tilde{B}$ of $S_N$ such that
$\tilde{A}, \tilde{B} \subseteq I S_N$ and
$\tilde A S_H \tilde B = \{0\}$.
\end{enumerate}

	\item There exist:
	\begin{enumerate}[{\rm (i)}]
\item subgroups $N \lhd H \subseteq G$ with $N$ finite,
\item an
$H$-invariant ideal $I$ of $S_e$ such that 
$I^x I = \{0\}$ for every $x \in G \setminus H$, and
\item nonzero ideals $\tilde{A}, \tilde{B}$ of $S_N$ such that
$\tilde{A}, \tilde{B} \subseteq I S_N$ and
$\tilde A S_H \tilde B = \{0\}$.
\end{enumerate}

\item There exist:
	\begin{enumerate}[{\rm (i)}]
\item subgroups $N \lhd H \subseteq G$ with $N$ finite,
\item an
$H$-invariant ideal $I$ of $S_e$ such that 
$I^x I = \{0\}$ for every $x \in G \setminus H$, and
\item nonzero $H$-invariant ideals $\tilde{A}, \tilde{B}$ of $S_N$ such that
$\tilde{A}, \tilde{B} \subseteq I S_N$ and
$\tilde A S_H \tilde B = \{0\}$.
\end{enumerate}
	
	\item There exist:
	\begin{enumerate}[{\rm (i)}]
\item subgroups $N \lhd H \subseteq G$ with $N$ finite,
\item an
$H$-invariant ideal $I$ of $S_e$ such that 
$I^x I = \{0\}$ for every $x \in G \setminus H$, and
\item nonzero $H/N$-invariant ideals $\tilde{A}, \tilde{B}$ of $S_N$ such that
$\tilde{A}, \tilde{B} \subseteq I S_N$ and
$\tilde A \tilde B = \{0\}$.
\end{enumerate}

\end{enumerate}
The following assertions hold:
\begin{enumerate}[{\rm (1)}]
    \item If $S$ is non-degenerately $G$-graded,  
    then  
{\rm (e)}$\Longrightarrow${\rm (d)}$\Longrightarrow${\rm (c)}$\Longrightarrow${\rm (b)}$\Longrightarrow${\rm (a)}.

\item If $S$ is nearly epsilon-strongly $G$-graded, then
{\rm (a)}$\Longleftrightarrow${\rm (b)}$\Longleftrightarrow${\rm (c)}$\Longleftrightarrow${\rm (d)}$\Longleftrightarrow${\rm (e)}.
\end{enumerate}
\end{thm}

Let us make four 
remarks on Theorem~\ref{thm:mainNew}.
First of all, this result is applicable to 
rings which are {\it not necessarily unital}.
Secondly, 
unital strongly $G$-graded rings (see Lemma~\ref{lem:strongSunital}) and cancellatively $G$-graded rings (see \cite[Lem.~1.2]{PassmanCancellative}) satisfy $r.\Ann_S(S_x)=\{0\}$, and $l.\Ann_S(S_x)=\{0\}$, for every $x\in G$.
However, many important classes of group graded rings rarely satisfy such annihilator conditions, for instance Leavitt path algebras
\cite{
abrams2017leavitt,
abrams2005leavitt,
abrams2014prime,
ara2007nonstable,
larki2015ideal,
tomforde2009leavitt}
and partial crossed products \cite{dokuchaev2011partial,dokuchaev2005associativity,dokuchaev2008crossed}.
Thus, Theorem~\ref{thm:mainNew} allows us to consider classes of rings which are unreachable by the results of \cite{PassmanCancellative,passman1984infinite}. 
Thirdly, we would like to motivate why assertions (b), (c) and (d) appear in Theorem~\ref{thm:mainNew}.
By allowing $N$ to be infinite, assertion (b) creates more flexibility when attempting to prove that $S$ is non-prime.
Assertion (c) is identical to the assertion in \cite[Thm.~2.3]{PassmanCancellative},
and assertion (d) is essentially identical to the assertion in \cite[Thm.~1.3]{passman1984infinite}.
Finally, it might be possible to generalize assertion (2) of Theorem~\ref{thm:mainNew} beyond the class of nearly epsilon-strongly graded rings (see Remark \ref{rem:12_8}).

Here is a detailed outline of this article.

In Section~\ref{Sec:Prel}, we state our conventions on groups, rings
and modules. We also provide preliminary 
results on different types of graded rings such as
epsilon-strongly graded rings, nearly epsilon-strongly
graded rings and cancellatively graded rings.
In Section~\ref{Sec:Invariant}, we
consider $H$-invariant ideals and record some of their basic properties.
In Section~\ref{Sec:GradedPrime}, we obtain a one-to-one correspondence between graded ideals of $S$ and
$G$-invariant ideals of the principal component $S_e$.
We also give a characterization of prime nearly epsilon-strongly $G$-graded rings in the case when $G$ is an ordered group.
In Section~\ref{Sec:Easy}, 
we prove the implication (b)$\Rightarrow$(a) of Theorem~\ref{thm:mainNew} for non-degenerately $G$-graded rings.
In Section~\ref{Sec:PassmanPairs},
we obtain some technical results that will be necessary in
Section~\ref{Sec:PassmanForms},
where we provide the bulk of results needed to establish Theorem~\ref{thm:mainNew}. 
Our approach is very much influenced by Passman \cite{passman1984infinite}. 
In particular, we utilize a version of the 
$\Delta$-method. 
In Section~\ref{Sec:HardRight},
we prove the implication (a)$\Rightarrow$(e) of Theorem~\ref{thm:mainNew} for nearly epsilon-strongly graded rings.
In Section~\ref{Sec:Proofs}, the proof of Theorem~\ref{thm:mainNew} is finalized. We also show that Theorem~\ref{maintheorem} can be recovered from Theorem~\ref{thm:mainNew}.
In Section~\ref{Sec:Sufficient}, we use Theorem~\ref{thm:mainNew} 
to obtain the following generalization of a result by Passman (see~\cite[Cor.~4.6]{passman1984infinite}):

\begin{thm}\label{thm:NearlyTorsion}
Suppose that $G$ is torsion-free and that $S$ is nearly epsilon-strongly $G$-graded.
Then $S$ is prime if and only if $S_e$ is $G$-prime.
\end{thm}

The remaining sections are devoted to 
applications of our findings.
In Section~\ref{Sec:Strongly}, we obtain
an $s$-unital analogue of Passman's Theorem~\ref{maintheorem} (see Corollary~\ref{cor:PassmanSunital})
and consider $\Z$-graded 
Morita context algebras
and
$\Z$-graded infinite matrix rings.
In Section \ref{Sec:GroupRing}, we apply Theorem~\ref{thm:mainNew} to group rings. Notably, we obtain the following non-unital generalization of Connell's \cite{connell1963group} classical characterization:

\begin{thm}\label{thm:Connel}
Suppose that $R$ is an $s$-unital ring and that $G$ is a group. Then the group ring $R[G]$ is prime if and only if $R$ is prime and $G$ has no non-trivial finite normal subgroup.
\end{thm}

In Section~\ref{Sec:Partial}, we apply our results to $s$-unital partial skew group rings (see Theorem~\ref{thm:primepartialskewgroupring} and Theorem~\ref{thm:partialSGR}) and to unital partial crossed products (see Theorem~\ref{thm:primepartialcrossedproduct} and Theorem~\ref{thm:partial}). 
In Section~\ref{Sec:LPA}, we use Theorem~\ref{thm:mainNew} to
obtain a characterization of prime 
Leavitt path algebras, thereby generalizing
Theorem~\ref{larki}
by allowing the coefficient ring $R$ to be non-commutative:

\begin{thm}\label{thm:LPA}
Suppose that $E$ is a directed graph and that $R$ is a unital ring.
Then the Leavitt path algebra $L_R(E)$ is prime if and only if 
$R$ is prime and $E$ satisfies condition (MT-3).
\end{thm}

\section{Preliminaries}
\label{Sec:Prel}

In this section, we recall some useful notions and conventions on
groups, rings and modules.
We also provide some
preliminary results
on different types of graded rings
such as epsilon-strongly graded rings,
nearly epsilon-strongly graded rings and
cancellatively graded rings.
These results will be utilized in subsequent sections.

\subsection{Groups}\label{sec:groups}
For the entirety of this article, $G$ denotes a 
multiplicatively written group with neutral element $e$.
Let $H$ be a subgroup of $G$. The
\emph{index} of $H$ in $G$ is denoted by $[G:H]$. Take $g \in G$.
The \emph{order} of $g$ is denoted by $\mbox{ord}(g)$.
The {\it centralizer} of $g$ in $G$
is defined to be the subgroup 
$C_G(g):=\{ x \in G \mid xg = gx \}$ of $G$. 
Recall that the \emph{finite conjugate center} of $G$ is the subgroup
$\Delta(G) := \{ g \in G \mid [ G : C_G(g) ] < \infty \}$ of $G$. 
The \emph{almost centralizer of $H$ in $G$}
is the subgroup
$D_G(H) := \{ x \in G \mid [ H : C_H(x) ] < \infty \}$ of $G$.
Note that $D_G(H) \cap H = \Delta(H).$
By the orbit-stabilizer theorem, $\Delta(G)$ can 
equivalently be described as the set of elements of $G$ 
with only finitely many conjugates in $G$. 
If $G$ is equipped with a total order relation $\leq$ 
such that for all $a,b,x,y \in G$ 
the inequality $a \leq b$ implies the inequality
$xay \leq xby$,
then $G$ is called an {\it ordered group}.

\subsection{Rings and modules}
Throughout this article, 
all rings are assumed to be associative but not necessarily unital. Let $R$ be a ring.
If $U$ and $V$ are subsets of $R$, then 
$UV$ denotes the set of finite sums of elements of 
the form $uv$ where $u \in U$ and $v \in V$.
We say that $R$ is \emph{unital} if it has a nonzero multiplicative identity element. 
In this article, we will also consider 
the following weaker notion
of unitality.  
The ring $R$ is called \emph{$s$-unital} if 
for every $r \in R$ the inclusion
$r \in rR \cap Rr$ holds.
For future reference,
we recall the following:

\begin{prop}[Tominaga {\cite[Prop.~12]{nystedt2018unital}, \cite{tominaga1976s}}]\label{prop:tominaga}
 A ring $R$ is $s$-unital if and only if for any finite
 subset $V$ of $R$ there is $u \in R$ such that
 for every $v \in V$ the equalities
 $uv = vu = v$ hold.
\end{prop}
If $M$ is a left $R$-module and $U$ is a subset
of $M$, then  
the \emph{left annihilator of $U$} is defined 
to be the set 
$l.\Ann_R(U) := \{ r \in R \mid r \cdot u = 0, \, \, \forall u \in U \}.$ 
If $N$ is a right $R$-module and $V$ is a subset
of $N$, then
the right annihilator
$r.\Ann_R(V)$
is defined analogously.

\subsection{Group graded rings}\label{sectioncharacterization}
For the rest of this article $S$ denotes a nonzero $G$-graded ring.
Note that the \emph{principal component} $S_e$ is a subring of $S$ and every $x \in G$, the set 
$S_x S_{x^{-1}}$ is an ideal of $S_e$.
The \emph{support of $S$}, denoted by
$\Supp(S)$,
is the set of $x \in G$ with 
$S_x \neq \{ 0 \}$.
In general, $\Supp(S)$ need not be a subgroup of $G$ (see \cite[Rmk.~46]{nystedt2016epsilon}). 
Take $s \in S$. Then $s = \sum_{x \in G} s_x$,
for unique $s_x \in S_x$, such that $s_x = 0$ for 
all but finitely many $x \in G$.
The \emph{support of $s$}, denoted by $\Supp(s)$, is the set of $x \in G$ with $s_x \neq 0$.

\begin{prop}\label{prop:stefan}
The ring $S$ is strongly $G$-graded if and only if for every $x \in G$ the equalities $S_x S_e = S_e S_x = S_x$ and $S_x S_{x^{-1}} = S_e$ hold.
\end{prop}

\begin{proof}
Suppose that for every $x \in G$ the equalities $S_x S_e = S_e S_x = S_x$ and $S_x S_{x^{-1}} = S_e$ hold. Take $x,y \in G$. Then 
$S_{xy} = S_{xy} S_e = S_{xy} S_{y^{-1}} S_y \subseteq S_{xy y^{-1}} S_y = S_x S_y \subseteq S_{xy}$.
Thus $S_x S_y = S_{xy}$.
The converse statement is trivial.
\end{proof}

\begin{rem}\label{rem:strongly_are_fully_supported}
Suppose that $S$ is unital strongly $G$-graded.
Then, for every $x \in G$, the relations
$0 \ne 1_S \in S_e = S_x S_{x^{-1}}$ hold
(see e.\,g. \cite[Prop.~1.1.1]{nastasescu2004methods}).
Therefore, $\Supp(S) = G$.
\end{rem}

The following notion 
was first introduced by Clark, Exel and Pardo in the context of Steinberg algebras
{\cite[Def.~4.5]{clark2018generalized}}:

\begin{defi}
The ring $S$ is said to be \emph{symmetrically $G$-graded} 
if for every $x \in G$, the equality $S_x S_{x^{-1}} S_x = S_x$ holds.
\end{defi}

\begin{rem}\label{rem:symmetrical}
If $S$ is symmetrically $G$-graded, then
 $\Supp(S)^{-1} = \Supp(S)$.
\end{rem}

Note that strongly $G$-graded rings are symmetrically $G$-graded.
As the following example shows, a grading which is not strong may fail to be symmetrical:

\begin{exa}
Let $R$ be a unital ring and consider the standard $\mathbb{Z}$-grading on the polynomial ring $R[x] = \bigoplus_{i \in \mathbb{Z}} S_i$ where $S_i := R x^i$ for $i \geq 0$, and $S_i := \{ 0 \}$ for $i < 0$. 
Clearly, $\Supp(S)^{-1} \neq \Supp(S)$
and thus, by Remark~\ref{rem:symmetrical},
it follows that the grading is not symmetrical.
\end{exa}

Next, we will consider another special type of grading. Passman appears to have been the first to give the following definition (see also \cite{cohen1983group,oinert2012ideal}):

\begin{defi}[{\cite[p.~32]{passman1989infinite}}]
The ring $S$ is said to be \emph{non-degenerately $G$-graded} if for every $x \in G$ and every 
nonzero $s \in S_x$,
we have $s S_{x^{-1}} \ne \{0\}$ and $S_{x^{-1}} s \ne \{0\}$.
\end{defi}

Clearly, every unital strongly $G$-graded ring is  
non-degenerately $G$-graded.

\subsection{Epsilon-strongly graded rings}

Now, we consider a generalization of unital strongly graded rings,
introduced by Nystedt, \"{O}inert and Pinedo
{\cite[Def.~4, Prop.~7]{nystedt2016epsilon}}.

\begin{defi}
The ring $S$ is called {\it epsilon-strongly $G$-graded} if for every $x \in G$ there exists $\epsilon_x \in S_x S_{x^{-1}}$ such that for all $s \in S_x$
the equalities $\epsilon_x s = s = s \epsilon_{x^{-1}}$
hold.
\end{defi}

\begin{exa}\label{ex:matrix1}
Let $R$ be a unital ring and consider the following $\mathbb{Z}$-grading on the ring $M_2(R)$ of 
$2\times2$-matrices with entries in $R$:
\begin{equation*}
(M_2(R))_0 := \begin{pmatrix}
R & 0 \\
0 & R 
\end{pmatrix}, \quad (M_2(R))_{1} := \begin{pmatrix}
0 & 0 \\
R & 0
\end{pmatrix}, \quad
(M_2(R))_{-1} := \begin{pmatrix}
0 & R \\
0 & 0
\end{pmatrix},
\end{equation*}
and $(M_2(R))_i$ zero 
if $|i| > 1$. Clearly, this grading is not strong, but epsilon-strong with 
\begin{equation*}
\epsilon_1 = \begin{pmatrix}
1_R & 0 \\
0 & 0 
\end{pmatrix} \quad \mbox{and} \quad 
\epsilon_{-1} = \begin{pmatrix}
0 & 0 \\
0 & 1_R
\end{pmatrix}.
\end{equation*}
Suppose that $R$ is prime.
Then $M_2(R)$ is also prime, but
$(M_2(R))_0$
is not prime. This is an example of a prime epsilon-strongly $\Z$-graded ring whose principal component is not prime.
\end{exa}

Moreover, unital partial crossed products (see \cite{nystedt2016epsilon}), Leavitt path algebras of finite graphs (see \cite{nystedt2017epsilon}), and certain Cuntz-Pimsner rings (see \cite{lannstrom2019graded}) are classes of graded rings that are epsilon-strongly graded.
A further generalization was introduced by Nystedt and \"{O}inert:

\begin{defi}[{\cite[Def.~10]{nystedt2017epsilon}}]
The ring $S$ is called {\it nearly epsilon-strongly $G$-graded}
if for every $x \in G$ and every $s \in S_x$ there exist $\epsilon_x(s) \in S_x S_{x^{-1}}$ and $\epsilon_x(s)' \in S_{x^{-1}} S_x$ such that the equalities $\epsilon_x(s) s = s = s \epsilon_x(s)'$ hold.
\end{defi}

Notably, every Leavitt path algebra with its natural $\mathbb{Z}$-grading is nearly epsilon-strongly $\mathbb{Z}$-graded whereas only Leavitt path algebras of finite graphs are epsilon-strongly $\Z$-graded (cf. \cite[Thm.~28, Thm.~30]{nystedt2017epsilon}).

\begin{prop}[{\cite[Prop.~11]{nystedt2017epsilon}}]\label{prop:6}
The ring $S$ is nearly epsilon-strongly $G$-graded 
if and only if 
$S$ is symmetrically $G$-graded and  
for every $x \in G$ the ring
$S_x S_{x^{-1}}$ is  
$s$-unital.
\end{prop}

\begin{rem}\label{rem:1}
The following implications hold for all
$G$-graded rings:
\begin{align*}
\text{unital strong} &\implies \text{epsilon-strong}  \implies \text{nearly epsilon-strong} \implies \text{symmetrical} 
\end{align*}
\end{rem}

\begin{prop}\label{prop:s-unital}
Suppose that $S$ is nearly epsilon-strongly $G$-graded.
Then, $s \in s S_e \cap S_e s$ for every $s \in S$. 
In particular, $S$ is $s$-unital and $S_e$ is an $s$-unital subring of $S$.
\end{prop}
\begin{proof}
Proposition~\ref{prop:6} yields, in particular, that (i) $S_e = S_e S_e S_e$ and (ii) $S_e S_e = S_e^2$ is an $s$-unital ring. But (i) gives that $S_e = S_e^3 \subseteq S_e^2 \subseteq S_e$. Thus, $S_e = S_e^2$ is $s$-unital.

Let $s = \sum_{y \in G} s_y \in S$ with $s_y \in S_y$.
Fix $x \in \Supp(s)$. 
By Proposition~\ref{prop:6}, there are finitely many elements $a_i \in S_x S_{x^{-1}} \subseteq S_e$, $b_j \in S_{x^{-1}} S_x \subseteq S_e$ and $s_i, s_j' \in S_x$ such that $s_x = \sum_i a_i s_i = \sum_j s_j' b_j$. 
Now, since $S_e$ is $s$-unital, there is some $e_x \in S_e$ such that $e_x a_i = a_i$ and $b_j e_x = b_j$ for all $i,j$ (see Proposition~\ref{prop:tominaga}). 
Then, $e_x s_x = s_x = s_x e_x$.
Hence, we can find such an $e_x \in S_e$ for every $x \in \Supp(s)$. Since $\Supp(s)$ is a finite set, it follows from Proposition~\ref{prop:tominaga} that there is some $e_s \in S_e$ such that $e_s e_x = e_x = e_x e_s$ for every $x \in \Supp(s)$. 
Then $e_s s = \sum e_s s_x = \sum e_s (e_x s_x) = \sum (e_s e_x) s_x = \sum e_x s_x = \sum s_x = s$,
where the sum runs over $\Supp(s)$.
Similarly, $s e_s = s$.
\end{proof}

Not every symmetrically $G$-graded ring is nearly epsilon-strongly $G$-graded:

\begin{exa}
Let $R$ be an idempotent ring that is not $s$-unital (see e.\,g. \cite[Expl.~2.5]{nystedt2018unital}). 
Consider the $G$-graded ring $S$ defined by $S_e := R$ and 
$S_x := \{ 0 \}$ if $x \in G \setminus \{ e \}$. 
Clearly, $S$ is symmetrically $G$-graded, but
by Proposition~\ref{prop:s-unital} $S$ is not nearly epsilon-strongly $G$-graded.
\end{exa}

\begin{prop}[{\cite[Prop.~3.4]{nystedt2017epsilon}}]\label{prop:nearly_implies_nondegenerate}
If $S$ is nearly epsilon-strongly $G$-graded,
then $S$ is non-degenerately $G$-graded.
\end{prop}

\begin{lem}\label{lem:strongSunital}
If $S$ is $s$-unital strongly $G$-graded, then
the following assertions hold:
\begin{enumerate}[{\rm (a)}]
	\item $S$ is nearly epsilon-strongly $G$-graded.
	\item $s \in s S_e \cap S_e s$ for every $s \in S$.
	\item $r.\Ann_S(S_x)=\{0\}$ for every $x\in G$.
\end{enumerate}
\end{lem}

\begin{proof}
(a): Clearly, $S$ is symmetrically $G$-graded.
By \cite[Lem.~6.8]{lannstrom2019structure}, $S_x S_{x^{-1}} = S_e$ is $s$-unital for every $x \in G$.
The desired conclusion now follows from Proposition~\ref{prop:6}.

(b): This follows from (a) and Proposition~\ref{prop:s-unital}.

(c): Take $x \in G$ and $s\in r.\Ann_S(S_x)$.
Then
$\{0\} = S_{x^{-1}} S_x s = S_e s$.
Thus, $s=0$ by (b).
\end{proof}

\begin{rem}
If $S$ is strongly $G$-graded, then $S$ is $s$-unital if and only if $S_e$ is $s$-unital.
\end{rem}

In the rest of this article, we will freely use
the fact that nearly epsilon-strongly graded rings are symmetrically graded, non-degenerately graded, and $s$-unital without further comment. For additional characterizations of (nearly) epsilon-strongly graded rings, we refer to \cite{MarPinSol20}.

\subsection{Induced gradings}\label{sec:indgrad}
Now, we recall two important functorial constructions. For more details, we refer the reader to \cite{lannstrom2018induced}. The first construction assigns a subring of $S$ with an inherited grading. 
Let $H$ be a subgroup of $G$ and put $S_H := \bigoplus_{x \in H} S_x$. Note that $S_H$ is an $H$-graded ring that is also a subring of $S$. Consider the map $\pi_H \colon S \to S_H$ defined by
$$ \pi_H \left(\sum_{x \in G} s_x\right) = \sum_{x \in H} s_x.$$

The following result is well-known (see  e.\,g.~\cite[Lem.~2.4]{oinert2019units}):

\begin{lem}\label{lem:bimodule_homo}
The map $\pi_H \colon S \to S_H$ is an $S_H$-bimodule homomorphism.
\end{lem}

We can ``map down'' nonzero ideals when the ring is non-degenerately $G$-graded:

\begin{lem}\label{lem:pi_non-degenerate}
Suppose that $H$ is a subgroup of $G$.
If $A$ is a 
left (resp. right)
ideal of $S$, then $\pi_H(A)$ is a 
left (resp. right)
ideal of $S_H$. 
If, in addition, $S$ is non-degenerately $G$-graded and $A$ is nonzero, then $\pi_H(A)$ is nonzero.
\end{lem}
\begin{proof}
The first statement immediately follows from Lemma~\ref{lem:bimodule_homo}.
For the second statement suppose that $S$ is non-degenerately $G$-graded and that $A$ is a nonzero left $S$-ideal. 
Pick a nonzero $a \in A$ and $x \in \Supp(a)$.
Then, since $S$ is non-degenerately $G$-graded, 
$\{0\} \ne S_{x^{-1}} a_x= \pi_{\{ e \}}(S_{x^{-1}} a) \subseteq \pi_{\{ e \}}(A) \subseteq \pi_H(A).$
The case when $A$ is a right ideal is proved similarly.
\end{proof}

We now describe the second construction: Given a 
normal subgroup $N$ of $G$, we define 
\emph{the induced $G/N$-grading}
on $S$ in the following way. 
For every $C \in G/N$, put $ S_C := \bigoplus_{x \in C} S_x.$ This yields a $G/N$-grading on $S$. The following non-trivial
result, proved by L\"{a}nnstr\"o{m}, will be essential later on in this article:

\begin{prop}[{\cite[Prop.~5.8]{lannstrom2018induced}}]\label{prop:lannstrom1}
Suppose that $S$ is nearly epsilon-strongly $G$-graded and that $N$ is a normal subgroup of $G$. Then
the induced $G/N$-grading on $S$ is nearly epsilon-strong.
\end{prop}

We will also need the following result:

\begin{prop}\label{prop:QuotientNonDeg}
Suppose that $S$ is non-degenerately $G$-graded
and that $N$ is a normal subgroup of $G$.
Then the induced $G/N$-grading on $S$ is non-degenerate.
\end{prop}

\begin{proof}
Take $x\in G$ and a nonzero $a \in S_{xN}$.
Write $a = a_{x n_1} + a_{x n_2} + \ldots + a_{x n_k}$
where $n_1,\ldots,n_k \in N$ are all distinct and $a_{x n_i} \neq 0$ for every $i$.
By non-degeneracy of the $G$-grading there is some
$c_{n_1^{-1} x^{-1}} \in S_{Nx^{-1}} = S_{x^{-1}N}$ 
such that
$c_{n_1^{-1} x^{-1}} a_{x n_1} \neq 0$.
Note that
\begin{align*}
\pi_{\{e\} }(c_{n_1^{-1}x^{-1}} a)
= c_{n_1^{-1}x^{-1}} a_{x n_1} \neq 0.
\end{align*}
Hence,
$c_{n_1^{-1}x^{-1}} a \neq 0$.
This shows that $S_{x^{-1}N} a \neq \{0\}$.
Similarly, 
$a S_{x^{-1}N} \neq \{0\}$.
Thus, the induced $G/N$-grading on $S$ is non-degenerate.
\end{proof}

\subsection{Cancellatively graded rings}\label{sec:cancellative}
We now briefly discuss Passman's notion of cancellatively graded rings. 
We will, however, not work with this class of rings outside of this section.

In \cite{PassmanCancellative} Passman extended his results from \cite{passman1984infinite}
to the class of \emph{cancellatively group graded rings}
which generalizes the class of unital strongly group graded rings.
To avoid any confusion, we wish to point out that 
Passman's \cite{passman1984infinite} 
notion of \emph{$H$-stability} is used interchangeably with our notion of $H$-invariance.
Recall from the introduction that a unital $G$-graded ring $S$ is called \emph{cancellative}
if for all $x,y \in G$ and all homogeneous
subsets $U,V \subseteq S$, the implication
$U S_x S_y V = \{0\}
\Rightarrow
U S_{xy} V = \{0\}$
holds. 
Clearly, all strongly graded rings are cancellative.
However, e.\,g. canonical $\Z$-gradings on
Leavitt path algebras (see Section~ \ref{Sec:LPA}) need not be  cancellative.

\begin{rem}\label{rem:central}
    L\"{a}nnstr\"{o}m has observed that
    if $S$ is epsilon-strongly $G$-graded, then $S$ must be unital (see \cite[Prop.~3.8]{lannstrom2018induced}). Moreover, $\epsilon_x$ is central in $S_e$ for every $x \in G$ (see \cite{nystedt2016epsilon}).
\end{rem}

\begin{lem}\label{lem:AnnSym}
The following assertions hold for each $x\in G$:
\begin{enumerate}[{\rm (a)}]

\item If $S$ is symmetrically $G$-graded, 
then $r.\Ann_S(S_x) = r.\Ann_S(S_{x^{-1}} S_x)$.

\item If $S$ is epsilon-strongly $G$-graded, then 
$r.\Ann_S(S_x) = r.\Ann_S(S_{x^{-1}} S_x) = r.\Ann_S(\epsilon_{x^{-1}} )$.

\end{enumerate}
\end{lem}

\begin{proof}
(a): 
Suppose that $S$ is symmetrically $G$-graded.
If $s \in r.\Ann_S(S_x)$,
then $S_{x^{-1}} S_x s = \{0\}$, which implies that $r.\Ann_S(S_x) \subseteq r.\Ann_S(S_{x^{-1}} S_x)$.
If, conversely, $s \in{} r.\Ann_S(S_{x^{-1}} S_x)$, then
$S_x s = S_x S_{x^{-1}} S_x s = \{0\}$.
Thus,
$r.\Ann_S(S_{x^{-1}} S_x) \subseteq r.\Ann_S(S_x)$.

(b): Suppose that $S$ is epsilon-strongly $G$-graded.
Then $S_{x^{-1}}S_x  = \epsilon_{x^{-1}} S_e = S_e \epsilon_{x^{-1}}$,
which entails that
$r.\Ann_S( S_{x^{-1}} S_x) = r.\Ann_S(S_e \epsilon_{x^{-1}}) = r.\Ann_S(\epsilon_{x^{-1}})$, where the last equality follows from the fact that $1_S = 1_{S_{e}}$.
\end{proof}

\begin{prop}\label{prop:EpsCancellativeStrong}
Suppose that $S$ is 
epsilon-strongly $G$-graded. Then 
the following assertions are equivalent:
\begin{enumerate}[{\rm (a)}]
    \item the grading on $S$ is strong;
    \item for every $x \in G$, the equality
    $r.\Ann_S(S_x) = \{0 \}$ holds;
    \item the grading on $S$ is cancellative.
\end{enumerate}
\end{prop}

\begin{proof}
(a)$\Rightarrow$(b):
Take $x \in G$. For $s\in S$, we note that
\begin{displaymath}
    S_x s = \{0\}
    \ \Longrightarrow \ S_{x^{-1}} S_x  s = \{0\}
    \ \Longrightarrow \ S_e s = \{0\}
    \ \Longrightarrow \ 1_S \cdot s= 0.
\end{displaymath}
Hence, $r.\Ann_S(S_x) = \{0 \}$.

(b)$\Rightarrow$(c):
 By \cite[Lem.~1.2]{PassmanCancellative}, $S$ is cancellative if and only if, for every $x \in G$, (i) $S_x S_{x^{-1}}$ is a so-called \emph{middle cancellable ideal} of $S_e$ and (ii) $r.\Ann_S(S_x) = \{ 0 \}$. In the special case of epsilon-strongly graded rings, (ii) actually implies (i). Let $x \in G$ and recall that $S_x S_{x^{-1}}$ being middle cancellable means that $U S_x S_{x^{-1}} V = \{0\}$ implies that $UV = \{0\}$ for all subsets $U, V \subseteq S_e$. Moreover, note that $S_x S_{x^{-1}} = \epsilon_x S_e$ for some central element $\epsilon_x \in S_e$ and $$U S_x S_{x^{-1}} V = \{0\} \iff U \epsilon_x S_e V = \{0\} \iff \epsilon_x U S_e V = \{0\} \implies \epsilon_x UV = \{0\}.$$
Now, note that, using Lemma~\ref{lem:AnnSym}, we get
$$ \{0\} = r.\Ann_S(S_{x^{-1}}) = r.\Ann_S(S_x S_{x^{-1}}) = r.\Ann_S(\epsilon_x S_e) \supseteq r.\Ann_S(\epsilon_{x}).$$
Hence, $UV=\{0\}$ whenever
$U S_x S_{x^{-1}} V = \{0\}$. Thus, $S_x S_{x^{-1}}$ is middle cancellable for every $x \in G$. In other words, (ii) implies (i).

(c)$\Rightarrow$(a):
Suppose that the grading on $S$ is not strong.
There is some $x\in G$ such that $\epsilon_x \neq 1_S$.
Put $U=V:=\{ 1_S-\epsilon_x \}$ and note that
$1-\epsilon_x$ is an idempotent.
Clearly, $UV=\{1_S-\epsilon_x\} \neq \{0\}$, since $\epsilon_x \neq 1_S$.
However, we also have that
$U S_x S_{x^{-1}} V = U \epsilon_x S_e V = \{0\}$
which shows that $S_x S_{x^{-1}}$ is not a middle cancellable ideal of $S_e$.
By \cite[Lem.~1.2]{PassmanCancellative} (see also the above proof of (b)$\Rightarrow$(c)), the grading is not  cancellative.
\end{proof}

\begin{prop}\label{prop:cancellative1}
If $S$ is unital and 
cancellatively $G$-graded, then $\Supp(S) = G$. 
\end{prop}
\begin{proof}
Take $x \in G$. Since $S$ is unital, we get that
$S_e S_{x x^{-1}} S_e = S_e S_e S_e = S_e \neq \{ 0 \}$.
Thus, by cancellativity, we get
$S_e S_x S_{x^{-1}} S_e \neq \{ 0 \}$.
Hence, $S_x \neq \{ 0 \}$.
\end{proof}

Recall that a unital strongly $G$-graded ring $S$ also satisfies $\Supp(S) = G$ (see Remark~\ref{rem:strongly_are_fully_supported}). However, $\Supp(S)=G$ need not hold, in general, for nearly epsilon-strongly graded rings. 

\begin{rem}
Proposition~\ref{prop:EpsCancellativeStrong} demonstrates that
epsilon-strongly graded rings which can be reached by Passman's ``cancellative results'' \cite{PassmanCancellative} are, in fact, unital strongly graded. 
Thus, that case has already been treated by Passman in \cite{passman1984infinite}.
\end{rem}

\section{Invariant ideals}
\label{Sec:Invariant}

Recall that $S$ is a $G$-graded ring.
If $S$ is strongly $G$-graded, then there is an action of $G$ on the lattice of ideals of $S_N$ for any normal subgroup $N$ of $G$ (see \cite[Sec.~5.2]{passman1984infinite}). The purpose of this section is to investigate this construction 
for more general classes of $G$-graded rings.

\begin{defi}
If $I$ is a subset of $S$ and $x \in G$, then
we define $I^x := S_{x^{-1}} I S_x$.
\end{defi}

\begin{lem}
\label{lem:IxIdeal}
If $x \in G$ and
$I$ is an ideal of $S_e$, 
then $I^x$ is an ideal of $S_e$.
\end{lem}
\begin{proof}
Clearly, $I^x$ is an additive subgroup of $S_e$. Since $S_{x^{-1}}$ and $S_x$ are $S_e$-bimodules, it follows that $S_e I^x = S_e S_{x^{-1}} I S_x \subseteq S_{x^{-1}} I S_x = I^x.$ Similarly, $I^x S_e \subseteq I^x$. 
\end{proof}

Recall that if $H,K$ are subsets of $G$, then $K$ is said to be \emph{normalized by $H$} if $Kx=xK$ for every $x\in H$. 

\begin{defi}[{cf.~\cite[p.~406]{PassmanCancellative}}]\label{def:weakly_invariant}
Suppose that $H$ is a subgroup of $G$ and that $I$ is a 
subset of $S$.
Then $I$ is called 
\emph{$H$-invariant} if $I^x \subseteq I$ for every $x \in H$.
Furthermore, if $K$ is a subset of $G$ which is normalized by $H$, then we say that $I$ is \emph{$H/K$-invariant}
if $S_{x^{-1}K} I S_{xK} \subseteq I$ for every $x\in H$.
\end{defi}

In the special case of $s$-unital (and in particular unital) strongly $G$-graded rings, our 
definition coincides with Passman's notion of invariance used in \cite{passman1984infinite}:

\begin{lem}\label{lem:WeakStrongInvariance}
Suppose that $H$ is a subgroup of $G$ and that $S$ is $s$-unital strongly $G$-graded.
Then a subset $I$ of $S$ is $H$-invariant if and only if
$I^x = I$
for every $x \in H$.
\end{lem}

\begin{proof}
Suppose that $I$ is $H$-invariant.
Take $x \in H$.
By Lemma~\ref{lem:strongSunital}(b) we have
$$I \subseteq S_e I S_e = (S_{x^{-1}} S_x) I (S_{x^{-1}} S_x) = S_{x^{-1}} (S_x I S_{x^{-1}}) S_x \subseteq  S_{x^{-1}} I S_x = I^{x} \subseteq I.$$
This shows that $I^x = I$.
The converse statement is trivial.
\end{proof}

\begin{exa}\label{ex:matrix2}
	Let us again look at Example~\ref{ex:matrix1}. Let $J,J'$ be nonzero $R$-ideals and consider the following ideals of $(M_2(R))_0$:
	\begin{align*}
		I = 
		\begin{pmatrix}
			J & 0 \\
			0 & J
		\end{pmatrix}
		\qquad
		\text{and}
		\qquad
		I' = 
		\begin{pmatrix}
			J & 0 \\
			0 & J'
		\end{pmatrix}.
	\end{align*}
	It is easily checked that $I$ is $\Z$-invariant but $I^x = I$ does not hold for every $x\in \Z$.
	Moreover, if $J \not \subseteq J'$, then a quick verification shows that $I'$ is not $\Z$-invariant. However, $I'$ is invariant with respect to any proper non-trivial subgroup of $\Z$.
\end{exa}
More examples of invariant ideals may be found in Example~\ref{exp:par.free,grp}
and
Example~\ref{ex:lpa_invariant}. 
The following result is essential and will often be used implicitly in the rest of this article:

\begin{lem}[{cf.~\cite[Lem.~5.7]{passman1984infinite}}]\label{lem:2}
Suppose that $I$ and $J$ are subsets of $S$. Then the following assertions hold for all $x,y \in G$:
\begin{enumerate}[{\rm (a)}]
\begin{item}
$(I^x)^y \subseteq I^{xy}$
\end{item}
\begin{item}
$I^x J^x \subseteq (IJ)^x$ if $I$ or $J$
is an ideal of $S_e$.
\end{item}
\begin{item}
If $I \subseteq J$, then $I^x \subseteq J^x$.
\end{item}
\end{enumerate}
\end{lem}
\begin{proof}

(a): 
$(I^x)^y = S_{y^{-1}} (S_{x^{-1}} I S_x) S_y = (S_{y^{-1}}S_{x^{-1}}) I (S_x S_y) \subseteq S_{(xy)^{-1}} I S_{xy} = I^{xy}.$

(b): 
$I^x J^x = S_{x^{-1}} I (S_x S_{x^{-1}}) J S_x \subseteq S_{x^{-1}} I S_e J S_x \subseteq S_{x^{-1}} IJ S_x = (IJ)^x$. 

(c): 
$I^x = S_{x^{-1}} I S_x \subseteq S_{x^{-1}} J S_x = J^x$. 
\end{proof}

For unital strongly $G$-graded rings, the inclusions in (a) and (b) of Lemma~\ref{lem:2} are, in fact, equalities (see \cite[Lem.~5.7]{passman1984infinite}). However, Example~\ref{ex:2} below shows that the inclusion in Lemma~\ref{lem:2}(a) can be strict
for some 
nearly epsilon-strongly graded rings. We now prove that the inclusion in Lemma~\ref{lem:2}(b) is actually an equality for nearly epsilon-strongly graded rings. 

\begin{defi}
If $I$ is a subset of $S$, then we say that
$I$ is \emph{$\epsilon$-invariant} if for every
$x \in G$, the equality 
$S_x S_{x^{-1}} I = I S_x S_{x^{-1}}$ holds.
\end{defi}

\begin{rem}
If $S$ is epsilon-strongly $G$-graded, 
$H$ is a subgroup of $G$ and $I$
is an ideal of $S_H$,
then the 
statement
\begin{equation}\label{firstcondition}
S_x S_{x^{-1}} I = I S_x S_{x^{-1}}, \qquad \forall x \in G  
\end{equation}
is equivalent to the statement
\begin{equation}\label{secondcondition}
\epsilon_x I = I \epsilon_x, \qquad \forall x \in G.
\end{equation}
Note that if
$H=\{e\}$,
then \eqref{secondcondition} (and hence also \eqref{firstcondition}) is true 
since the elements $\epsilon_x$, for $x \in G$, 
are central idempotents in $S_e$ (see Remark~\ref{rem:central}). This 
justifies our usage of the term ``$\epsilon$-invariant''.
\end{rem}

\begin{lem}\label{lem:ideal_central}
If $S$ is nearly epsilon-strongly $G$-graded,
then every ideal of $S_e$ is
$\epsilon$-invariant.
\end{lem}
\begin{proof}
Take $x \in G$ and let $I$ be an ideal of $S_e$. 
We prove that $S_x S_{x^{-1}} I \subseteq I S_x S_{x^{-1}}$. The reversed inclusion can be shown
in an analogous fashion and is therefore 
left to the reader. 
Take $s_x \in S_x$, $s_{x^{-1}} \in S_{x^{-1}}$ 
and $a \in I$. 
Since $s_{x^{-1}} a \in S_{x^{-1}}$,  
there is $\epsilon_{x^{-1}}' (s_{x^{-1}} a) \in S_x S_{x^{-1}}$ such that $s_{x^{-1}} a =  s_{x^{-1}} a \cdot \epsilon_{x^{-1}}'(s_{x^{-1}} a)$. 
Using that $s_x s_{x^{-1}} a \subseteq I$, it follows that $
    s_x s_{x^{-1}} a = (s_x s_{x^{-1}} a) \cdot \epsilon_{x^{-1}}'(s_{x^{-1}} a) \in I S_x S_{x^{-1}}. $
\end{proof}

\begin{prop}
\label{prop:epsilon-invariant}
Suppose that $S$ is symmetrically $G$-graded, 
$N$ is a normal subgroup of $G$,
and that $I,J$ are ideals of $S_N$.
If 
$I$ or $J$ is $\epsilon$-invariant, 
then $(IJ)^x = I^x J^x$ for every $x \in G$.
\end{prop}
\begin{proof}
Suppose that $I$ is $\epsilon$-invariant. Then, 
since $S$ is symmetrically $G$-graded, we get
$$ I^x J^x = S_{x^{-1}} I S_x S_{x^{-1}} J S_x =
S_{x^{-1}} S_x S_{x^{-1}} IJ S_x =
S_{x^{-1}} IJ S_x = (IJ)^x$$ 
for every $x \in G$.
The case when $J$ is $\epsilon$-invariant can be treated  analogously.
\end{proof}

Combining Lemma~\ref{lem:ideal_central} and Proposition~\ref{prop:epsilon-invariant} we obtain the following result:

\begin{cor}\label{cor:multi1}
If $S$ is nearly epsilon-strongly $G$-graded and
$I,J$ are ideals of $S_e$,
then $(IJ)^x = I^x J^x$ for every $x \in G$.
\end{cor}

\begin{prop}
\label{prop:switch1}
If $S$ is nearly epsilon-strongly $G$-graded and
$I$ is an ideal of $S_e$,
then the following assertions hold:
    \begin{enumerate}[{\rm (a)}]
        \item 
            $ I S_y = S_y I^y$ and $S_{y^{-1}} I = I^y S_{y^{-1}}$ for every $y \in G$.
        \item
            If $H$ is a subgroup of $G$
            and
            $I$ is  
    $H$-invariant, 
    then $I S_y = S_y I$ for every $y \in H$.
    \end{enumerate}
\end{prop}
\begin{proof}
(a):
Take $y \in G$. 
Since $S$ is symmetrically $G$-graded and $I$ is $\epsilon$-invariant by Lemma~\ref{lem:ideal_central}, we have
$I S_y = I (S_y S_{y^{-1}} S_y) = I (S_y S_{y^{-1}}) S_y = (S_y S_{y^{-1}}) I S_y = S_y (S_{y^{-1}} I S_y) = S_y I^y.$
Similarly,
$S_{y^{-1}} I = (S_{y^{-1}} S_y  S_{y^{-1}}) I = S_{y^{-1}} (S_y  S_{y^{-1}}) I  = S_{y^{-1}} I (S_y S_{y^{-1}})  =  (S_{y^{-1}} I S_y) S_{y^{-1}} = I^y S_{y^{-1}}.$ 

(b):
Take $y \in H$. By (a), 
we get
$I S_y = S_y I^y \subseteq S_y I = I^{y^{-1}} S_y \subseteq I S_y$.
Thus, $S_y I = I S_y$.
\end{proof}

In the following lemma we use the induced quotient grading described in Section~\ref{sec:indgrad}.

\begin{lem}\label{lem:induced_grading_ideals}
Suppose that
$N$ is a normal subgroup of $G$. 
If
$I$ is a 
$G/N$-invariant subset of $S_N$, then $I$ is 
$G$-invariant.
\end{lem}
\begin{proof}
Suppose that $I$ is 
$G/N$-invariant.
Take $x \in G$. 
Then $S_{x^{-1}} I S_x \subseteq S_{x^{-1}N} I S_{xN} \subseteq I$.
\end{proof}

\begin{rem}\label{rem:switch2}
In Passman's original setting of unital strongly $G$-graded rings an important property that is repeatedly used is that, for $y \in G$, $S_y I = I S_y$ if and only if $I^y = I$ for any ideal $I$ of $S_H$, where $H$ is a subgroup of $G$.  
In our generalized setting, we will have to make 
do with the result in Proposition~\ref{prop:switch1} which only holds for ideals of the principal component.
\end{rem}

The identity $(I^x)^y = I^{xy}$, for all $x,y \in G$, does not hold in general when working with nearly epsilon-strongly $G$-graded rings. 
Before giving an example for which this identity fails, 
note that if $x \not \in \Supp(S)$, then $I^x = \{0\}$ for every ideal $I$ of $S_e$.

\begin{exa}\label{ex:2}
	Let $R$ be an $s$-unital ring and let $G$ be a non-trivial group.
	Consider the
	nearly epsilon-strong $G$-graded ring $S$ defined by $S_e := R$ and $S_x := \{ 0 \}$ for $x \in G \setminus \{e\}$.
Now, consider the nonzero ideal $R$ of $R$ and let $x \in G \setminus\{e\}$. Then $\{0\} \ne R = R^{x x^{-1}} \neq (R^{x})^{x^{-1}} = \{0\}$,
  because $x \not \in \Supp(S)$.
\end{exa}

\begin{lem}\label{lem:misc1}
Suppose that $S$ is nearly epsilon-strongly $G$-graded, $K$ is a subgroup of $G$ and that $I$ and $J$
are ideals of $S_e$.
Then the following assertions hold:
\begin{enumerate}[{\rm (a)}]
\begin{item}
If $I,J$ are  
$K$-invariant, then 
$IJ$ is 
$K$-invariant.
\end{item}
\begin{item}
If $I$ is  
$K$-invariant, then $r.\Ann_{S_e}(I)$ is
$K$-invariant.
\end{item}
\end{enumerate}
\end{lem}
\begin{proof}
(a): 
This follows from Corollary~\ref{cor:multi1}.

(b): Take $x \in G$. 
From Proposition~\ref{prop:switch1}, it follows that
\begin{equation*}
	I \cdot S_{x^{-1}} (r.\Ann_{S_e}(I)) S_x \subseteq S_{x^{-1}} I  (r.\Ann_{S_e}(I)) S_x = S_{x^{-1}} (I \cdot  r.\Ann_{S_e}(I)) S_x  = \{0\}. \qedhere
\end{equation*}
\end{proof}

\begin{lem}\label{lem:sums1}
If $x \in G$ and $F$ is a family of 
subsets of $S$, then
$(\sum_{I \in F} I)^x = \sum_{I \in F} I^x$. 
\end{lem}
\begin{proof}
$(\sum_{I \in F} I)^x = S_{x^{-1}} (\sum_{I \in F} I) S_x = \sum_{I \in F} S_{x^{-1}} I S_x = \sum_{I \in F} I^x$. 
\end{proof}

\begin{defi}
For $H \subseteq G$ and $M \subseteq S$ we define
$M^H := \sum_{h \in H} S_{h^{-1}} M S_h$.
\end{defi}

\begin{lem}\label{lem:power}
With the above notation the following assertions hold:
\begin{enumerate}[{\rm(a)}]
\item If $H$ is a subgroup of $G$
and $M \subseteq S$, then 
$M^H$ is an $H$-invariant subset of $S$.

\item If $S_e$ is $s$-unital and $I$ is an ideal of $S_e$, 
then $I^G$
is the smallest  
$G$-invariant ideal of $S_e$ containing $I$.

\end{enumerate}
\end{lem}

\begin{proof}
(a): Take $x \in H$. Combining Lemma~\ref{lem:2}(a)
and Lemma~\ref{lem:sums1}, we deduce that $(M^H)^x = \left( \sum_{y \in H} M^y \right)^x =
\sum_{y \in H} (M^y)^x \subseteq \sum_{y \in H} M^{yx} = M^H$.

(b): From (a), it follows that $I^G$ is $G$-invariant.
Clearly, $I^G$ is an ideal of $S_e$ and $I = I^e \subseteq I^G$ by $s$-unitality of $S_e$.
Suppose now that $J$ is a  
$G$-invariant $S_e$-ideal such that $I \subseteq J$. Then, by Lemma~\ref{lem:2}(c), $I^x \subseteq J^x \subseteq J$ for every $x \in G$
and hence we get $I^G = \sum_{x \in G} I^x \subseteq J$.
\end{proof}

\begin{lem}\label{lem:12}
The following assertions hold:
\begin{enumerate}[{\rm (a)}]
\begin{item}
Suppose that $S$ is non-degenerately $G$-graded. Let $I$ be a subset of $S_e$ and let $x \in \Supp(S)$ be such that 
$I (S_x S_{x^{-1}}) = I$ or $(S_x S_{x^{-1}})I = I$. If $I \ne \{0\}$, then $I^x \ne \{0\}.$
\end{item}
\begin{item}
Suppose that $S$ is symmetrically $G$-graded. Then for every $S_e$-ideal $I$ and every $x \in G$, we have $I^x (S_{x^{-1}} S_x) = I^x.$
\end{item}
\end{enumerate}
\end{lem}
\begin{proof}
(a): Suppose that $I^x = \{0\}$. 
Since $S$ is non-degenerately $G$-graded, we have $I S_x = \{0\}$ or $S_{x^{-1}} I = \{0\}$. 
Hence, $\{0\} = I S_x S_{x^{-1}} = I$ or $\{0\} = S_x S_{x^{-1}} I = I$.

(b): For every $x \in G$, we have
$I^x (S_{x^{-1}} S_x) = S_{x^{-1}} I  S_x (S_{x^{-1}} S_x) = S_{x^{-1}} I S_x = I^x.$
\end{proof}

Later on, we need to consider ideals $I$ satisfying $I^x I = \{0\}$ for every $x \in G \setminus H$ for some subgroup $H$ of $G$.
The following result will allow us to replace $I$ with  
$I^H$.

\begin{prop}[{cf.~\cite[Lem.~5.5]{passman1984infinite}}]\label{prop:power}
Suppose that $S$ is nearly epsilon-strongly $G$-graded and that $H$ is a subgroup of $G$. Let $I$ be an ideal of $S_e$ such that $I^x I = \{0\}$ for every $x \in G \setminus H$. Then $(I^H)^x (I^H) = \{0\}$ for every $x \in G \setminus H$.
\end{prop}
\begin{proof}
Take $x \in G$ such that
$(I^H)^x I^H \ne \{0\}$. There exist $h_1, h_2 \in H$ such that $\{0\} \ne (I^{h_1})^x I^{h_2} \subseteq I^{h_1 x} I^{h_2}$,  
by Lemma~\ref{lem:2}(a). 
By Lemma~\ref{lem:12}(b), we have $I^{h_1 x} \cdot (I^{h_2} (S_{{h_2}^{-1}} S_{h_2} )) =  I^{h_1 x} \cdot ( I^{h_2}) = I^{h_1 x} I^{h_2}$. 
Hence, Lemma~\ref{lem:12}(a) applies to the $S_e$-ideal $I^{h_1x} I^{h_2}$. Thus, $\{0\} \ne (I^{h_1 x} I^{h_2})^{{h_2}^{-1}} \subseteq I^{h_1 x{h_2}^{-1}} I$.
By assumption, $h_1 x{h_2}^{-1} \in H$ and hence $x\in H$.
\end{proof}

\begin{lem}\label{lem:18}
Suppose that $S$ is nearly epsilon-strongly $G$-graded and that $S_e$ is
$G$-semiprime. Furthermore, let
$H$ be a subgroup of $G$ and let $I$ be an
$H$-invariant ideal of $S_e$ such that $I^x I = \{0\}$ for every $x \in G \setminus H$. Then the following assertions hold:
\begin{enumerate}[{\rm (a)}]
\begin{item}
The ideal $I$ does not contain any nonzero nilpotent
$H$-invariant ideal. 
\end{item}
\begin{item}
Let $W$ be a subgroup of $H$ of finite index.
Then $I$ does not contain any nonzero nilpotent
$W$-invariant ideal.
\end{item}
\end{enumerate}
\end{lem}
\begin{proof}
(a): Seeking a contradiction, suppose that 
$J$ is a nonzero $H$-invariant ideal of $S_e$ such that 
$J^2 = \{ 0 \}$ and $J \subseteq I$.
First we show that $J^x J = \{0\}$ for every $x \in G$. 
Indeed, for $x \in H$ we have $J^x J \subseteq J^2 =~\{0\}$ while
for $x \in G \setminus H$ we have $J^x J \subseteq I^x I = \{0\}$ by Lemma~\ref{lem:2}(c).  
Next, note that $J^G = \sum_{x \in G} J^x$ is a nonzero $G$-invariant ideal of $S_e$ by Lemma~\ref{lem:power}. 
We claim that $J^G J^G = \{0\}$. 
If we assume that the claim holds, then
we get the desired contradiction, since $S_e$ 
is assumed to be $G$-semiprime.
Now, we prove the claim.
Seeking a contradiction, suppose that $J^G J^G \ne \{0\}$. 
Then $J^G J^G = \left( \sum_{x \in G} J^x \right) \left(\sum_{y \in G} J^y \right) = \sum_{x,y \in G} J^x J^y \ne~\{0\}.$
Hence there are $x, y \in G$ such that $J^x J^y \ne \{0\}$. 
By Lemma~\ref{lem:12}(b), we have $J^x J^y (S_{y^{-1}} S_y) = J^x J^y$, and therefore Lemma~\ref{lem:12}(a) implies that $\{0\} \ne (J^x J^y)^{y^{-1}}$. 
Moreover, by Corollary~\ref{cor:multi1}, we have $\{0\} \ne (J^x J^y)^{y^{-1}} = (J^x)^{y^{-1}} (J^y)^{y^{-1}} \subseteq J^{x y^{-1}} J = \{0\}$, which is a contradiction.

(b): Seeking a contradiction, suppose that $J \subseteq I$ is a nonzero 
$W$-invariant ideal of $S_e$ such that $J^2 = \{0\}$. Let $Wx_1, Wx_2, \dots, Wx_n$ be a set of representatives of the right cosets of $W$ in $H$ and, for every $i \in \{1, \dots, n \}$, let $J^{Wx_i} := \sum_{y \in W} J^{y x_i}$. 
We wish to prove that $J' := J^{Wx_1} + J^{W x_2} + \ldots + J^{W x_n}$ is a nonzero 
$H$-invariant nilpotent ideal contained in $I$. 

To begin with, note that for all $y_1, y_2 \in W$ and $i \in \{1, \dots, n\}$ we have $$J^{y_1 x_i} J^{y_2 x_i} = S_{(y_1 x_i)^{-1}} J S_{y_1 x_i} S_{(y_2 x_i)^{-1}} J S_{y_2 x_i }
\subseteq
S_{(y_1 x_i)^{-1}} J S_{y_1 y_2^{-1}} J S_{y_2 x_i }.
$$ 
Using that 
$y_1 y_2^{-1} \in W$
and that
$J$ is 
$W$-invariant,  Proposition~\ref{prop:switch1}(b) yields $ J S_{y_1 y_2^{-1}} J = S_{y_1 y_2^{-1}}  J J  = \{0\}$. 
Hence, $J^{y_1 x_i} J^{y_2 x_i} = \{0\}$, and therefore it follows that $$\left( J^{W x_i} \right)^2 = \left( \sum_{y_1 \in W} J^{y_1 x_i} \right) \left(\sum_{y_2 \in W} J^{y_2 x_i} \right) = \sum_{y_1, y_2 \in W} J^{y_1 x_i} J^{y_2 x_i} = \{0\}. $$ 
In other words, $J^{W x_i}$ is a nilpotent ideal for every $i \in \{1, \dots, n\}$. Since $J'$ is a finite sum of nilpotent ideals,
we conclude that $J'$ is also a nilpotent ideal.

Next, we prove that $J'$ is
$H$-invariant. 
For this we repeatedly use Lemma~\ref{lem:2}.
Note that for all $i \in \{1, \dots, n\}$ and $y \in H$, we have $(J^{W x_i})^y \subseteq J^{W x_i y} = J^{W x_j}$ for some
$j \in \{1,\ldots,n\}$ with $W x_i y = W x_j$. 
Now, by Lemma~\ref{lem:sums1},
$(J')^y = (J^{Wx_1})^y + \ldots + (J^{Wx_n})^y \subseteq J'$
and hence $J'$ is
$H$-invariant.
Finally, we show that $J' \subseteq I$.
Note that $J \subseteq I$ implies $J^{y x_i} \subseteq I^{y x_i}$ for every $y \in W$. 
In addition, we have $I^{y x_i} \subseteq I$, since $I$ is 
$H$-invariant.
It follows that $J^{W x_i} \subseteq I$ for every $i \in \{1, \dots, n\}$, which gives the inclusion $J' \subseteq I$. 

Summarizing, we have established that $J'$ is indeed an 
$H$-invariant nilpotent ideal contained in $I$, but by virtue of (a) we must have $J' = \{0\}$. 
However, writing $Wx_j$ for the right coset containing $e$, we 
get
$\{0\} \ne J = J^e \subseteq J^{Wx_j} \subseteq J'$. 
This contradiction proves the assertion.
\end{proof}

\section{Graded prime ideals}

Recall that $S$ is a $G$-graded ring.
\label{Sec:GradedPrime}
In this section, we
obtain a correspondence between graded prime ideals of $S$ and
$G$-prime ideals of $S_e$, in the case when $S$ is nearly epsilon-strongly $G$-graded.
Using that correspondence, we establish a primeness result in the case when $G$ is ordered (see Corollary~\ref{Cor:Ordered}).
That result will be generalized in Section~\ref{Sec:Sufficient}, using more elaborate methods.
We wish to emphasize that the
rest of this article does not depend on the results of this section.

\begin{defi}
An ideal $I$ of $S$ is called \emph{graded} if 
$I = \bigoplus_{x \in G} (I \cap S_x)$. 
\end{defi}

\begin{exa}
This example illustrates that a graded ring may have infinitely many ideals but only trivial graded ideals.
Indeed, consider the complex Laurent polynomial ring equipped with the standard $\mathbb{Z}$-grading, that is, $\C[t, t^{-1}] = \bigoplus_{i \in \mathbb{Z}} \C t^i$. 
This is clearly a strong $\mathbb{Z}$-grading and hence also nearly epsilon-strong. 
Every point of the circle gives rise to a maximal ideal of $\C[t, t^{-1}]$.
On the other hand, the only graded ideals are $\{0\}$ and $\C[t, t^{-1}]$.   
\end{exa}

Let $I$ be an ideal of $S$. Then $I_e := I \cap S_e$ is an $S_e$-ideal. Conversely, if $J$ is an $S_e$-ideal, then $SJS$ is a graded ideal of $S$.
For strongly graded rings we have the following bijection:

\begin{prop}[{\cite[Prop.~2.11.7]{nastasescu2004methods}}]\label{prop:strongly-bijection}
If $S$ is unital strongly $G$-graded, 
then the map $I \mapsto I_e$ is a bijection between 
the set of graded ideals of $S$ and the set of $G$-invariant ideals of $S_e$.
\end{prop}

We now generalize Proposition~\ref{prop:strongly-bijection} to nearly epsilon-strongly graded rings
(see Theorem~\ref{thm:bijection1}).
To this end, we need three lemmas.

\begin{lem}[{cf.~\cite[Expl.~2.7.3]{passman1984infinite}}]\label{lem:weak_char}
If $S_e$ is $s$-unital and $I$ is an ideal of $S_e$,
then $I$ is 
$G$-invariant if and only if $(SIS)_e= I$.
\end{lem}

\begin{proof}
Suppose that $I^x = S_{x^{-1}} I S_x \subseteq I$ for every $x \in G$. Then $(SIS)_e = SIS \cap S_e \subseteq I$. The reversed inclusion follows since $S_e$ is $s$-unital. 
Conversely, suppose that $(S I S)_e = I$. 
Then $S_{x^{-1}} I S_x \subseteq S I S \cap S_e = (SIS)_e = I$ for every $x \in G$. Thus, $I$ is
$G$-invariant.
\end{proof}

\begin{lem}\label{lem:weak1}
If $I$ is a graded ideal of $S$, then $I_e$ is a  
$G$-invariant ideal of $S_e$.
\end{lem}
\begin{proof}
Take $x \in G$. Then
 $S_{x^{-1}} I_e S_x \subseteq ( S_{x^{-1}} I S_x ) \cap
( S_{x^{-1}} S_e S_x )  \subseteq I \cap S_e = I_e$.
\end{proof}

\begin{lem}[{cf.~\cite[Prop.~1.1.34]{hazrat2016graded}}]\label{lem:nearly1}
Suppose that $S$ is nearly epsilon-strongly $G$-graded. If $I$ is a graded ideal of $S$, then $SI_e S = SI_e = I_e S = I$. 
\end{lem}
\begin{proof}
Using that $S$ is $s$-unital, we get $I_e \subseteq I_e S$ and $I_e \subseteq S I_e$. 
Hence, $S I_e \subseteq S I_e S \subseteq I$ and similarly 
$I_e S \subseteq S I_e S \subseteq I$. 
Next, we prove that $I \subseteq S I_e$.  
Since $I$ is graded, it is enough to show that $a_x \in S I_e$ for every homogeneous $a_x \in I \cap S_x$.
Indeed, since $S$ is nearly epsilon-strongly $G$-graded, we have $a_x = \epsilon_x(a_x) \cdot a_x$ for some $\epsilon_x(a_x) \in S_x S_{x^{-1}}$. 
Write $\epsilon_x(a_x) = \sum_i c_i b_i$ for finitely many $c_i \in S_x$ and $b_i \in S_{x^{-1}}.$ 
Then $a_x = \sum_i c_i b_i a_x$. 
Note that for any $i$ we have $b_i a_x \in S_{x^{-1}} S_x \subseteq S_e$ and $b_i a_x \in I$ thus yielding $a_x \in I \cap S_e = I_e$. Hence, $a_x = \sum_i c_i b_i a_x \in S_x I_e \subseteq S I_e.$ 
By an analogous argument the inclusion
$I \subseteq I_e S$ follows.
We conclude that $I = S I_e = I_e S$.
Consequently, $I = S I_e \subseteq S I_e S \subseteq I$.
\end{proof}

\begin{thm}
\label{thm:bijection1}
Suppose that $S$ is nearly epsilon-strongly $G$-graded. 
The map $I \mapsto I_e$ is a bijection between the sets
$\{$graded ideals of $S \}$
and 
$\{ G$-invariant ideals of $S_e \}$.
The inverse map is given by $J \mapsto SJS.$
\end{thm}

\begin{proof}
Let $I$ be a graded ideal. By Lemma~\ref{lem:weak1}, $I_e$ is a 
$G$-invariant ideal of $S_e$. In other words, the map $I \mapsto I_e$ is well-defined.
Furthermore, by Lemma~\ref{lem:nearly1}, we have $S I_e S = I$ establishing that $I \mapsto I_e$ is injective.
Next, suppose that $J$ is a 
$G$-invariant ideal of $S_e$. By Lemma~\ref{lem:weak_char}, $(SJS)_e=J$ proving that $I \mapsto I_e$ is surjective.
\end{proof}

Later on we will apply Theorem~\ref{thm:bijection1} to Leavitt path algebras (see Section~\ref{Sec:LPA}).

\begin{defi}\label{def:g-prime}
A proper graded ideal $P$ of $S$ is called \emph{graded prime} if for all graded ideals $A,B$ of $S$,
	we have $A \subseteq P$ or $B \subseteq P$
	whenever
	$A B \subseteq P$.
		A proper $G$-invariant ideal $Q$ of $S_e$
		is called \emph{$G$-prime}
		if
		for all $G$-invariant ideals $A,B$ of $S_e$, we have $A \subseteq Q$ or $B \subseteq Q$
		whenever
		$A B \subseteq Q$.
	The ring $S_e$ is called \emph{$G$-prime}
if $\{ 0 \}$ is a $G$-prime ideal of $S_e$.
\end{defi}

For unital strongly $G$-graded rings, the bijection $I \mapsto I_e$ from Theorem~\ref{thm:bijection1} restricts to a bijection between graded prime ideals of $S$ and $G$-prime ideals of $S_e$ (see \cite[Prop.~2.11.7]{nastasescu2004methods}).
We proceed to show that the same holds for nearly epsilon-strongly $G$-graded rings.

\begin{lem}
\label{lem:Lemma8745}
Suppose that $S$ is nearly epsilon-strongly $G$-graded.
If $I$ is a
graded ideal of $S$ such that $I_e$ is 
a $G$-prime ideal of $S_e$, then $I$ is graded prime. 
\end{lem}
\begin{proof}
Suppose that $A,B$ are graded ideals of $S$ such that $A B \subseteq I$. 
Then $A_e B_e \subseteq AB \cap S_e \subseteq I \cap S_e = I_e$. By Theorem~\ref{thm:bijection1}, the $S_e$-ideals $A_e$, $B_e$ are 
$G$-invariant. 
Since $I_e$ is
$G$-prime, we have $A_e \subseteq I_e$ or $B_e \subseteq I_e$. 
Assume
w.l.o.g. 
that $A_e \subseteq I_e$. Then $A = S A_e S \subseteq S I_e S = I $ by Theorem~\ref{thm:bijection1}. 
Thus, $I$ is a graded prime ideal of $S$. 
\end{proof}

\begin{lem}
\label{lem:Lemma422}
Suppose that $S$ is nearly epsilon-strongly $G$-graded.
If $I$ is a graded prime ideal of $S$, then $I_e$ is a
$G$-prime ideal of $S_e$.
\end{lem}
\begin{proof}
Clearly, $I_e$ is an ideal of $S_e$.
Suppose that $A, B$ are 
$G$-invariant ideals of $S_e$ such that $A B \subseteq I_e$. We need to show that $ A \subseteq I_e$ or $ B \subseteq I_e$. By Theorem~\ref{thm:bijection1}, $ A = (S A S)_e$, $ B = (S B S)_e$
and $S (A B) S \subseteq S I_e S = I$. 
Clearly, $S A S$ and $S B S$ are graded ideals of $S$.
By Lemma~\ref{lem:nearly1},
$SAS SBS = S (A B) S \subseteq I$. 
Since $I$ is graded prime, we 
have $SAS \subseteq I$ or $SBS \subseteq I$.
Assume
w.l.o.g. 
that $SAS \subseteq I$. Then $A \subseteq I_e$ and thus $I_e$ is 
$G$-prime. 
\end{proof}

By combining Lemma~\ref{lem:Lemma8745} and Lemma~\ref{lem:Lemma422} we get the desired bijection:

\begin{thm}\label{thm:graded_prime}
Suppose that $S$ is nearly epsilon-strongly $G$-graded.
The map $I \mapsto I_e$ restricts to a bijection
between
the sets
$\{$graded prime ideals of $S \}$
and 
$\{ G$-prime ideals of $S_e \}$.
\end{thm}

We now generalize a well-known result by N\u{a}st\u{a}sescu and Van Oystaeyen 
to the setting of $s$-unital group graded rings:

\begin{prop}[cf.~{\cite[Prop.~II.1.4]{nastasescu1982graded}}]\label{prop:graded_ideals2}
Suppose that $G$ is
an ordered group
and that $S$ is $s$-unital.
If $I$ is a graded ideal of $S$, then $I$ is graded prime if and only if $I$ is prime.
\end{prop}

\begin{proof}
Suppose that $I$ is graded prime.
For every $k \geq 0$, let $P(k)$ be the following statement: 
$$
\mbox{$a,b \in S$ 
satisfy $a S b \subseteq I$ and $|\Supp(a)|+|\Supp(b)| \leq k$
$\Longrightarrow$ $a \in I$ or $b \in I$.}
$$
We proceed by induction to show that $P(k)$ holds for every $k \geq 0$.

Base case: $k=0$. If $|\Supp(a)|+|\Supp(b)| = 0$, then $a=b=0 \in I$.

Inductive step: Take $k \geq 0$ such that $P(k)$ holds.
Suppose that 
$a S b \subseteq I$ and $|\Supp(a)|+|\Supp(b)| = k+1$.
Put $m := |\Supp(a)|$ and $n := |\Supp(b)|$. Then we can write 
$a = \sum_{i=1}^m a_{x_i}$ and 
$b = \sum_{j=1}^n b_{y_j}$ where
$x_1,\ldots,x_m \in G$ and $y_1,\ldots,y_n \in G$ satisfy 
$x_1 < \dots < x_m$ and 
$y_1 < \dots < y_n$.
Take $z \in G$. 
For any $s_z\in S_z$ we have
$a s_z b \in I$.
Using that $G$ is an ordered group and that $I$ is a graded ideal, we get 
$a_{x_m} s_z b_{y_n} \in I$.
This shows that
$a_{x_m} S b_{y_n} \subseteq I$. 
By graded primeness of $I$, and $s$-unitality of $S$, we get $a_{x_m} \in S a_{x_m} S \subseteq I$ or $b_{y_n} \in S b_{y_n} S \subseteq I$.

\underline{Case 1: $a_{x_m} \in I$.} Put $a' = a-a_{x_m}$. Then $a' S b  = a S b - a_{x_m} S b \subseteq I - I = I$. 
Since $|\Supp(a')|+|\Supp(b)| < k+1$, the induction hypothesis yields that $a' \in I$ or $b \in I$, and hence that $a = a' + a_{x_m} \in I$ or $b \in I$.

\underline{Case 2: $b_{y_n} \in I$.} Put $b' = b-b_{y_n}$. Then $a S b' = a S b - a S b_{y_n} \subseteq I - I = I$.
Since $|\Supp(a)|+|\Supp(b')| < k+1$, the induction hypothesis yields that $a \in I$ or $b' \in I$, and hence that $a \in I$ or $b = b' + b_{y_n} \in I$.

Therefore, $P(k+1)$ holds.

Now, let $A,B$ be nonzero ideals of $S$ with 
$AB \subseteq I$.
Seeking a contradiction, suppose that there are $a \in A \setminus I$ and $b \in B \setminus I$.
Since $A$ and $B$ are ideals, it follows that $a S b \subseteq AB \subseteq I$. Since $P(k)$ holds for every $k \geq 0$, we get that $a \in I$ or $b \in I$, which is a contradiction.

The converse statement is trivial.
\end{proof}

By Proposition~\ref{prop:graded_ideals2}, we immediately obtain the following partial generalization of a result by Abrams and Haefner~\cite[Thm.~3.2]{AbHa93}:

\begin{cor}
Suppose that $G$ is
an ordered group
and that $S$ is $s$-unital. Then $S$ is graded prime if and only if $S$ is prime.
\end{cor}

Combining the above result with Theorem~\ref{thm:graded_prime}, we
immediately get the following:

\begin{cor}\label{Cor:Ordered}
Suppose that $G$ is an ordered group and that $S$ is nearly epsilon-strongly $G$-graded.
Then $S$ is prime if and only if $S_e$ is 
$G$-prime.
\end{cor}

\begin{exa}
Let $R$ be a unital ring.

(a)
Consider the Laurent polynomial ring 
$R[t, t^{-1}] = \bigoplus_{i \in \mathbb{Z}} R t^i$
equipped with its canonical strong $\mathbb{Z}$-grading.
Since $t$ is central in $R[t,t^{-1}]$, any ideal $I$ of $R$ satisfies $t^{-n} I t^n = I$. Thus, every ideal of $R$ is $\mathbb{Z}$-invariant. Hence, $R$ being 
$\mathbb{Z}$-prime is equivalent to $R$ being prime. Therefore, Corollary \ref{Cor:Ordered} implies that $R[t, t^{-1}]$ is prime if and only if $R$ is prime.

(b)
More generally,
let $G$ be an ordered group and consider the group ring $R[G]$. 
Note that for any ideal $I$ of $R=(R[G])_e$ we have $\delta_{x^{-1}} I \delta_x = \delta_{x^{-1}} \delta_x I = I$ for every $x \in G$. 
Thus, every ideal of $R$ is 
$G$-invariant. 
By Corollary \ref{Cor:Ordered}, it follows that $R[G]$ is prime if and only if $R$ is prime (see e.\,g. \cite[Thm.~6.29]{lam2001first}).
\end{exa}

\section{The ``easy'' direction}
\label{Sec:Easy}

Recall that $S$ is a $G$-graded ring.
In this section, 
we prove the implication (b)$\Rightarrow$(a) of Theorem~\ref{thm:mainNew}  for non-degenerately $G$-graded rings (see Proposition~\ref{prop:easy_dir}).

\begin{lem}[{cf.~\cite[Lem.~1.4]{passman1984infinite}}]\label{lem:1}
Suppose that $S$ is non-degenerately $G$-graded, that $H$ is a subgroup of $G$, and that $I$ is an ideal of $S_e$ which satisfies $I^x I = \{0\}$ for every $x \in G \setminus H$. Then the following two assertions hold:
\begin{enumerate}[{\rm (a)}]
\begin{item}
$IS_x I = \{0\}$ for every $ x \in G \setminus H$.
\end{item}
\begin{item}
$I S I \subseteq I S_H \subseteq S_H$.
\end{item}
\end{enumerate}
\end{lem}
\begin{proof}
(a):
Take $x\in G \setminus H$ and $s \in I S_x I$.
By assumption,
$S_{x^{-1}} I S_x I = I^x I = \{0\}$ and hence $S_{x^{-1}} s = \{0\}$.
Using that $S$ is non-degenerately $G$-graded, we 
get that
$s = 0$.

(b): Employing part (a), we get $I S I = \bigoplus_{x \in G} I S_x I = \bigoplus_{x \in H} I S_x I  \subseteq S_H$.
\end{proof}

\begin{lem}\label{lem:NonDegIdeals}
Suppose that $S$ is non-degenerately $G$-graded and that $N$ is a subgroup of $G$.
If $\tilde{A}$ is a nonzero subset of $S_N$, then $S\tilde{A}S$ is a nonzero ideal of $S$.
\end{lem}

\begin{proof}
Clearly, $S\tilde{A}S$ is an ideal of $S$.
Choose a nonzero $a \in \tilde{A}$.
Let $n\in \Supp(a) \subseteq N$.
By non-degeneracy of the $G$-grading, 
there is some $s_{n^{-1}} \in S_{n^{-1}}$ and some $t_e \in S_{e}$ such that
$t_e a_n s_{n^{-1}} \neq 0$.
Therefore,
$t_e a s_{n^{-1}} \in S\tilde{A}S \setminus \{0\}$.
This shows that $S \tilde{A} S$ is nonzero.
\end{proof}

\begin{prop}[{cf.~\cite[Thm.~1.3]{passman1984infinite}}]
\label{prop:easy_dir}
Suppose that $S$ is
non-degenerately
$G$-graded
and that there exist
	\begin{enumerate}[{\rm (i)}]
\item subgroups $N \lhd H \subseteq G$,
\item an
$H$-invariant ideal $I$ of $S_e$ such that 
$I^x I = \{0\}$ for every $x \in G \setminus H$, and
\item nonzero ideals $\tilde{A}, \tilde{B}$ of $S_N$ such that
$\tilde{A}, \tilde{B} \subseteq I S_N$, and
$\tilde A S_H \tilde B = \{0\}$.
\end{enumerate}
Then
$S$ is not prime.
\end{prop}

\begin{proof}
If $x \in H$, then the second condition in (iii) implies that $\tilde A S_x \tilde B =\{0\}$. 

If $x \in G \setminus H$, then the first condition in (iii) implies that $\tilde{A} S_x \tilde{B} \subseteq (I S_N) S_x (I S_N)$. Since $S_N S_x = \bigoplus_{n \in N} S_n S_x \subseteq \bigoplus_{n \in N} S_{nx}$ and $nx \in G \setminus H$, it follows from Lemma~\ref{lem:1} that $$I S_N S_x I \subseteq \bigoplus_{n \in N} I S_{nx} I = \{0\}.$$ 
Hence, $\tilde{A} S_x \tilde{B} = \{0\}$ for every $x\in G$, and thus $\tilde{A} S \tilde{B} = \{0\}$.
Now, by (iii) and 
Lemma~\ref{lem:NonDegIdeals} 
it follows that $A := S \tilde{A} S$ and $B:= S \tilde{B} S$ are nonzero ideals of $S$ satisfying
$AB=(S \tilde A S)(S \tilde B S) \subseteq S (\tilde A  S \tilde B) S=\{0\}$.
This shows that $S$ is not prime.
\end{proof}

\begin{rem}
Note that $N$ is not required to be finite in Proposition~\ref{prop:easy_dir}.
\end{rem}

In an attempt to ease the technical notation, we now introduce the following notion.

\begin{defi}[NP-datum]
Let $S$ be a $G$-graded ring.
An \emph{NP-datum} for $S$ is a quintuple
$(H,N,I,\tilde{A},\tilde{B})$
with the following three properties:
\begin{itemize}
\item[(NP1)]
$H$ is a subgroup of $G$, and $N$ is a finite normal subgroup of $H$,
\item[(NP2)] $I$ is a nonzero $H$-invariant ideal of $S_e$ such that 
$I^x I = \{0\}$ for every $x \in G \setminus H$, and
\item[(NP3)] $\tilde{A}, \tilde{B}$ 
are nonzero ideals
of $S_N$ such that
$\tilde{A}, \tilde{B} \subseteq I S_N$, and
$\tilde{A} \tilde{B} = \{0\}$.
\end{itemize}
An NP-datum $(H,N,I,\tilde{A},\tilde{B})$
is said to be \emph{balanced} if it satisfies the following property:
\begin{itemize}
    \item[(NP4)] $\tilde{A}, \tilde{B}$ 
are nonzero ideals
of $S_N$ such that
$\tilde{A}, \tilde{B} \subseteq I S_N$, and
$\tilde{A} S_H \tilde{B} = \{0\}$.
\end{itemize}
\end{defi}

\begin{rem}\label{rem:NP3NP4}
(a) If $S$ is
nearly epsilon-strongly $G$-graded, 
then (NP4) implies (NP3).

(b) Suppose that $S$ is $s$-unital strongly $G$-graded.
An NP-datum  $(H,N,I,\tilde{A},\tilde{B})$ for $S$
is necessarily balanced whenever
$\tilde{A}$ or
$\tilde{B}$ is $H$-invariant.
Indeed, suppose that
$\tilde{A}$ is $H$-invariant.
For any $h \in H$, we get that 
$\tilde{A} S_h \tilde{B} =
S_e \tilde{A} S_h \tilde{B} = 
S_h S_{h^{-1}} \tilde{A} S_h \tilde{B} \subseteq S_h \tilde{A} \tilde{B} = \{ 0 \}$
by Lemma~\ref{lem:strongSunital}.
The proof of the case when $\tilde{B}$ is $H$-invariant is analogous.
\end{rem}

\begin{cor}\label{cor:easy_dir2}
Suppose that $S$ is non-degenerately $G$-graded.
If $S_e$ is not $G$-prime, then $S$ has a balanced NP-datum
$(H,N,I,\tilde{A},\tilde{B})$
for which $\tilde{A},\tilde{B}$ are $H/N$-invariant.
\end{cor}

\begin{proof}
If $S_e$ is not 
$G$-prime, then there are 
nonzero $G$-invariant ideals $\tilde A, \tilde B$ of $S_e$ such that $\tilde{A} \tilde{B} = \{0\}$.
We claim that $(G,\{e\},S_e,\tilde{A},\tilde{B})$ is a balanced NP-datum. 
Conditions~(NP1),~(NP2) and (NP3) are  trivially satisfied. 
We now check condition (NP4). 
Take $x\in G$.
Seeking a contradiction, suppose that $\tilde{A} S_x \tilde{B} \neq \{ 0 \}$. Note that $\tilde{A} S_x \tilde{B} \subseteq S_x$.
By non-degeneracy of the $G$-grading, $S_{x^{-1}} \cdot \tilde{A} S_x \tilde{B} \neq \{ 0 \}$.
Since $\tilde{A}$ is $G$-invariant, we get that $S_{x^{-1}} \tilde{A} S_x \tilde{B} \subseteq \tilde{A} \tilde{B} = \{ 0 \}$, which is a contradiction. 
Note that, trivially, $\tilde{A},\tilde{B}$ are both $G/\{e\}$-invariant.
\end{proof}

By combining the above results we get the following.

\begin{cor}\label{cor:PrimeNecessary}
Suppose that $S$ is non-degenerately
$G$-graded.
If $S$ is prime, then
$S_e$ is $G$-prime.
\end{cor}

\section{Passman pairs and the Passman replacement argument}
\label{Sec:PassmanPairs}
In this section, we generalize a technical result by Passman \cite{passman1984infinite}. Recall that
$S$
is a  $G$-graded ring. 
We are interested in pairs $(J, M)$ where $J$ is a nonzero ideal of $S_e$ and $M \subseteq G$ is a subset such that $J^x J = \{0\}$ for every $x \in G \setminus M$. 
Given such a pair $(J,M)$, where $M$ is of a certain type, we will find another pair $(K,L)$ where $K \subseteq J$ is a nonzero ideal of $S_e$ and $L$ is a subgroup of $G$. Crucially, the new pair $(K,L)$ satisfies $K^x K = \{0\}$ for $x \in G \setminus L$. 
Passman's original proof relies on $S$ being unital and strongly $G$-graded, and provides a construction of the ideal $K$. 
As we will see, 
his main argument generalizes to our extended setting, although we do not get an explicit description of the ideals.

\begin{defi}
If $I$ is a nonzero ideal of $S_e$ and $M\subseteq G$ is such that $I^x I = \{0\}$ for every $x \in G \setminus M$, then we call $(I, M)$ a \emph{Passman pair}.
\end{defi}

\begin{prop}[{cf.~\cite[Lem.~2.1]{passman1984infinite}}]\label{prop:coset1}
Suppose that $S$ is nearly epsilon-strongly $G$-graded
and that $(J, M)$ is a Passman pair where $M = \bigcup_{k=1}^n g_k G_k$
for some subgroups
$G_1,\ldots,G_n$ of $G$
and $g_1,\ldots,g_n \in G$.
Then there exist a nonzero ideal $K \subseteq J$ of $S_e$ and a subgroup $L$ of $G$
such that $(K, L)$ is a Passman pair. 
In addition, $[ L : L \cap G_k ] < \infty$ 
for some $k\in \{1,\ldots,n\}$.
\end{prop}

We now fix a group $G$ and 
a finite family $\{G_1,\ldots,G_n\}$ of 
subgroups of $G$.
To establish Proposition~\ref{prop:coset1}, we need the following:

\begin{lem}\label{lem:coset1}
Suppose that $S$ is nearly epsilon-strongly $G$-graded. 
Let $B = \{ A_1, A_2, \dots, A_l \}$ be a family of subgroups of $G$ such that for all $i,j\in \{1,\ldots,l\}$ there is some $k\in \{1,\ldots,n\}$ such that $A_j \subseteq G_k$, 
and $A_i \cap A_j \in B$. 
Let $(J,M)$ be a Passmain pair. Suppose that $M = \bigcup_{k=1}^t  g_k A_{n_k}$ where
$n_1,\ldots,n_t \in \{1,\ldots,l\}$ and $g_1,\ldots,g_t \in G$. 
Then there exist a nonzero ideal $K \subseteq J$ of $S_e$ and a subgroup $L$ of $G$ such that $(K, L)$ is a Passman pair. 
In addition, if $B$ is non-empty then $[L : L \cap A_j ] < \infty$ for some $j \in \{1,\ldots,l\}$. 
\end{lem}
\begin{proof}
The proof proceeds by induction over $|B|$. If $|B| = 0$, then the assumption that $(J, \emptyset)$ is a Passman pair implies that $(J, \{e \})$ is a Passman pair. 
Next, suppose that $|B| \geq 1$. Let $A$ be a maximal element of $B$ ordered by inclusion and note that $B' = B \setminus \{ A \}$ is closed under intersections. 
We consider the following set of Passman pairs:
$$ P := \{ (K,N) \mid \{0\} \ne K \subseteq J, \,\, N = \bigcup_{j=1}^s g_j A_{k_j} \text{ for some } k_1,\ldots,k_s \in \{1,\ldots,l\}, \, g_1,\ldots,g_s \in G \} $$
Note that $P$ is non-empty since $(J,M) \in P$. 
For $(K,N) \in P$ with $N = \cup_{j=1}^s g_j A_{k_j}$, we let $\Supp(K,N)$ be the subset 
$\{A_{k_1}, \ldots, A_{k_{s}} \} \subseteq B$.
Let $\deg (K,N)$ be the number of times that $A = A_{k_j}$ in the expression of $N$.

Now, choose $(K,N) \in P$ of minimal degree. 
We consider two mutually exclusive cases:

\underline{Case 1: $\deg (K,N) = 0$.} In this case $\Supp(K,N) \subseteq B'$. Hence, the induction hypothesis applies and we conclude that there exists some Passman pair $(I, L)$ such that $\{0\} \ne I \subseteq K \subseteq J$ and $L$ is a subgroup of $G$.

\underline{Case 2: $\deg (K,N) = m > 0$.} Let $ N = z_1 A \cup z_2 A \cup \ldots \cup z_m A \cup T$ where $T$ is a finite union of cosets of groups in $B'$. Put $$ L := \Big \{ g \in G \mid g \Big ( \bigcup_{i=1}^m z_i A \Big ) = \bigcup_{i=1}^m z_i A \Big \}. $$
Our goal 
is to prove that $(K,L)$ is a Passman pair. 
Note that $L$ is the stabiliser of $\bigcup_{i=1}^m z_i A$. Thus, $L$ is in fact a subgroup of $G$.
Take $x\in G$ such that $K^x K \ne \{0\}$. We will show that $x \in L$. Indeed, if $h = x^{-1} h'$ for some $h' \in G \setminus N$, then
\begin{align*}(K^x K)^h (K^x K)
=
((K^x)^h K^h) (K^x K) \subseteq K^{xh} K^h K^x K 
\subseteq K^{xh} K = K^{x x^{-1} h'} K = K^{h'} K = \{0\}
\end{align*}
where we have used Lemma~\ref{lem:IxIdeal}, Lemma~\ref{lem:2} and Corollary~\ref{cor:multi1}.
Similarly, if $h \in G \setminus N$, then $(K^x K)^h (K^x K) \subseteq (K^{xh} K^{h}) (K^x K) \subseteq K^h K = \{0\}.$ 
In other words, $(K^x K)^h (K^x K) = \{0\}$ for every $ h \in (G \setminus N) \cup x^{-1}(G \setminus N) = G \setminus (N \cap x^{-1} N).$ Thus, $(K^x K, N \cap x^{-1} N)$ is a Passman pair. 
Since $K^x K \subseteq K \subseteq J$ and $N \cap x^{-1} N$ is a finite union of cosets in $B$, it follows that $(K^x K, N \cap x^{-1}N) \in P$. 
Let $m' := \deg (K^x K, N \cap x^{-1}N )$. 
By minimality of $m$, we have $m' \geq m > 0$. 
Note that $ x^{-1} N = x^{-1} z_1 A \cup x^{-1} z_2 A \cup \ldots \cup x^{-1} z_m A \cup x^{-1} T.$ 
Since $A$ is maximal, the $m'$ cosets of $A$ in $N \cap x^{-1}N$ must come from $( \cup_{i=1}^m z_i A  ) \cap ( \cup_{i=1}^m x^{-1} z_i A ).$ 
Moreover, cosets are either equal or disjoint, and hence $m' \leq m$. 
This shows that $m'=m$ and $ \cup_{i=1}^m z_i A = x^{-1} ( \cup_{i=1}^m z_i A )$ which in turn shows that $x \in L$. 
Summarizing, we have established that $K^x K= \{0\}$ for every $x \in G \setminus L$, i.\,e. $(K, L)$ is a Passman pair. 

Now, suppose that $B$ is non-empty. It remains to show that $[L : L \cap A_j ] < \infty$ for some $j \in \{1,\ldots,l\}$. 
Consider $G$ acting from the left on the left cosets of $A$,
i.\,e. $G \curvearrowright
 \{ g A \mid g \in G \}$ by $g_1 \cdot g_2 A = g_1 g_2A$ for all $g_1, g_2 \in G.$ 
Note that $L$ acts on the finite set of cosets $ D = \{ z_1 A, z_2 A, \dots, z_m A \}$.
Let $i \in \{1,\ldots,m\}$ be arbitrary.
A short computation shows that $\Stab_G(z_i A) =z_i A z_i^{-1}.$  Thus, $\Stab_L(z_iA)= z_iA z_i^{-1} \cap L$. 
Hence, by the orbit-stabilizer theorem we have $| L  \cdot z_i A | = [ L : L \cap z_i A z_i^{-1} ]$. 
Using that $D$ is a finite set, we conclude that the orbit of $z_iA$ is finite, i.\,e. $| L  \cdot z_i A | < \infty$.  Thus, we have $[ L : L \cap z_i A z_i^{-1} ] < \infty$ for every $i \in \{1,\ldots,m\}$.

We consider two mutually exclusive cases.

\underline{Case A: $L \cap z_i A = \emptyset$ for every $i$.} Note that $K^x K = \{0\}$ for every $x \in G \setminus N \cup G \setminus L = G \setminus (N \cap L).$ By the case assumption, we have $N \cap L = T \cap L.$ We see that $T \cap L$ is a finite union of cosets from the set $B'' = \{ A' \cap T \mid A' \in B' \}$.

Note that $| B'' | \leq |B'| < |B|$. By the induction hypothesis, it follows that there is a Passman pair $(I, L')$ satisfying the required properties. 

\underline{Case B: $L \cap z_i A \ne \emptyset$ for some $i$.} Let $a \in A$ be such that $z_i a \in L \cap z_i A$. 
Since $(z_ia)A=z_iA$, we may assume that $z_i \in L$ by choosing another representative of the coset. 
It follows that $L \cap A \cong L \cap z_i A z_i^{-1}$ via the map defined by $a \mapsto z_i a z_i^{-1}$ for every $a \in A \cap L$. 
As noted above we have $[L : L \cap z_i A z_i^{-1} ] < \infty$ and hence $[L : L \cap A] < \infty$. 
Consequently, $[L : L \cap A_j ]< \infty$ with $A_j := A$ as required.
\end{proof}

We are now ready to give a proof of Proposition~\ref{prop:coset1}:

\begin{proof}[Proof of Proposition~\ref{prop:coset1}]
Let $(J,M)$ be a Passman pair where $M = \bigcup_{k=1}^n g_k G_k$ for some subgroups $G_1, G_2,\dots,G_n$ of $G$. 
Furthermore, let $B$ 
denote the closure of $\{ G_1, G_2, \dots, G_n \}$ with respect to intersections. 
Then $M$ is a finite union of left cosets of subgroups of $B$, and we may apply Lemma~\ref{lem:coset1}. 
Hence, there is a Passman pair $(K, L)$ where $L$ is a subgroup of $G$ and $K \subseteq J$ is a nonzero ideal of $S_e$. 
In addition, using that $B$ is non-empty, we have $[ L : L \cap A_i ] < \infty $ for some $A_i \in B$. Since $A_i \subseteq G_k$ for some $k$, it follows that $L \cap A_i \subseteq L \cap G_k$. Consequently, $[L : L \cap G_k ] \leq [L : L \cap A_i ] < \infty$. 
\end{proof}

The following result is a stronger version of Proposition~\ref{prop:coset1}:

\begin{prop}[{cf. \cite[Lem.~2.2]{passman1984infinite}}]\label{prop:coset2}
Suppose that $S$ is nearly epsilon-strongly
$G$-graded and that $W$ is a subgroup of $G$ of finite index.
Let $J$ be a nonzero ideal of $S_e$ such that  $$J^x J = \{0\}, \qquad \forall x \in W \setminus \bigcup_{k=1}^n w_k H_k$$ where $H_1,\ldots,H_n$ are subgroups of $W$ and $w_1,\ldots,w_n \in W$.
Then there is a subgroup $L$ of $G$ and a nonzero ideal $I \subseteq J$ of $S_e$ such that $(I,L)$ is a Passman pair of $S$. In other words, $I^x I = \{0\}$ for every $x \in G \setminus L$. 
In addition, $[L : L \cap H_k ] < \infty$ 
for some $k \in \{1,\ldots,n\}$.
\end{prop}

\begin{proof}
For each positive integer $m$ we let
$A_m$ be the set consisting of all $(h_1, h_2, \ldots, h_m) \in G^m$
such that
\begin{itemize}
\item $J^{h_1} J^{h_2} \cdots J^{h_m} \ne \{0\}$,
\item $e=h_i \text{ for some } i \in \{1,\ldots,m\}$, and
\item $W h_j = W h_i$ if and only if $i=j$.
\end{itemize}
By Proposition~\ref{prop:s-unital}, $S_e$ is $s$-unital and hence $J=J^e \neq \{0\}$.
This shows that $e \in A_1$.
Now, by assumption $[G : W] < \infty$, and hence there is a greatest integer $s$ such that $A_s$ is non-empty.
Pick $\alpha = ( h_1, h_2, \dots, h_s ) \in A_s$ 
and put $K := J^{h_1} J^{h_2} \cdots J^{h_s}.$ 
Using that $\alpha \in A_s$ and that $J^e$ is an ideal of $S_e$, we get that $K \subseteq J^e = J$.
We will construct a set $M \subseteq G$ such that $(K, M)$ is a Passman pair of $S$ where $M$ has the required form for Proposition~\ref{prop:coset1}. 

Take $x \in G$ such that $K^x K \ne \{0\}$.
We
begin by showing 
that $\{ h_1 x, h_2x, \dots, h_s x \}$ represents the same set of right cosets of $W$ as
$\{h_1,h_2,\ldots,h_s\}$. 
Seeking a contradiction, suppose that
there is some $i\in \{1,\ldots,s\}$
such that
$W h_i x \neq W h_j$
for each $j\in \{1,\ldots,s\}$.
By Corollary~\ref{cor:multi1} and Lemma~\ref{lem:2}(a), we get that
\begin{equation}
    \{0\} \ne K^x K
    \subseteq
    (J^{h_1 x} J^{h_2 x} \cdots J^{h_s x}) (J^{h_1} J^{h_2} \cdots J^{h_s}) \subseteq J^{h_i x} J^{h_1} \cdots J^{h_s}.
    \label{eq:4}
\end{equation} 
Hence, $( h_i x, h_1, h_2, \dots, h_s ) \in A_{s+1}$
which contradicts the assumption on $s$. 
Thus,
\linebreak
$\{Wh_1, Wh_2, \ldots, Wh_s\} = \{Wh_1x, Wh_2x, \ldots, Wh_sx\}$.
In particular, 
$h_i x \in W$ for some 
$i \in \{1,\ldots,s\}$. 
By a computation similar to that in \eqref{eq:4}, we get that $\{0\} \ne K^x K \subseteq J^{h_i x} J$. 
Hence, by assumption we have $h_i x \in  \bigcup_{k=1}^n w_k H_k$. 
We have thus proved that $$ K^x K = \{0\}, \qquad \forall x \in G \setminus
\left( \bigcup_{i=1}^n
\bigcup_{k=1}^n h_i^{-1} w_k H_k \right). $$ 
By Proposition~\ref{prop:coset1}, there is a nonzero ideal $I \subseteq K \subseteq J$ of $S_e$ and a subgroup $L$ of $G$ such that $(I, L)$ is a Passman pair. 
Moreover, $[L : L \cap H_k ] < \infty$
for some $k \in \{1,\ldots,n\}$.
\end{proof}

\begin{rem}
Let $S$ be 
nearly epsilon-strongly $G$-graded 
and let $W$ be a subgroup of $G$. 
Then $S_W$ is a nearly epsilon-strongly $W$-graded ring and $(J, \bigcup_{k=1}^n w_k H_k)$ is a Passman pair of $S_W$. 
By Proposition~\ref{prop:coset1}, there is a subgroup $L$ of $W$ and a nonzero ideal $K \subseteq J$ of $S_e$ such that $(K, L)$ is a Passman pair of $S_W$. 
In other words, $K^x K = \{0\}$ for every $x \in W \setminus L$. 
In contrast, note that Proposition~\ref{prop:coset2} gives a Passman pair $(K,L)$ of the larger ring $S$, i.\,e. we have $K^x K = \{0\}$ for every $x \in G \setminus L$. 
\end{rem}

\section{Passman forms and the $\Delta$-method}
\label{Sec:PassmanForms}

Let $S$ be a $G$-graded ring. For nonzero graded ideals $A,B$ of $S$, we have that $A B = \{0\}$ implies $\pi_N(A) \pi_N(B) \subseteq A B = \{0\}$ for every normal subgroup $N$ of $G$. Moreover, if $S$ is non-degenerately $G$-graded, then $\pi_N(A) \ne \{0\}$ and $\pi_N(B) \ne \{0\}$  
by Lemma~\ref{lem:pi_non-degenerate}. In this section, we consider nonzero ideals $A,B$ of $S$ such that $A B = \{0\}$ and show that there exist a normal subgroup $N$ of $G$ and nonzero ideals $\tilde A, \tilde B$ of $S_N$ such that $\tilde A \tilde B = \{0\}$.

Recall that a ring is called \emph{semiprime} if it contains no nonzero nilpotent ideal. Analogously, we make the following definition:

\begin{defi}
If for every $G$-invariant ideal $I$ of $S_e$, $I^2=\{0\}$ implies $I=\{0\}$, then the ring $S_e$ is called \emph{$G$-semiprime}. \end{defi}

\begin{rem}\label{rem:22}
(a) $S_e$ is $G$-semiprime if and only if 
$S_e$ contains no nonzero nilpotent
$G$-invariant ideal.\\
(b) If $S_e$ is $G$-prime, then $S_e$ is $G$-semiprime.
\end{rem}

We record the following result which follows directly from Remark~\ref{rem:22}(b) and Corollary~\ref{cor:easy_dir2}:

\begin{cor}\label{cor:not_semiprime}
Suppose that $S$ is nearly epsilon-strongly $G$-graded.  
If $S_e$ is not 
$G$-semiprime,
then $S$ has a balanced NP-datum.
\end{cor}

Our main task for the remainder of this section is to establish Proposition~\ref{prop:main1} below.
Recall that, for a given group $H$,  $\Delta(H) := \{ h \in H \mid [ H : C_H(h) ] < \infty \}$ 
denotes
its finite conjugate center 
(cf. Section~\ref{Sec:Prel}).

\begin{prop}[{cf.~\cite[Prop.~3.1]{passman1984infinite}}]\label{prop:main1}
Suppose that $S$ is nearly epsilon-strongly $G$-graded  
and that $S_e$ is 
$G$-semiprime. Let $A, B$ be nonzero ideals of $S$ such that $AB=\{0\}$. Then there exist a subgroup $H$ of $G$, a nonzero
$H$-invariant ideal $I$ of $S_e$ and an element $\beta \in B$ such that the following assertions hold:
\begin{enumerate}[{\rm (a)}]
\begin{item}
$I^x I = \{0\}$ for every $x \in G \setminus H$;
\end{item}
\begin{item}
$I \pi_{\Delta(H)}(A) \ne \{0\}$, $I \pi_{\Delta(H)}(\beta) \ne \{0\}$;
\end{item}
\begin{item}
$I \pi_{\Delta(H)}(A) \cdot I \beta = \{0\}$.
\end{item}
\end{enumerate}
\end{prop}

Using Connell's result (cf.~\cite[Lem.~5.2]{passman1989infinite}),
we show
that Proposition~\ref{prop:main1}
holds for the special case of group rings
in the following example.

\begin{exa}\label{ex:7_5}
Let $R$ be a 
unital semiprime ring,
and consider the group ring $R[G] = \bigoplus_{x \in G} R \delta_x$ with its natural strong $G$-grading. 
Let $\Delta := \Delta(G)$ and let $a,b \in R[G]$. The $\Delta$-argument was used by Connell to prove that if $a \delta_x b = 0$ for every $x \in G$, then $\pi_\Delta(a)b=0$.
We show that Proposition~\ref{prop:main1} holds in this special case:

Let $A,B$ be nonzero ideals of $R[G]$ such that $AB = \{0\}$. Put $H := G$ and $I := R$. Since $R[G]$ is non-degenerately $G$-graded, we can choose $\beta \in B$ such that $\beta_e \ne 0$.
Now, note that (a) is trivially satisfied. Moreover, (b) follows from Lemma~\ref{lem:pi_non-degenerate} and the fact that $\beta_e \ne 0$. Next, note that (c) 
asserts
that $R \pi_\Delta(A) \cdot R \beta = \{0\}$. Also note that $R \pi_\Delta(A)  R \beta = R \pi_\Delta(AR) \beta = R \pi_\Delta(A) \beta$.
Now, let $\alpha \in A$ and let $x \in G$. Then $\alpha \delta_x \beta \subseteq A S B = AB = \{0\}$. Applying Connell's $\Delta$-result, we have $\pi_\Delta(\alpha) \beta = 0$, and since $\alpha$ is arbitrary it follows that $\pi_\Delta(A) \beta = \{ 0 \}$. Thus, $R \pi_\Delta(A) \beta = \{0\}$ which shows that (c) is satisfied.
\end{exa}

Before proving Proposition~\ref{prop:main1} 
we show that it also holds in the following special case:

\begin{exa}\label{ex:21}
    Suppose that $G$ is an FC-group  
    and that $S$ is nearly epsilon-strongly $G$-graded. 
    Let $A, B$ be nonzero ideals of $S$ such that $A B = \{0\}$. Put $H := G$, $I := S_e$ and choose a nonzero $\beta \in B$. Since $G$ is an FC-group, it follows that $\Delta := \Delta(G) = G$. Note that (a) is trivially satisfied. Moreover, $I \pi_\Delta(A) = S_e A =A \ne \{0\}$ and $I \pi_\Delta(\beta) = S_e \beta \ni \beta \ne 0$.
    Thus, (b) holds. Finally, $I \pi_\Delta(A) \cdot I \beta = S_e A \cdot S_e \beta \subseteq A S_e B \subseteq AB = \{0\}$. Hence, (c) is satisfied.
\end{exa}

The key bookkeeping device used by Passman \cite{passman1984infinite} is the notion of a \emph{form}. 
We extend his definition to our generalized setting: 
\begin{defi}\label{def:passmanform}
Let $S$ be a $G$-graded ring. Suppose that $A,B$ are nonzero ideals of $S$ such that $AB=\{0\}$. We say that the quadruple $(H,D,I,\beta)$ is a \emph{Passman form} for $(A,B)$ if the following conditions are satisfied:
\begin{enumerate}[(a)]
\begin{item}
$H$ is a subgroup of $G$ and $D = D_G(H)= \{ x \in G \mid [ H : C_H(x) ] < \infty \}$;
\end{item}
\begin{item}
$I$ is an $H$-invariant ideal of $S_e$ such that $I^xI=\{0\}$ for every $x \in G \setminus H$;
\end{item}
\begin{item}
$0 \ne \beta \in B$, $I \beta \ne \{0\}$, and $IA \ne \{0\}$. 
\end{item}
\end{enumerate}
The \emph{size} of a Passman form $(H,D,I,\beta)$ is defined to be the number of right $D$-cosets in $G$ meeting $\Supp(\beta)$. 
\end{defi}

\begin{rem}
Passman (see  \cite[Prop.~7.1]{passman1984infinite}) only considers forms
coming from
unital strongly $G$-graded rings. For that class of rings our definition coincides with 
his
original definition.
\end{rem}

\begin{exa}
Here are two examples of Passman forms:

(a) 
In Example~\ref{ex:7_5}, $(G, \Delta(G), R, \beta)$ is a Passman form. 
Let $g_1, g_2, \dots, g_n \in G$ be such that $\Supp(\beta) \subseteq \bigcup_{i=1}^n g_i \Delta(G)$ is a minimal cover (meaning that it is not possible to choose elements $h_1, \dots, h_m \in G$ such that $\Supp(\beta) \subseteq \bigcup_{i=1}^m h_i \Delta(G)$, for any $m < n$). 
The size of the Passman form $(G, \Delta(G), R, \beta)$ is $n$. 

(b) 
In Example~\ref{ex:21}, $(G,G,S_e, \beta)$ is a Passman form of size $1$.
\end{exa}

Later in this section we will consider Passman forms of minimal size, whose existence is guaranteed by the following:

\begin{prop}[{cf.~\cite[Lem.~7.2]{passman1984infinite}}]\label{prop:passman_exists}
Suppose that $S$ is nearly epsilon-strongly $G$-graded.  
If $A,B$ are nonzero ideals of $S$ such that $AB=\{0\}$, then $(A,B)$ has a Passman form. 
\end{prop}

\begin{proof}
Put $H:=G, D:= \Delta(G)$, and $I:=S_e$.
Note that $I$ is $G$-invariant. 
Furthermore, $IA = S_e A = A \ne \{0\}$. 
Now, let $\beta \in B \setminus \{ 0 \}$. It remains to show that $I \beta \ne \{0\}$. 
To this end, 
write $\beta = \sum_{x \in G} \beta_x$. 
Since $S$ is nearly epsilon-strongly $G$-graded, 
 for every $x \in \Supp(\beta)$, there exists some $\epsilon_x(\beta_x) \in S_x S_{x^{-1}} \subseteq S_e =I$ such that $\epsilon_x(\beta_x) \beta_x = \beta_x$. 
Moreover, there is some $s \in S_e=I$ such that $s \epsilon_x(\beta_x) = \epsilon_x(\beta_x)$ for every $x \in \Supp(\beta)$ (see Proposition~\ref{prop:s-unital} and Proposition~\ref{prop:tominaga}). 
Thus, $$I \beta \ni s \beta = s \sum \beta_x  
= \sum s (\epsilon_x(\beta_x) \beta_x) 
= \sum (s \epsilon_x(\beta_x)) \beta_x = \sum \epsilon_x(\beta_x) \beta_x 
= \beta \ne 0,$$
where all sums run over $\Supp(\beta)$.
This shows that $(G,\Delta(G),S_e,\beta)$ is a Passman form.
\end{proof}

\begin{prop}[{cf.~\cite[Lem.~3.3(ii)]{passman1984infinite}}]\label{prop:passman1}
Suppose that $S$ is non-degenerately $G$-graded  
and that $A,B$ are nonzero ideals of $S$ such that $AB=\{0\}$. 
Let $(H,D,I,\beta)$ be a Passman form
for $(A,B)$. 
Then the following assertions hold:
 \begin{enumerate}[{\rm (a)}]
 \begin{item}
 $I \pi_{\Delta(H)}(A) \ne \{0\}$
 \end{item} 
 \begin{item} 
 There exists a Passman form $(H,D,I,\beta')$ for $(A,B)$ such that $I \pi_{\Delta(H)}(\beta') \ne \{0\}$, and hence $I\pi_{D}(\beta') \ne \{0\}$.
 Moreover, the size of $(H,D,I,\beta')$ is not greater than the size of $(H,D,I,\beta)$.
 \end{item}
 \end{enumerate}
\end{prop}
\begin{proof}
(a): Note that $IA \ne \{0\}$ is a right $S$-ideal. By Lemma~\ref{lem:bimodule_homo} and Lemma~\ref{lem:pi_non-degenerate}, we have $\{0\} \ne \pi_{\Delta(H)}(I A) = I \pi_{\Delta(H)}(A)$.

(b): We construct a Passman form with the required properties. 
Write $\beta = \sum_{x \in G} \beta_x$. 
By assumption, $I \beta \ne \{0\}$.  
Hence there is some $r \in I \subseteq S_e$ and $x \in G$ such that $r \beta_x \ne 0$. 
By non-degeneracy of the $G$-grading, we have $(r \beta_x) S_{x^{-1}} \ne \{0\}$, i.\,e. there is some $\sigma_{x^{-1}}\in S_{x^{-1}}$ such that $r \beta_x \sigma_{x^{-1}} \ne 0$. 
Thus, $I \beta_x \sigma_{x^{-1}} \ne \{0\}$.
Hence, $(H,D,I,\beta')$ with $\beta' := \beta \sigma_{x^{-1}}$ is a Passman form for $(A,B)$ such that $I \pi_{\Delta(H)}(\beta') \ne \{0\}.$ 
We now show that the size of $(H,D,I,\beta')$ is less than or equal to the size of $(H,D,I, \beta)$.
Suppose that $(H,D,I,\beta)$ has size $m$ and that $Dg_1,\dots,Dg_m$  
 form a minimal set of right $D$-cosets covering $\Supp(\beta)$. 
 Then $$\Supp(\beta \sigma_{x^{-1}}) \subseteq \Supp(\beta)x^{-1} \subseteq Dg_1x^{-1} \cup Dg_2x^{-1} \cup \ldots \cup Dg_m x^{-1}$$ and hence the $m$ right $D$-cosets $\{ D g_i x^{-1} \}_{i=1}^m$ cover $\Supp(\beta \sigma_{x^{-1}})$. 
 Thus, the size of $(H,D,I,\beta')$ is less than or equal to $m$. 
Finally, since $\Delta(H) \subseteq D$, we get $\{0\} \ne I \pi_{\Delta(H)}(\beta') \subseteq I \pi_D(\beta')$. 
\end{proof}

\begin{lem}\label{lem:annhilators}
Suppose that $S$ is nearly epsilon-strongly $G$-graded and that $S_e$ is 
$G$-semiprime. 
For any 
$G$-invariant ideal $I$ of $S_e$ the following assertions hold: 
\begin{enumerate}[{\rm (a)}]
\begin{item}
$r.\Ann_{S_e}(I) = r.\Ann_{S_e}(I^2).$
\end{item}
\begin{item}
$r.\Ann_{S}(I) = r.\Ann_{S}(I^2).$
\end{item}
\end{enumerate}
\end{lem}
\begin{proof}
(a): Put $J:=r.\Ann_{S_e}(I^2)$.
Clearly, $r.\Ann_{S_e}(I) \subseteq J$.
By Corollary~\ref{cor:multi1}, $(IJ)^x = I^x J^x$ for every $x \in G$, and hence $IJ$ is a 
$G$-invariant ideal of $S_e$ by Lemma~\ref{lem:misc1}. 
Moreover, by definition $I^2 J = \{0\}$, and hence $(IJ)^2 = (IJ) (IJ) \subseteq I (IJ) = I^2 J = \{0\}$. 
Since $S_e$ is $G$-semiprime, it follows that $IJ=\{0\}$. Thus, $J$ annihilates $I$, i.\,e. $J \subseteq r.\Ann_{S_e}(I).$ 

(b): Similarly, the inclusion $r.\Ann_S(I) \subseteq r.\Ann_S(I^2)$ is immediate. 
We now show the
reversed 
inclusion. 
Take $\gamma = \sum_{x \in G} \gamma_x \in r.\Ann_S(I^2).$ 
Since $I^2 \subseteq S_e$, we have $I^2 \gamma_x = \{0\}$ for every $x \in G$.
Next, let $x \in G$. Using~(a), we obtain that $\gamma_x S_{x^{-1}} \subseteq r.\Ann_{S_e}(I) $. In other words, $I \gamma_x S_{x^{-1}} = \{0\}$ which, by non-degeneracy of the $G$-grading, yields $I \gamma_x = \{0\}$, and hence $\gamma_x \in r.\Ann_S(I)$. Since $x \in G$ is arbitrary, it follows that $\gamma \in r.\Ann_S(I)$. 
\end{proof}

\begin{lem}[{cf.~\cite[Lem.~3.3(iii)]{passman1984infinite}}]\label{lem:66}
Suppose that $S$ is nearly epsilon-strongly $G$-graded and that $S_e$ is 
$G$-semiprime. 
Furthermore, let $A, B$ be nonzero ideals of $S$ such that $AB=\{0\}$. 
If $(H,D,I,\beta)$ is a Passman form for $(A,B)$ of minimal size with $I \pi_D(\beta) \ne \{0\}$, then for every $\gamma \in S_D$ we have $I \gamma \beta = \{0\}$ if and only if $I \gamma \pi_D(\beta) = \{0\}.$
\end{lem}
\begin{proof}
Suppose that $I \gamma \beta = \{0\}$. Since $\pi_D$ is an $S_D$-bimodule homomorphism by Lemma~\ref{lem:bimodule_homo}, it follows that $\{0\} = \pi_D(I\gamma \beta) = I \gamma \pi_D(\beta).$

Conversely, suppose that $I \gamma \pi_D(\beta)=\{0\}.$ 
Take $s \in I$ and note that $s \gamma \beta \in I S_D B \subseteq B$. Seeking a contradiction, suppose that $(H,D,I, s \gamma \beta)$ is a Passman form for the pair $(A,B)$. We show that $(H,D,I, s \gamma \beta)$ has less than minimal size. Indeed, suppose that $n \in \N$ is the size of $(H,D,I, \beta)$, i.\,e. the minimal number such that $\Supp(\beta) \subseteq \bigcup_{i=1}^n D g_i$ 
for some 
$g_1,\ldots,g_n \in G$.
Since $I \pi_D(\beta) \ne \{0\}$, we have $\pi_D(\beta) \ne 0$. Hence, we may w.l.o.g. assume that $g_1=e$. Moreover, it is immediate that
    \begin{align*}
        \Supp(s \gamma \beta) \subseteq \Supp(s \gamma) \Supp(\beta) \subseteq D \left( \bigcup_{i=1}^n D g_i \right) \subseteq \left( \bigcup_{i=1}^n D g_i \right).
    \end{align*}
By assumption, however, $0 = s \gamma \pi_D(\beta) = \pi_D(s \gamma \beta)$ which entails that $\Supp(s \gamma \beta) \subseteq \bigcup_{i=2}^n D g_i $. This is a contradiction since $(H,D, I, \beta)$ is assumed to be minimal. Thus, $(H,D,I, s \gamma \beta)$ is not a Passman form, and hence $I s \gamma \beta = \{0\}$ (cf. Definition~\ref{def:passmanform}). As this holds for every $s \in I$, we have $I^2 \gamma \beta = \{0\}$. By Lemma~\ref{lem:annhilators}(b), this yields $I \gamma \beta = \{0\}.$
\end{proof}

\subsection{Properties of a Passman form of minimal size}

In what follows, we fix a nearly epsilon-strongly $G$-graded ring
$S$ such that $S_e$ is $G$-semiprime, nonzero ideals $A,B$ of $S$ with $AB=\{0\}$, and a  
Passman form $(H,D,I,\beta)$ for $(A,B)$ of minimal size. 
Throughout this section we assume that
$I \pi_{\Delta(H)}(A) \cdot I \beta \ne \{0\}$.

\begin{lem}[{cf.~\cite[Lem.~3.4]{passman1984infinite}}]\label{lem:finite_index}
The following assertions hold:
\begin{enumerate}[{\rm (a)}]

\item There exists $\alpha \in A \cap S_H$ such that $I \pi_D(\alpha) \beta \ne \{0\}$. 

\item For every $\alpha \in A$ there is a subgroup $W$ of $H$ of finite index that centralizes  $\Supp(\pi_D(\alpha))$ and $\Supp(\pi_D(\beta))$. 

\end{enumerate}
\end{lem}
\begin{proof}
(a): By assumption, we have $I \pi_{\Delta(H)}(A) \cdot I \beta \ne \{0\}$. In other words, $\pi_{\Delta(H)}(A) \cdot I \beta$ is not contained in $r.\Ann_S(I)$. 
Furthermore, by Lemma~\ref{lem:annhilators}(b), we have $r.\Ann_{S}(I) = r.\Ann_{S}(I^2)$. 
Applying Lemma~\ref{lem:bimodule_homo}, we get $ I \pi_{\Delta(H)}(IAI) \beta = I^2 \pi_{\Delta(H)}(A) \cdot I \beta \ne \{0\}$. Hence, there exists some $\alpha \in I A I \subseteq A$ such that $I \pi_{\Delta(H)}(\alpha) \beta \ne \{0\}$. 
Additionally, we have  $\alpha \in I S I \subseteq S_H$ by Lemma~\ref{lem:1}(b). 
Since $D \cap H = \Delta(H)$, we get $\pi_D(\alpha) = \pi_{\Delta(H)}(\alpha)$ and thus $I \pi_D(\alpha) \beta \ne \{0\}$. 

(b): Note that $ P := \Supp(\pi_D(\alpha)) \cup \Supp(\pi_D(\beta))$ is a finite subset of $D=D_{G}(H)$ and consider  $W := \bigcap_{x \in P} C_H(x)$. Since $P \subseteq D$ and $[H : C_H(x)] < \infty$ for every $x \in D$,
we get that
$[H : W ] < \infty$. 
\end{proof}

\begin{lem}[{cf.~\cite[Lem.~3.4]{passman1984infinite}}]\label{lem:zero}
Suppose that $\alpha \in A \cap S_H$ is such that $I \pi_D(\alpha) \beta \ne \{0\}$. 
Let $W$ be given by Lemma~\ref{lem:finite_index}.
Then there are $d_0 \in \Supp(\pi_D(\alpha))$ and $u \in W$ such that $I (S_e \alpha_{d_0} S_{d_0^{-1}})^u \pi_D(\alpha) \beta \ne \{0\}.$ 
\end{lem}

\begin{proof}
First put $\gamma := \pi_D(\alpha) \beta$ and write $\alpha = \sum_{x \in G} \alpha_x$, $\beta = \sum_{x \in G} \beta_x$, and $\gamma = \sum_{x \in G} \gamma_x$.
Furthermore, let $J := \sum_{d \in D} (S_e \alpha_d S_{d^{-1}})^W \subseteq S_e$.
Note that $J$ is a $W$-invariant ideal of $S_e$.
Using that $S$ is nearly-espilon strongly $G$-graded, note that for all $d \in D, y \in G$, we have 
\begin{equation}\label{eq:715}
\alpha_d \beta_y S_{y^{-1} d^{-1}} \subseteq S_e \alpha_d  (S_{d^{-1}} S_d) \beta_y S_{y^{-1} d^{-1}} = S_e \alpha_d S_{d^{-1}} \cdot S_d \beta_y S_{y^{-1} d^{-1}} \subseteq J \cdot S_e \subseteq J.
\end{equation}
Take $x \in G$.
Then by \eqref{eq:715}, $\gamma_x S_{x^{-1}} \subseteq J$. Seeking a contradiction, suppose that $IJ \gamma = \{0\}$. 
Then  $I J \gamma_x S_{x^{-1}} = \{0\}$. 
Hence $\gamma_x S_{x^{-1}} \subseteq r.\Ann_{S_e} (IJ)$. Using that $IJI \subseteq IJ$, we get
\begin{equation}
    I \gamma_x S_{x^{-1}} \subseteq IJ \cap r.\Ann_{S_e}(IJ).
    \label{eq:22}
\end{equation}
By Lemma~\ref{lem:misc1}, we know that $IJ \cap r.\Ann_{S_e}(IJ)$ is a 
$W$-invariant nilpotent ideal of $S_e$ contained in $I$, and hence $IJ \cap r.\Ann_{S_e}(IJ)= \{0\}$ by Lemma~\ref{lem:18}(b).
By non-degeneracy of the $G$-grading, 
\eqref{eq:22} implies that $I \gamma_x = \{0\}$. Since $x \in G$ is arbitrary, this yields $I \gamma = \{0\}$, i.\,e. $I \pi_D(\alpha) \beta = \{0\}$. This contradicts the properties  
of $\alpha$. Consequently, $IJ \pi_D(\alpha) \beta \ne \{0\}$,
i.\,e. there exist some $d_0 \in D$ and some $u \in W$ such that $I (S_e \alpha_{d_0} S_{d_0^{-1}})^u \pi_D(\alpha) \beta \ne \{0\}$.
\end{proof}

For the remainder of this section,
we fix $\alpha \in A \cap S_H$ such that $| \Supp(\pi_D(\alpha))|$ is minimal subject to $I \pi_D(\alpha) \beta \ne \{0\}$. 
We also fix $W$ given by Lemma~\ref{lem:finite_index}.

\begin{lem}[{cf.~\cite[Lem.~3.4]{passman1984infinite}}]\label{lem:exchange}
For every $y \in W$ and every $d \in D$, we have $$ I S_{y^{-1}} \alpha_d S_{d^{-1} y} \pi_D(\alpha) \pi_D(\beta) = I S_{y^{-1}} \pi_D(\alpha) S_{d^{-1} y} \alpha_d \pi_D(\beta).$$ 
\end{lem}
\begin{proof}
Take $y \in W$, $d \in D$, $a_{y^{-1}} \in S_{y^{-1}}$, and $b_{d^{-1} y} \in S_{d^{-1} y }$.
Note that if $d \notin \Supp(\alpha)$,
then the claim trivially holds.
Therefore, we now suppose that 
$d \in D \cap \Supp(\alpha)$.
Define $$ \gamma := a_{y^{-1}} \alpha_d b_{d^{-1} y} \alpha - a_{y^{-1}} \alpha b_{d^{-1} y } \alpha_d.$$ A short computation, using Lemma~\ref{lem:finite_index}(b), shows that $\gamma \in A \cap S_H$. Moreover, since $\pi_D$ is an $S_e$-bimodule homomorphism by Lemma~\ref{lem:bimodule_homo}, we get $$ \pi_D(\gamma) = a_{y^{-1}} \alpha_d b_{d^{-1} y} \pi_D(\alpha) - a_{y^{-1}} \pi_D(\alpha) b_{d^{-1} y} \alpha_d. $$ 
From this we get that $\Supp(\pi_D(\gamma)) \subseteq \Supp(\pi_D(\alpha))$.
We claim that the minimality assumption on $\alpha$ implies that $I \pi_D(\gamma) \pi_D(\beta) = \{0\}$. If the claim holds,
then we get that
$$ I a_{y^{-1}} (\alpha_d b_{d^{-1}y } \pi_D(\alpha) - \pi_D(\alpha) b_{d^{-1} y} \alpha_d) \pi_D(\beta) = \{0\}$$ and hence that $$I S_{y^{-1}} \alpha_d S_{d^{-1} y} \pi_D(\alpha) \pi_D(\beta) = I S_{y^{-1}} \pi_D(\alpha) S_{d^{-1} y} \alpha_d \pi_D(\beta).$$ 
Now we show the claim.
Write
$\gamma=\sum_{x \in G}\gamma_x$. 
By considering the cases when $x\in \Supp(\alpha)$ and $x\notin \Supp(\alpha)$ separately, for each $x \in G$ we get that
\begin{equation}
    \gamma_x = a_{y^{-1}} \alpha_d b_{d^{-1} y} \alpha_x - a_{y^{-1}} \alpha_x b_{d^{-1} y} \alpha_d.
    \label{eq:6}
\end{equation} 
Now, recall that $\pi_D(\gamma) = \sum_{x \in D} \gamma_x$.  However, due to \eqref{eq:6}, $\gamma_d = 0$, and thus, $| \Supp(\pi_D(\gamma)) | < | \Supp(\pi_D(\alpha)) |$, since $\alpha_d \neq 0$.
The minimality assumption on $\alpha$ therefore implies that $I \pi_D(\gamma) \beta = \{0\}$. Applying the map $\pi_D$ to the former equation yields $I \pi_D(\gamma) \pi_D(\beta) = \{0\}.$
\end{proof}

\begin{lem}\label{lem:delta1}
There are elements $x_1,\ldots, x_n \in W$ and
$g_1,\ldots,g_n \in \Supp(\beta) \setminus D$
such that if
$S_{y^{-1}} I \pi_D(\alpha) S_y \pi_D(\beta) \ne \{0\}$, then
$y \in \bigcup_{k=1}^n x_k H_k$
whenever $y \in W$.
Here, $H_k := C_W(g_k)$.
\end{lem}
\begin{proof}
Let $ \tilde \alpha := \alpha - \pi_D(\alpha)$ and let $\tilde \beta := \beta - \pi_D(\beta) $. Then
\begin{equation*}
S_{y^{-1}} I(\pi_D(\alpha) + \tilde \alpha) S_y (\pi_D(\beta) + \tilde \beta) = S_{y^{-1}} I\alpha S_y \beta \subseteq S_{y^{-1}}I A S_y B \subseteq A B = \{0\}.
\end{equation*}
Note that
$S_{y^{-1}} I\pi_D(\alpha) S_y \tilde \beta$ and $S_{y^{-1}}I \tilde \alpha S_y \pi_D(\beta)$ 
have support
disjoint
from $D$. 
On the other hand, 
$ \{0\} \ne S_{y^{-1}} I \pi_D(\alpha) S_y \pi_D(\beta)  \subseteq S_D. $ 
Hence, $S_{y^{-1}}I \pi_D(\alpha) S_y \pi_D(\beta) $ must be additively cancelled out by $S_{y^{-1}}  I\tilde \alpha S_y \tilde \beta$. In particular, these two expressions must have a support element in common, i.\,e. there exist $a \in \Supp(\tilde \alpha), b \in \Supp(\tilde \beta), g \in \Supp( \pi_D(\alpha))$, and $f \in \Supp (\pi_D(\beta))$ such that $y^{-1}ay b = y^{-1} g y f$.  
Multiplying with $y$ from the left and with $y^{-1}$ from the right gives $a y b y^{-1} = g y f y^{-1} = gf$, where we have used the fact that $y\in W$ commutes with both $\Supp(\pi_D(\alpha))$ and $\Supp(\pi_D(\beta))$. Consequently, $y b y^{-1} = a^{-1} gf$, and hence $y \in x C_W(b)$ for some fixed $x$ depending on $a,b,g,f$. Since there are only 
a finite number of choices for the parameters $a,b,g,f$ and $b \in \Supp(\tilde \beta) = \Supp (\beta) \setminus D$, the 
desired conclusion follows.
\end{proof}

Next, we will construct an ideal $J$ of $S_e$ that allows us to apply the Passman replacement argument (see Section~\ref{Sec:PassmanPairs}). 
In the following two lemmas we make use 
of the notation introduced in
Lemma~\ref{lem:delta1}.
\begin{lem}[{cf.~\cite[Lem.~3.5]{passman1984infinite}}]\label{lem:prod1}
For every $d \in D$, 
$$ I (S_e \alpha_d S_{d^{-1}})^y \cdot \pi_D(\alpha)\beta = \{ 0 \}, \qquad \forall y \in W \setminus \bigcup_{k=1}^n x_k H_k.$$ 
\end{lem}
\begin{proof}
Take $y \in W$ such
that $I(S_e \alpha_d S_{d^{-1}})^y \cdot \pi_D(\alpha)\beta \ne \{0\}$. Expanding this expression, 
we get 
$I S_{y^{-1}}  \alpha_d S_{d^{-1}} S_y \pi_D(\alpha) \beta \ne \{ 0 \}.$
Since $S_{y^{-1}} \alpha_d S_{d^{-1}} S_y \pi_D(\alpha) \subseteq S_D$, Lemma~\ref{lem:66} implies that
$I S_{y^{-1}} \alpha_d S_{d^{-1}} S_y \pi_D(\alpha) \pi_D(\beta) \ne \{0\}.$ As a consequence, $I S_{y^{-1}} \alpha_d S_{d^{-1}y} \pi_D(\alpha) \pi_D(\beta) \ne \{ 0 \}$, since $S_{d^{-1}} S_y \subseteq S_{d^{-1} y}$. 
By Lemma~\ref{lem:exchange}, we get that $ I S_{y^{-1}} \pi_D(\alpha) S_{d^{-1} y} \alpha_d \pi_D(\beta) \ne \{0\} $ and, due to $d^{-1}yd = y$, we even have $I S_{y^{-1}} \pi_D(\alpha) S_y \pi_D(\beta) \ne \{ 0 \}$.
Next, note that $I$ is also a  
$W$-invariant ideal, and thus $I S_{y^{-1}} = S_{y^{-1}} I$ by Proposition~\ref{prop:switch1}(b). It follows that $$ \{0\} \ne I S_{y^{-1}} \pi_D(\alpha) S_y \pi_D(\beta) =  S_{y^{-1}} I \pi_D(\alpha) S_y \pi_D(\beta)$$ 
which, combined with Lemma~\ref{lem:delta1}, yields the desired conclusion.
\end{proof}

\begin{lem}\label{lem:ideal_exist}
There exists an ideal $J$ of $S_e$ such that $J^y J = \{0\}$ for every 
$y \in W \setminus \bigcup_{k=1}^n u^{-1} x_k H_k$.
\end{lem}
\begin{proof}
Set $\gamma := \pi_D(\alpha) \beta$ and
write  
$ \gamma = \sum_{x \in G} \gamma_x$. 
By Lemma~\ref{lem:zero} there exist $d_0 \in D$ and $u \in W$ such that $I (S_e \alpha_{d_0} S_{d_0^{-1}})^u \gamma \ne \{0\}$. Hence, there exists $x \in G$ such that $I (S_e \alpha_{d_0} S_{d_0^{-1}})^u \gamma_x \ne \{0\}$. 
By non-degeneracy of the $G$-grading, 
we have that $J := I (S_e \alpha_{d_0} S_{d_0^{-1}})^u \gamma_x S_{x^{-1}} \ne \{0\}$ is an ideal of $S_e$ contained in $I$. 
Recall that $I$ is 
$W$-invariant, since $W$ is a subgroup of $H$. Now, combining the fact that $J \subseteq I(S_e \alpha_{d_0} S_{d_0^{-1}})^u S_e = I(S_e \alpha_{d_0} S_{d_0^{-1}})^u$ with Lemma~\ref{lem:2}, for every $y \in W$ we get
\begin{equation*}
    J^{u^{-1} y} \subseteq I^{u^{-1} y}((S_e \alpha_{d_0} S_{d_0^{-1}})^u)^{u^{-1} y} \subseteq I^{u^{-1} y}(S_e \alpha_{d_0} S_{d_0^{-1}})^y \subseteq I (S_e \alpha_{d_0} S_{d_0^{-1}})^y.
\end{equation*}  
By Lemma~\ref{lem:prod1}, it follows that $J^{u^{-1} y} \pi_D(\alpha) \beta = \{0\}$ for every $y \in W \setminus \bigcup_{k=1}^n x_k H_k$ or, equivalently, that $J^y \gamma = \{0\}$ for every $y \in W \setminus \bigcup_{k=1}^n u^{-1} x_k H_k$. 
In particular, we have $J^y \gamma_x = \{0\}$, and hence, $J^y (S_e \gamma_x S_{x^{-1}}) = \{0\}$. This shows that $J^y J = \{0\}$ for every $y \in W \setminus \bigcup_{k=1}^n u^{-1} x_k H_k$.
\end{proof}

\subsection{Establishing Proposition~\ref{prop:main1}}

We still
assume that $I \pi_{\Delta(H)}(A) \cdot I \beta \ne \{0\}$. 
Combining that assumption with the following lemma, we will establish
Proposition~\ref{prop:main1}. 

\begin{lem}
\label{lem:contradiction}
There is a Passman form for $(A,B)$ of size smaller than the size of $(H,D,I, \beta)$.
\end{lem}
\begin{proof}
By Lemma~\ref{lem:finite_index}, we have $[H : W] < \infty$. Let $J$ be the ideal of $S_e$ from Lemma~\ref{lem:ideal_exist}. 
By Proposition~\ref{prop:coset2} there exists a subgroup $L$ of $H$ and a nonzero ideal $K \subseteq J$ of $S_e$ such that $K^y K = \{0\}$ for every $y \in H \setminus L$. 
Furthermore, we have $[L \colon L \cap H_k] < \infty$ for some subgroup $H_k$ of $W$. 
We claim that $(L, D_G(L), K^L, \pi_D(\alpha) \beta)$ is a Passman form of size smaller than the size of $(H,D,I, \beta)$. 
We first check that it satisfies the conditions in Definition~\ref{def:passmanform}.

Note that condition~(a) is trivially satisfied. 
Moreover, it follows from Lemma~\ref{lem:power} that $K^L$ is an 
$L$-invariant ideal of $S_e$. 
Since $K \subseteq I$, we have $K^x K = \{0\}$ for every $x \in G \setminus H$. 
This shows that $K^x K = \{0\}$ for every $x \in G \setminus L$. 
Thus, by Proposition~\ref{prop:power}, $(K^L)^x (K^L) = \{0\}$ for every $x \in G \setminus L$. 
Hence, condition
~(b) is satisfied. 
Next, note that $\gamma := \pi_D(\alpha) \beta \in B$. 
It remains to show that $K^L \gamma \ne \{0\}$ and $K^L A \ne \{0\}$. 
Seeking a contradiction, suppose that $K^L \gamma = \{0\}$. 
Then $K^L \gamma_x = \{0\}$, and hence $K^L ( S_e \gamma_x S_{x^{-1}}) = \{0\}$.
We get that $K^L J = \{0\}$. This implies that $J \subseteq r.\Ann_{S_e}(K^L)$. 
But since $r.\Ann_{S_e}(K^L)$ is an 
$L$-invariant ideal by Lemma~\ref{lem:misc1}, we deduce from Lemma~\ref{lem:power} that $J \subseteq J^L \subseteq r.\Ann_{S_e}(K^L)$ and hence that $K^L J^L = \{0\}$. 
As $K \subseteq J$, this yields $(K^L)^2 = \{0\}$, which is a contradiction by Lemma~\ref{lem:18}(a). 
Therefore, $K^L \pi_D(\alpha) \beta \ne \{0\}$. It follows that $K^L \pi_D(\alpha) \ne \{0\}$, and hence $K^L A \ne \{0\}$, by Lemma~\ref{lem:bimodule_homo}.
Summarizing, we have shown that $(L, D_G(L), K^L, \pi_D(\alpha) \beta)$ is a Passman form.

To proceed, let $n$ be the size of the Passman form $(H,D,I, \beta)$, i.\,e. the number of cosets of $D$ in $H$ meeting $\Supp(\beta)$. Furthermore, let $m$ denote the size of $(L, D_G(L), K^L, \pi_D(\alpha) \beta)$. 
We claim that $m < n$. 
To show this, first note that $D = D_G(H) \subseteq D_G(L)$ and that $\Supp(\pi_D(\alpha) \beta) \subseteq D \cdot \Supp(\beta)$. Hence, $ m \leq n$. 
Combining the facts that $[ L : L \cap H_k] < \infty$ for some $H_k = C_W(g)$ with $g \in \Supp(\beta) \setminus D$ and $L \cap H_k = L \cap C_W(g) = C_L(g)$, we infer that $[L : C_L(g) ] < \infty$ and hence that $g \in D_G(L)$. 
This means that the two distinct $D$-cosets $Dg$ and $D$ are contained in $D_G(L)$. 
Consequently, $ m < n$, as claimed.
\end{proof}

We are now fully prepared to prove the following:

\begin{proof}[Proof of Proposition~\ref{prop:main1}]
Let $S$ be nearly epsilon-strongly $G$-graded such that $S_e$ is $G$-semi\-prime.
Furthermore, let $A,B$ be nonzero ideals of $S$ such that $AB=\{0\}$. 
We now show that conditions (a)-(c) in Proposition~\ref{prop:main1} are satisfied.
By Proposition~\ref{prop:passman_exists}, $S$ admits a minimal Passman form for $(A,B)$, 
say $(H,D,I, \beta)$. 
Moreover, by Propositon~\ref{prop:passman1}, we may assume that $I \pi_{\Delta(H)}(A) \ne \{0\}$ and that $I \pi_{\Delta(H)}(\beta) \ne \{0\}$. 
Hence, conditions (a) and (b) hold.
Seeking a contradiction, suppose that $ I \pi_{\Delta(H)}(A) \cdot I \beta \ne \{0\}$. Then the previous results, in particular Lemma~\ref{lem:contradiction}, yields a Passman form of size smaller than that of $(H,D,I,\beta)$, which is the desired contradiction. 
Hence, $I \pi_{\Delta(H)}(A) \cdot I \beta = \{0\}$ which shows that condition (c) holds.
\end{proof}

\section{The ``hard'' direction}
\label{Sec:HardRight}

Recall that $S$ is a $G$-graded ring.
In this section, we prove the implication (a)$\Rightarrow$(e) of Theorem~\ref{thm:mainNew} for nearly epsilon-strongly $G$-graded rings
(see Proposition~\ref{prop:suff1}). 
We remind the reader that if $H,K$ are subgroups of $G$, then $H$ \emph{normalizes} $K$ if $K x = x K$ for every $x \in H$. 
In that case it follows that $H \subseteq N_G(K)$, where $N_G(K):=\{ x \in G \mid x K = K x \}$ denotes the normalizer of $K$ in $G$,
and we allow ourselves to speak of \emph{$H/K$-invariance} in the sense of Definition~\ref{def:weakly_invariant}.

\begin{lem}\label{lem:normalpi}
Suppose that $H,K$ are subgroups of $G$ such that $H$ normalizes $K$.
If $x \in H$, $k_1,k_2 \in K$, $r \in S_{x k_1}$, $\alpha \in S$ and $s \in S_{k_2 x^{-1}}$, then $\pi_K( r \alpha s ) = r \pi_K( \alpha ) s$.
\end{lem}

\begin{proof}
Write $\alpha = \sum_{y \in G} \alpha_y$, where $\alpha_y \in S_y$ for $y \in G$.
Take $y \in G$. Note that $x k_1 \cdot y \cdot k_2 x^{-1} \in K$ if and only if $y \in k_1^{-1} x^{-1} K x k_2^{-1} = k_1^{-1} K k_2^{-1} = K$. Thus, 
$\pi_K( r \alpha s ) = \sum_{y \in G} 
\pi_K( r \alpha_y s ) = 
\sum_{y \in K} \pi_K( r \alpha_y s ) = 
\sum_{y \in K} r \alpha_y s = 
r \pi_K(\alpha) s$.
\end{proof}

By Lemma~\ref{lem:normalpi}, with $k_1 = k_2 = e$, we get the following result
(cf.~\cite[p.~721]{passman1984infinite}).

\begin{cor}\label{cor:normalized}
Suppose that $H,K$ are subgroups of $G$ such that $H$ normalizes $K$.
For every $\alpha \in S$ and $x \in H$, we have $ S_{x^{-1}} \pi_K(\alpha) S_x = \pi_K(S_{x^{-1}} \alpha S_x).$
\end{cor}

Given ideals $A, B$ of $S$ such that $AB= \{0 \}$, we will find new ideals $A_1, A_2, B_1, B_2$ of subrings of $S$ satisfying $A_1 B_1 = \{ 0 \} $ and $A_2 B_2 = \{ 0 \}$. 

\begin{lem}\label{lem:not_prime}
Suppose that $S$ is nearly epsilon-strongly $G$-graded and that $S_e$ is  
$G$-semiprime. If $S$ is not prime, then there exists a subgroup $H$ of $G$ such that $S_{\Delta(H)}$ is not prime.
In fact, there exist nonzero
$H/\Delta(H)$-invariant
ideals $A_1,B_1$ of $S_{\Delta(H)}$ such that $A_1, B_1 \subseteq I S_{\Delta(H)}$ and $A_1 B_1 = \{0\}$.
\end{lem}
\begin{proof}
Let $A,B$ be nonzero ideals of $S$ such that $AB= \{0\}$. 
By Proposition~\ref{prop:main1}, there are a subgroup $H$ of $G$, a nonzero  
$H$-invariant ideal $I$ of $S_e$, and $\beta \in B$ such that:
\begin{enumerate}[(a)]
\begin{item}
$I \pi_{\Delta(H)}(A) \ne \{0\}$;
\end{item}
\begin{item}
$I \pi_{\Delta(H)}(\beta) \ne \{0\}$;
\end{item}
\begin{item}
$I \pi_{\Delta(H)}(A) \cdot I \beta = \{0\}$.
\end{item}
\end{enumerate}
Consider the set $A_1 := I \pi_{\Delta(H)}(A) \subseteq I S_{\Delta(H)}$.
Clearly, $A_1$ is nonzero by (a).
Take $h\in H$.
Then Proposition~\ref{prop:switch1}, the fact that $\Delta(H)\subseteq H$, and
Lemma~\ref{lem:normalpi}
yield
\begin{align*}
S_{h^{-1}\Delta(H)} A_1 S_{h\Delta(H)}
&= S_{h^{-1}\Delta(H)} I \pi_{\Delta(H)}(A) S_{h\Delta(H)}
= I S_{h^{-1}\Delta(H)} \pi_{\Delta(H)}(A) S_{h\Delta(H)}
\\
&= I \pi_{\Delta(H)} ( S_{h^{-1}\Delta(H)} A S_{h\Delta(H)} )
\subseteq
I \pi_{\Delta(H)}(A) = A_1.
\end{align*}
By taking $h=e$, the above computation yields $S_{\Delta(H)}A_1 S_{\Delta(H)} \subseteq A_1$.
Thus, $A_1$ is an   
$H/\Delta(H)$-invariant ideal of $S_{\Delta(H)}$.
Next, we define
$B_1 := \sum_{h \in H} I S_{h^{-1} \Delta(H)} \pi_{\Delta(H)}(\beta) S_{h\Delta(H)}.$
Clearly, $B_1 \subseteq I S_{\Delta(H)}$ and $B_1$ is nonzero. 
 Take $h_1 \in H$. 
 Using that $\Delta(H)$ is a normal subgroup of $H$ and that $I$ is $H$-invariant, we get 
\begin{align*}
 S_{h_1^{-1} \Delta(H)} B_1 S_{h_1 \Delta(H)} &= 
  \sum_{h \in H} S_{h_1^{-1} \Delta(H)} \cdot I S_{h^{-1}\Delta(H)} \pi_{\Delta(H)}(\beta) S_{h\Delta(H)} \cdot S_{h_1 \Delta(H)} \\
  &=
  \sum_{h \in H} I S_{h_1^{-1}\Delta(H)}  S_{h^{-1}\Delta(H)} \cdot \pi_{\Delta(H)}(\beta) \cdot S_{h\Delta(H)} S_{h_1 \Delta(H)} \\
 &\subseteq 
 \sum_{h \in H} I S_{h_1^{-1} h^{-1}\Delta(H)} \cdot \pi_{\Delta(H)}(\beta) \cdot S_{h h_1 \Delta(H)} =
 B_1.
\end{align*}
By taking $h_1=e$, the above computation yields $S_{\Delta(H)}B_1 S_{\Delta(H)} \subseteq B_1$.
Thus, $B_1$ is an  
$H/\Delta(H)$-invariant ideal of $S_{\Delta(H)}$.
By Proposition~\ref{prop:lannstrom1},
the induced $H/\Delta(H)$-grading on $S_H$ is nearly epsilon-strong and hence
$S_{h^{-1}\Delta(H)} S_{h\Delta(H)}
\cdot \pi_{\Delta(H)}(A) = 
\pi_{\Delta(H)}(A)$.
Using that $I$ is $H$-invariant, it follows from Lemma \ref{lem:bimodule_homo}, Lemma \ref{lem:normalpi}, and (c), that
\begin{align*}
 A_1 B_1 &= I \pi_{\Delta(H)}(A) \cdot \sum_{h \in H} I S_{h^{-1} \Delta(H)} \pi_{\Delta(H)}(\beta) S_{h\Delta(H)} \\
 &=
 \sum_{h \in H}I \pi_{\Delta(H)}(A)   S_{h^{-1}\Delta(H)} \pi_{\Delta(H)}(I\beta) S_{h\Delta(H)} \\
&=  \sum_{h \in H}I
\cdot S_{h^{-1}\Delta(H)} S_{h\Delta(H)}
\pi_{\Delta(H)}(A) \cdot S_{h^{-1}\Delta(H)}  \pi_{\Delta(H)}(I\beta) S_{h\Delta(H)} \\
&= \sum_{h \in H}I
S_{h^{-1}\Delta(H)} 
\cdot \pi_{\Delta(H)}( S_{h\Delta(H)} A S_{h^{-1} \Delta(H)}) \cdot \pi_{\Delta(H)}(I\beta) S_{h\Delta(H)} \\
&\subseteq
 \sum_{h \in H}I
S_{h^{-1}\Delta(H)} 
\cdot \pi_{\Delta(H)} ( A ) \cdot \pi_{\Delta(H)}(I\beta) S_{h\Delta(H)} \\
&= \sum_{h \in H}
S_{h^{-1} \Delta(H)} \cdot I \pi_{\Delta(H)} ( A )  \cdot \pi_{\Delta(H)}(I\beta) S_{h\Delta(H)} \\
&=
\sum_{h \in H}
S_{h^{-1}\Delta(H)} \cdot \pi_{\Delta(H)} \big( I \pi_{\Delta(H)} ( A )  I\beta \big) \cdot S_{h\Delta(H)} = \{ 0 \}.
\end{align*}
As a result, $S_{\Delta(H)}$ is not prime.
\end{proof}

\begin{lem}
\label{lem:8_1}
Suppose that we are in the setting of Lemma~\ref{lem:not_prime}. Then there exists a finitely generated normal subgroup $W$ of $H$ such that $W \subseteq \Delta(H)$. Moreover, there exist nonzero $H/W$-invariant ideals $A_2, B_2 $ of $S_W$ such that $A_2, B_2 \subseteq IS_W$ and $A_2 B_2=\{0\}$.
\end{lem}

\begin{proof}
Let $A_1, B_1 \subseteq I S_{\Delta(H)}$ be as in Lemma~\ref{lem:not_prime}. Then there exist nonzero elements $a_1 \in A_1$ and $b_1 \in B_1$. Putting $P := \Supp (a_1) \cup \Supp (b_1)$ and using that $A_1, B_1$ are ideals of $S_{\Delta(H)}$, we see that $P \subseteq \Delta(H)$. 
Moreover, let $W$ be the normal closure of $P$ in $H$. Then $W$ is clearly a finitely generated normal subgroup of $H$ with $W \subseteq \Delta(H)$. Now, consider $A_2 := A_1 \cap S_W$ and $B_2 := B_1 \cap S_W$. 
Using that $A_1$ (resp. $B_1$) is $H/\Delta(H)$-invariant, we get that $A_1$ (resp. $B_1$) is $H/W$-invariant.
Clearly, $S_W$ is $H/W$-invariant.
Thus, $A_2$ and $B_2$ are are nonzero
$H/W$-invariant ideals of $S_W$ such that $A_2, B_2 \subseteq IS_W$.  Furthermore, we have $A_2 B_2 \subseteq A_1 B_1 = \{0\}$, which completes the proof.
\end{proof}

We recall the following general result regarding the finite conjugate center $\Delta(H)$ of an arbitrary group $H$ and include parts of the proof for the convenience of the reader. 

\begin{prop}[{\cite[Lem.~II.4.1.5(iii)]{passman2011algebraic}}]
\label{prop:8_2}
Suppose that $H$ is a group and that $W$ is a finitely generated subgroup of $\Delta(H)$. Then there exists a finite characteristic subgroup $N \lhd W$ such that $W/N$ is torsion-free abelian.
\end{prop}
\begin{proof}
Put $N := \{ w \in W \mid \mbox{ord}(w) < \infty \}$. 
It can be shown that the commutator subgroup $W' = [W,W]$ is finite (see \cite[Lem.~II.4.1.5(ii)]{passman2011algebraic}). Thus $W' \subseteq N$. Moreover, note that $W/W'$ is a finitely generated abelian group. By the fundamental theorem of finitely generated abelian groups, $W/W'$ has a finite maximal torsion subgroup $K$, i.\,e. $W/W' \cong \mathbb{Z}^n \oplus K$ for some $n \geq 0$. By restricting to torsion elements, we see that $N/W' \cong K$. Thus, $N$ is a finite subgroup of $W$. Since every automorphism of $W$ preserves element order, it follows that $N$ is a characteristic subgroup of $W$. We also get that $W/N$ is torsion-free abelian, because $W' \subseteq N$.
\end{proof}

\begin{defi}[cf.~{\cite[p.~14]{passman1984infinite}}]
Suppose that $A$ is a nonzero ideal of a nearly epsilon-strongly $W$-graded ring $S_W$ and that $N \lhd W$. 
For any nonzero $a \in A$ we define
$\text{meet}_N(a)$ to be the number of cosets of $N$ in $W$ that meet $\Supp(a)$.
Define $m := {\rm min} \{ \text{meet}_N(b) \mid b \in A \setminus \{0 \} \}$.
Let $\text{min}_N(A)$ denote  
the additive span of all nonzero elements $a \in A$ such that $\text{meet}_N(a) = m$.
\end{defi}

\begin{lem}[{cf.~\cite[Lem.~4.1]{passman1984infinite}}]
\label{lem:8_3}
Suppose that $S$ is nearly epsilon-strongly $G$-graded and that $H$ is a subgroup of $G$. Furthermore, suppose that $N \lhd W$ are subgroups of $G$ that are normalized by $H$ and that $A$ is a nonzero  
$H/W$-invariant ideal of $S_W$. 
Then the following assertions hold:
\begin{enumerate}[{\rm (a)}]
\item
 $\text{min}_N(A)$ is a nonzero  
$H/W$-invariant ideal of $S_W$.
\item
$\pi_N(A)$ is a nonzero  
$H/W$-invariant ideal of $S_N$.
\end{enumerate}
\end{lem}
\begin{proof}
(a): Note that $\text{min}_N(A)$ is nonzero by definition. We show that $\text{min}_N(A)$ is an ideal of $S_W$. Let $\alpha \neq 0$ be a generator of $\text{min}_N(A)$ and take $w \in W$. It is enough to show that $S_w \alpha$ and $\alpha S_w$ are contained in $\text{min}_N(A)$. To this end, note that $\Supp(\alpha S_w) \subseteq (\Supp(\alpha))w$ and that $\Supp(S_w \alpha) \subseteq w (\Supp(\alpha)).$ Since $N \lhd W$, right and left cosets of $N$ in $W$ coincide. Let $\{ w_1 N, w_2 N, \dots, w_mN \}$ be a minimal set of cosets of $N$ that covers $\Supp(\alpha)$. That is, $\Supp(\alpha) \subseteq w_1N \cup \ldots \cup w_m N = Nw_1 \cup \ldots \cup N w_m$ with $m$ minimal among such covers. Hence, $\Supp(\alpha S_w) \subseteq Nw_1 w \cup \ldots \cup Nw_m w$ and $\Supp(S_w \alpha) \subseteq w w_1 N \cup \ldots \cup w w_m N$, and consequently, $\alpha S_w$ and $S_w \alpha$ meet less than or exactly $m$ cosets of $N$. It follows that $\alpha S_w, S_w \alpha \in \text{min}_N(A)$ and therefore $\text{min}_N(A)$ is an ideal of $S_W$. 

Next, let $\alpha \in A$ be a generator of $\text{min}_N(A)$ and take $h \in H$. To show that $\text{min}_N(A)$ is  
$H/W$-invariant, it is enough to show that $S_{h^{-1}W} \alpha S_{hW} \subseteq \text{min}_N(A)$. 
Take $k_1,k_2 \in W$.
We will show that
$S_{h^{-1}k_1} \alpha S_{hk_2} \subseteq \text{min}_N(A)$

Using that $A$ is assumed to be  
$H/W$-invariant, we have $S_{h^{-1}k_1} \alpha S_{hk_2} \subseteq A$. Hence, it only remains to show that $S_{h^{-1}k_1} \alpha S_{hk_2}$ meets a minimal number of cosets of $N$. 
As before, let $w_1 N \cup \ldots \cup w_m N$ be a minimal cover of $\Supp(\alpha)$. 
Then
\begin{align*}
\Supp(S_{h^{-1}k_1} \alpha S_{hk_2}) 
&\subseteq h^{-1} k_1 (\Supp(\alpha)) h k_2 
\\
&\subseteq h^{-1} k_1 (w_1 N) h k_2 \cup h^{-1} k_1 (w_2 N) h k_2 \cup \ldots \cup h^{-1}  k_1 (w_m N ) h k_2.
\end{align*}
Since both $H$ and $W$ normalize $N$, we get that $h^{-1} k_1 (w_i N) h k_2 = 
 (h^{-1} k_1  w_i h k_2) N$.
  Moreover, since $H$ normalizes $W$, and $k_1 w_i \in W$, we have $h^{-1} (k_1 w_i) h \in W$.
  Thus,
  $h^{-1} k_1  w_i h \cdot k_2 \in W$.
  Hence, $\Supp(S_{h^{-1}k_1} \alpha S_{hk_2})$ meets less than or exactly $m$ cosets of $N$ in $W$. Thus, $\text{min}_N(A)$ is   
$H/W$-invariant. 

(b): By Lemma~\ref{lem:pi_non-degenerate} and Proposition~\ref{prop:nearly_implies_nondegenerate}, it follows that $\pi_N(A)$ is a nonzero ideal of $S_N$.
Take $\alpha \in A$ and $h \in H$. 
Since $H$ normalizes $N$, Lemma~\ref{lem:normalpi} yields $S_{h^{-1}W} \pi_N(\alpha) S_{hW} = \pi_N(S_{h^{-1}W} \alpha S_{hW}) \subseteq \pi_N(S_{h^{-1}W} A S_{hW}) \subseteq \pi_N(A),$
where the last inclusion follows by the $H/W$-invariance of $A$.
This shows that $\pi_N(A)$ is  
$H/W$-invariant. 
\end{proof}

\begin{lem}[{cf.~\cite[Lem.~4.2]{passman1984infinite}}]\label{lem:passman2}
Suppose that $S$ is nearly epsilon-strongly $G$-graded and that $H$ is a subgroup of $G$. Let $N \lhd W$ be subgroups of $G$ such that $N,W$ are normalized by $H$ and $W/N$ is a unique product group. Furthermore, let $A, B$ be nonzero ideals of $S_W$ such that $AB=\{0\}$. Then there exist nonzero ideals $A', B'$ of $S_N$ such that $A' B' =\{0\}$. Moreover, the following assertions hold:
\begin{enumerate}[{\rm (a)}]
\begin{item}
If $A$ (resp. $B$) is  
$H/W$-invariant, then $A'$ (resp. $B'$) is  
$H/W$-invariant.
\end{item}
\begin{item}
If $A, B \subseteq IS_W$ for some ideal $I \subseteq S_e$, then $A', B' \subseteq I S_N$. 
\end{item}
\end{enumerate}
\end{lem}
\begin{proof}
Put $A' := \pi_N(\text{min}_N(A))$ and $B' := \pi_N(\text{min}_N(B))$, and note that they are both ideals of $S_N$ by Lemma~\ref{lem:8_3}.  
Let $\alpha = \sum_{x\in G} \alpha_x \in A$ and $\beta = \sum_{x \in G} \beta_x \in B$ be generators of $\text{min}_N(A)$ and $\text{min}_N(B)$, respectively.

Consider the induced $W/N$-grading on $S_W$ (see Section~\ref{sec:indgrad}). With this grading, $S_W$ has principal component $S_N$. Moreover, it follows from Proposition~\ref{prop:lannstrom1} that $S_W$ is a nearly epsilon-strongly $W/N$-graded ring. Thus, we may
w.l.o.g. 
assume that $N=\{e\}$. 

Now, using the fact that $W$ is a unique product group, we write $x_0 y_0$ for the unique product of $(\Supp(\alpha)) (\Supp(\beta))$ and deduce from $\alpha \beta \subseteq A B = \{0\}$ that $\alpha_{x_0} \beta_{y_0} = 0$, since no cancelling can occur. But then $\alpha \beta_{y_0} = \sum_{x \in G} \alpha_x \beta_{y_0}$ has smaller support size than that of $\alpha$. Since $\alpha$ meets a minimal number of cosets of $N$, it follows that $\alpha \beta_{y_0} = 0$. Hence, $\alpha_x \beta_{y_0} = 0$ for every $x \in W$, which in turn implies that $\alpha_x \beta$ has smaller support size than that of  $\beta$. As a result, we must have $\alpha_x \beta= 0$. In consequence, we have $\alpha_x \beta_y = 0$ for all $x, y \in W$, and hence $\pi_N(\alpha) \pi_N(\beta) = \alpha_e \beta_e = 0$. Thus, $A' B' = \{0\}$. 

Finally, we prove (a) and (b). 
If $A$ is  
$H/W$-invariant, then it follows from Lemma~\ref{lem:8_3} that $A'$ is 
$H/W$-invariant.
Next, suppose that $A \subseteq I S_W$. Then, $\text{min}_N(A) \subseteq A \subseteq I S_W$. Hence, by Lemma~\ref{lem:bimodule_homo}, $A' = \pi_N(\text{min}_N(A)) \subseteq \pi_N(I S_W) \subseteq IS_N$. 
The proof of the corresponding statements for $B$ and $B'$ is completely analogous.
\end{proof}

\begin{prop}
\label{prop:suff1}
Suppose that $S$ is nearly epsilon-strongly $G$-graded.
If $S$ is not prime, then it has an 
NP-datum $(H,N,I,\tilde{A},\tilde{B})$ for which $\tilde{A},\tilde{B}$ are $H/N$-invariant.
\end{prop}
\begin{proof}
If $S_e$ is not  
$G$-semiprime, then the desired conclusion follows from Corollary~\ref{cor:not_semiprime}.
Now, suppose that $S_e$ is 
$G$-semiprime. Then Proposition~\ref{prop:main1} provides us with a subgroup $H$ of $G$ and an $H$-invariant ideal $I$ of $S_e$ such that $I^x I = \{0\}$ for every $x \in G \setminus H$. In particular, condition~(NP2) holds.

To proceed, we apply Lemma~\ref{lem:not_prime}, which yields nonzero $H/\Delta(H)$-invariant ideals $A_1, B_1$ of $S_{\Delta(H)}$ such that $A_1 B_1 = \{0\}$.  Moreover, by Lemma~\ref{lem:8_1} there exists a finitely generated normal subgroup $W$ of $H$ with $W \subseteq \Delta(H)$ and nonzero $H/W$-invariant ideals $A_2, B_2$ of $S_W$ such that $A_2 B_2 = \{0\}$.

Next, by Proposition~\ref{prop:8_2} there is a finite characteristic subgroup $N \lhd W$ such that $W/N$ is torsion-free abelian. Since $N$ is a characteristic subgroup, we get that $N \lhd W \lhd H$. 
This establishes condition~(NP1). 
Moreover, by a well-known result by Levi \cite{levi1942ordered}, $W/N$ is an ordered group, and hence a unique product group. Note that $H$ normalizes $N$ and $W$. 
This means that Lemma~\ref{lem:passman2} is at our disposal, i.\,e. there are nonzero $H/W$-invariant, and in particular $H/N$-invariant, ideals $\tilde A, \tilde B$ of $S_N$ such that $\tilde A, \tilde B \subseteq IS_N$ and $\tilde A \tilde B  = \{0\}$.  Hence, condition~(NP3) holds.
This shows that $(H,N,I,\tilde{A},\tilde{B})$ is an NP-datum for $S$.
\end{proof}

\section{Proof of the main theorem}
\label{Sec:Proofs}

In this section, we finish the proof of Theorem~\ref{thm:mainNew} and show that Passman's result 
(see Theorem~\ref{maintheorem}) can be recovered from it. 

\begin{proof}[Proof of Theorem~\ref{thm:mainNew}]

(1)
Suppose that $S$ is non-degenerately $G$-graded.

(e)$\Rightarrow$(d):
Suppose that (e) holds.
By Lemma~\ref{lem:induced_grading_ideals}, $\tilde{A},\tilde{B}$ are $H$-invariant.
It only remains to show that $\tilde{A} S_H \tilde{B} = \{0\}$.
Take $x\in H$.
Seeking a contradiction, suppose that $\tilde{A} S_{xN} \tilde{B} \neq \{0\}$.
Note that $\tilde{A} S_{xN} \tilde{B} \subseteq S_{xN}$. 
By non-degeneracy of the $G$-grading on $S$, it follows that $S_H$ is non-degenerately $H$-graded.
Hence, by Proposition~\ref{prop:QuotientNonDeg}, the $H/N$-grading on $S_H$ is also non-degenerate.
Consequently,
$S_{x^{-1}N} \tilde{A} S_{xN} \tilde{B} \neq \{0\}$.
By the $H/N$-invariance of $\tilde{A}$ we get that
$\{0\} \neq S_{x^{-1}N} \tilde{A} S_{xN} \tilde{B} \subseteq \tilde{A} \tilde{B} = \{0\}$
which is a contradiction.
We conclude that
 $\tilde{A}S_x \tilde{B} \subseteq \tilde{A} S_{xN} \tilde{B} = \{0\}$.
Thus, $\tilde{A} S_H \tilde{B}=\{0\}$.

(d)$\Rightarrow$(c)$\Rightarrow$(b):  This is trivial.

(b)$\Rightarrow$(a): 
This follows from Proposition~\ref{prop:easy_dir}.

(2)
Suppose that $S$ is nearly epsilon-strongly $G$-graded.
By Proposition~\ref{prop:nearly_implies_nondegenerate}, $S$ is non-degenerately $G$-graded.
Hence, by (1) we get that
(e)$\Rightarrow$(d)$\Rightarrow$(c)$\Rightarrow$(b)$\Rightarrow$(a). The remaining implication,
(a)$\Rightarrow$(e),
 follows from
 Proposition~\ref{prop:suff1}. 
\end{proof}

\begin{proof}[Proof of Theorem~\ref{maintheorem}]
Let $S$ be a unital strongly $G$-graded ring. 
The claim of Theorem~\ref{maintheorem}
follows immediately from Remark~\ref{rem:NP3NP4}
and the equivalence (a)$\Leftrightarrow$(d) in Theorem~\ref{thm:mainNew}.
\end{proof}

\section{Applications for torsion-free grading groups}
\label{Sec:Sufficient}

Recall that $S$ is a $G$-graded ring.
In this section, we pay special attention to the case when 
$G$ is torsion-free.
The following result generalizes Corollary~\ref{Cor:Ordered} and establishes Theorem~\ref{thm:NearlyTorsion}:

\begin{thm}[{cf.~\cite[Cor.~4.6]{passman1984infinite}}]\label{thm:StrongPrimeSuffTorsionFree}
Suppose that $G$ is torsion-free and that $S$ is nearly epsilon-strongly $G$-graded.
Then $S$ is prime if and only if $S_e$ is $G$-prime.
\end{thm}

\begin{proof}
Suppose that $S$ is not prime.
By Theorem~\ref{thm:mainNew}, 
there is a balanced NP-datum
\linebreak 
$(H,N,I,\tilde{A},\tilde{B})$ for $S$.
Using that $G$ is torsion-free, we conclude that $N=\{e\}$. 
In consequence, $S_N=S_e$ and $I,\tilde{A},\tilde{B}$ are all ideals of $S_e$.
Consider the sets $\tilde{A}^G$ and $\tilde{B}^G$. 
By Proposition~\ref{lem:power} 
they are nonzero $G$-invariant ideals of $S_e$.
Note that
$\tilde{A} S_x \tilde{B}=\{0\}$ for every $x\in G$ by the same argument as  
in the proof of Proposition~\ref{prop:easy_dir}.
Using this, we get that
$\tilde{A}^G \tilde{B}^G =\{0\}$
and hence $S_e$ is not $G$-prime.

Now suppose that $S$ is prime.
By Corollary~\ref{cor:PrimeNecessary}, it follows that $S_e$ is $G$-prime.
\end{proof}

\begin{rem}
Note that a strongly $G$-graded ring with local units is necessarily nearly epsilon-strongly $G$-graded (see Lemma~\ref{lem:strongSunital}). Hence, \cite[Thm.~3.1]{AbHa93} by Abrams and Haefner follows from Theorem~\ref{thm:StrongPrimeSuffTorsionFree}.
\end{rem}

The following corollary is similar to a result by \"{O}inert \cite[Thm.~4.4]{oinert2019units}:

\begin{cor}\label{cor:torsion_free}
Suppose that $G$ is torsion-free and that $S$ is nearly epsilon-strongly $G$-graded. 
If $S_e$ is prime, then $S$ is prime.
\end{cor}

\begin{exa}
    Let $R$ be a unital ring, let $u$ be an idempotent of $R$, and let $\alpha:R \to uRu$ be a corner ring isomorphism.
    In this example we consider the corner skew Laurent polynomial ring
    $R[t_+,t_{-},\alpha]$ which was introduced by Ara, Gonzalez-Barroso, Goodearl and Pardo in~\cite{ara2004fractional}. 
    For the convenience of the reader we now briefly recall its definition: 
    $R[t_+,t_{-},\alpha]$ is the universal unital ring satisfying the following two conditions:
    \begin{enumerate}[(a)]
        \item 
            there is a unital ring homomorphism $i : R \to R[t_+,t_{-},\alpha]$;
        \item
            $R[t_+,t_{-},\alpha]$ is the $R$-algebra satisfying the following equations for every $r \in R$:
            \begin{align*}
                t_{-} t_+ = 1, \qquad t_{+} t_{-} = i (u), \qquad r t_{-} = t_{-} \alpha(r), \qquad t_+ r = \alpha(r) t_+.
            \end{align*}
    \end{enumerate}
    Assigning degrees $-1$ to $t_{-}$ and $1$ to $t_+$ turns $R[t_+,t_{-},\alpha]$
    into a $\Z$-graded ring with principal component~$R$.
    By~\cite[Prop.~8.1]{lannstrom2019graded}, $R[t_+,t_{-},\alpha]$ is nearly epsilon-strongly $\Z$-graded.
    Hence, if $R$ is prime, then it follows from Corollary~\ref{cor:torsion_free} that $R[t_+,t_{-},\alpha]$ is also prime.
    Of course, when $u=1$ and $\alpha$ is the identity map, then $R[t_+,t_{-},\alpha]$ is the familiar ring $R[t,t^{-1}]$.
\end{exa}

\section{Applications to $s$-unital strongly graded rings}
\label{Sec:Strongly}

In this section, we 
apply our results to $s$-unital strongly $G$-graded rings.
Recall that, by Lemma~\ref{lem:strongSunital}, every $s$-unital strongly $G$-graded ring is nearly epsilon-strongly $G$-graded.
Thus, by Theorem~\ref{thm:mainNew}, we obtain the following
$s$-unital generalization of Passman's Theorem~\ref{maintheorem}:

\begin{cor}\label{cor:PassmanSunital}
Suppose that $S$ is an $s$-unital strongly $G$-graded ring.
Then $S$ is not prime if and only if
it has an NP-datum $(H,N,I,\tilde{A},\tilde{B})$ for which $\tilde{A},\tilde{B}$ are both $H$-invariant.
\end{cor}

\subsection{Morita context algebras}

Let $S$ be an $s$-unital strongly $G$-graded ring. 
For every $x \in G$ the canonical multiplication map $m_x :  S_x \otimes_{S_e} S_{x^{-1}} \to S_e$, $	a \otimes b \mapsto ab$ 
is an isomorphism of $S_e$-bimodules. 
Indeed, $m_x$ is well-defined and surjective, using that $S$ is strongly $G$-graded. 
Moreover, the injectivity is a consequence of the $s$-unitality. Noteworthily, by associativity of the multiplication, for every $x \in G$ we also have
\begin{alignat*}{2}
	m_x \otimes \text{id} &= \text{id} \otimes m_{x^{-1}} &&: S_x \otimes_{S_e} S_{x^{-1}} \otimes_{S_e} S_x \to S_x
	\\
	m_{x^{-1}} \otimes \text{id} &= \text{id} \otimes m_x  &&: S_{x^{-1}} \otimes_{S_e} S_x \otimes_{S_e} S_{x^{-1}} \to S_{x^{-1}}.
\end{alignat*}
Thus, for every $x \in G$ we get a quintupel $(S_e,S_x,S_{x^{-1}},m_x,m_{x^{-1}})$ which is usually referred to as a \emph{strict Morita context}. 

Next, let us consider an $s$-unital ring $R$ and a strict Morita context $(R,M,N,\mu_1,\mu_{-1})$, i.\,e. we have $R$-bimodules $M, N$ and $R$-bimodule isomorphisms 
\begin{align*}
	\mu_1: M \otimes_R N \to R, 
	\qquad
	\mu_{-1}: N \otimes_R M \to R
\end{align*}
satisfying the mixed associativity conditions $\mu_1 \otimes \text{id} = \text{id} \otimes \mu_{-1}$ and $\mu_{-1} \otimes \text{id} = \text{id} \otimes \mu_1$. 
Furthermore, we assume that $RM=MR=M$ and $RN=NR=N$.
We form a $\Z$-graded module $S$ by putting
\begin{align*}
	S_n:=
	\begin{cases}
		R & \quad n = 0
		\\
		M^{\otimes^n_R} & \quad n > 0
		\\
		N^{\otimes^{-n}_R} & \quad n < 0.
	\end{cases}
\end{align*}
We wish to turn $S$ into a $\Z$-graded ring. 
The product of two positively graded elements is just the usual tensor product $\otimes_R$ of tensor products of $M$'s, and  similarly the product of two negatively graded elements is just the usual tensor product of $N$'s.
To deal with products of mixed elements, we repeatedly make use of the maps $\mu_1$ and $\mu_{-1}$. 
By the mixed associativity conditions, this multiplication becomes associative, and hence $S$ is a $\Z$-graded ring as desired.
In addition, as the maps $\mu_1$ and $\mu_{-1}$ are surjective, we may infer that $S$ is strongly $\Z$-graded.
Clearly, $S$ is $s$-unital.
Finally, Theorem~\ref{thm:StrongPrimeSuffTorsionFree} implies that if $R$ is $\Z$-prime, then $S$ is prime.

\subsection{$s$-unital strongly graded matrix rings}

In what follows, let $R$ be an $s$-unital ring. 
Let $M_\Z(R)$ denote the ring of infinite $\Z \times \Z$-matrices with only finitely many nonzero entries in $R$. 
For $r \in R$ and $i,j \in \Z$ we write $r e_{i,j}$ for the
matrix in $M_\Z(R)$ with $r$ in the $ij$th position and zeros elsewhere. 
We regard $M_\Z(R)$ as a $\Z$-graded ring with respect to 
\begin{align}
	\deg(r e_{i,j}) := i-j \quad \text{for all} \ i,j \in \Z \ \text{and all nonzero} \ r \in R.\label{eq:deg}
\end{align}
The corresponding homogeneous components of the $\Z$-grading are given by
\begin{align*}
	(M_\Z(R))_k = \bigoplus_{i \in \Z} R e_{i+k,i}, \qquad k \in \Z.
\end{align*}
In particular, $(M_\Z(R))_0 = \bigoplus_{i \in \Z} R e_{i,i}$ is the main diagonal.

\begin{lem}\label{lem:Zstronglygraded}
The ring $M_\Z(R)$ is $s$-unital and strongly $\Z$-graded with respect to the  
grading defined by \eqref{eq:deg}.
\end{lem}

\begin{proof}
Put $S := M_{\Z}(R)$.
By Proposition~\ref{prop:tominaga} and $s$-unitality of $R$, it follows that $S$ is $s$-unital and that
$S_0 S_n = S_n S_0 = S_n$, for every $n \in \Z$.
Take $k \in \Z$. Since $R$ is $s$-unital, and hence idempotent, we get that
$S_k S_{-k} = 
( \sum_{i \in \Z} R e_{i+k,i} ) ( \sum_{j \in \Z} R e_{j-k,j} ) = \sum_{i \in \Z} R^2 e_{i+k,i} e_{i,i+k} = \sum_{i \in \Z} R e_{i+k,i+k} = S_0$.
The claim now follows from Proposition~\ref{prop:stefan}.
\end{proof}

\begin{cor}\label{thm:matrix_prime}
The ring $M_\Z(R)$ is prime if and only if $R$ is prime.
\end{cor}
\begin{proof}
    Suppose that $R$ is not prime, i.\,e. there are nonzero ideals $A,B$ of $R$ such that $AB=\{0\}$. 
    Then $M_\Z(A) \cdot M_\Z(B) = \{0\}$ which shows that $M_\Z(R)$ is not prime.
    Conversely suppose that $R$ is prime.
    Note that any ideal $I$ of $(M_\Z(R))_0$ is of the form $I = \bigoplus_{i \in \Z} I_i e_{i,i}$
for some family of $R$-ideals $I_i$, $i \in \Z$, and it is $\Z$-invariant if and only if $I_i = I_0$ for every $i \in \Z$. 
Next, let $A,B$ be $\Z$-invariant ideals of $(M_\Z(R))_0$ such that $A B = \{ 0 \}$.
There are $R$-ideals $A_0,B_0$ such that $A = \bigoplus_{i \in \Z} A_0 e_{i,i} $ and $B = \bigoplus_{i \in \Z} B_0 e_{i,i} $.
Since $AB = \{0\}$, we see that $A_0B_0=\{0\}$ and thus $A_0 = \{0\}$ or $B_0 = \{0\}$ due to the primeness of $R$. 
Hence, $A = \{0\}$ or $B = \{0\}$.
Consequently, $(M_\Z(R))_0$ is $\Z$-prime and hence $M_\Z(R)$
is prime by 
Theorem~\ref{thm:StrongPrimeSuffTorsionFree}.
\end{proof}

\begin{rem}
    The above result shows that primeness of the principal component is not a necessary condition for primeness of a strongly graded ring.
    Nevertheless, by Corollary~\ref{cor:PrimeNecessary}, $G$-primeness of $S_e$ is a necessary condition.
\end{rem}

Now we fix $n \in \N$ and consider $M_n(R)$, the ring of $n \times n$-matrices with entries in $R$. 
The ring $M_n(R)$ comes equipped with a natural $\Z$-grading defined by
\begin{align}
	\deg(r e_{i,j}) := i-j \quad \text{for all} \
	i,j \in \{ 1,\ldots,n \}
	\
	\text{and all nonzero} 
	\
	r \in R.\label{eq:deg_n}
\end{align}

\begin{lem}\label{lem:mnRepsilon}
The ring $M_n(R)$ is nearly epsilon-strongly $\Z$-graded 
with respect to the  
	grading defined by~\eqref{eq:deg_n}.
\end{lem}

\begin{proof}
Put $S := M_n(R)$.
Take $k \in \Z$ and $r \in R$. 
Note that for $i,j$ such that $i-j = k$, and $a,b,c \in R$
such that $abc = r$,
we have
$a e_{i,j},c e_{i,j} \in S_k$,
$b e_{j,i} \in S_{-k}$ and
$a e_{i,j} b e_{j,i} c e_{i,j} = r e_{i,j}$. 
Take $s \in S_k$. 
Then $s = \sum_{i-j=k} r_{i,j} e_{i,j} \in S_k$ for some 
$r_{i,j} \in R$. 
By Proposition~\ref{prop:tominaga} and $s$-unitality of $R$,
there is $u \in R$ such 
that $u r_{i,j} = r_{i,j} u = r_{i,j}$ for all $i,j$.
Put $v := \sum_{i-j=k} u e_{i,j} u e_{j,i} \in S_k S_{-k}$ and 
$w := \sum_{i-j=k} u e_{j,i} u e_{i,j} \in S_{-k} S_k$. 
Then 
$vs = s$ and $sw = s$. 
This shows that $S$ is nearly epsilon-strongly $\Z$-graded.
\end{proof}

Note that if $R$ is prime, then $(M_n(R))_0$ is $\Z$-prime.
Hence, by Corollary~\ref{Cor:Ordered} and \linebreak Lemma~\ref{lem:mnRepsilon}, we obtain the following $s$-unital generalization of a well-known result: 

\begin{cor}[{cf.~\cite[Prop.~10.20]{lam2001first}}]
	\label{cor:finmatrixprime}
	The ring $M_n(R)$ is prime if and only if $R$ is prime.
\end{cor}

The $\Z$-grading on $M_n(R)$ defined above induces a $\Z /  n \Z$-grading on $M_n(R)$ (see Section~\ref{sec:indgrad}). 
By Lemma~\ref{lem:mnRepsilon} and Proposition~\ref{prop:lannstrom1}, this turns $M_n(R)$ into a nearly epsilon-strongly $\Z / n \Z$-graded ring.
By using an argument similar to the one in the proof of Lemma~\ref{lem:Zstronglygraded},
it is not difficult to see that this grading is, in fact, strong.
Hence, Corollary~\ref{cor:PassmanSunital} is applicable but presently it is not clear to the authors how to use it to prove Corollary~\ref{cor:finmatrixprime}.

\section{Applications to $s$-unital skew group rings}
\label{Sec:GroupRing}

Connell \cite{connell1963group} famously gave a characterization of when a unital group ring $R[G]$ is prime. 
In this section, we generalize and recover his result from our main theorem.
More precisely, we describe when an $s$-unital group ring $R[G]$ is prime. 

Let $R$ be a (possibly non-unital) ring and let $\alpha : G \to \Aut(R)$ be a group homomorphism.
We define the~\emph{skew group ring} $R \star_\alpha G$ as the set of all formal sums
$\sum_{x\in G} r_x \delta_x$ where $\delta_x$ is a symbol for each $x\in G$
and $r_x\in R$ is zero for all but finitely many $x\in G$.
Addition on $R \star_\alpha G$ is defined in the natural way and multiplication is defined by linearly extending the rules
$r \delta_x r' \delta_y = r\alpha_x(r') \delta_{xy}$, for all $r,r'\in R$ and $x,y\in G$.
This yields an associative ring structure on $R \star_\alpha G$.
Moreover, $S=R \star_\alpha G$ is canonically $G$-graded by putting $S_x := R \delta_x$ for every $x \in G$.
If $\alpha_x = \identity_R$ for every $x \in G$, then we simply write $R[G]$ for $R \star_\alpha G$ and call it a \emph{group ring}.
Note that $R \star_\alpha G$ is a so-called \emph{partial skew group ring} (see Section~\ref{Sec:Partial}).

\begin{prop}\label{prop:patrik1}
Suppose that $R$ is a ring and that $\alpha : G \to \Aut(R)$ is a group homomorphism.
The following assertions are equivalent:
\begin{enumerate}[{\rm (a)}]
\begin{item}
$R$ is idempotent;
\end{item}
\begin{item}
$R \star_\alpha G$ is strongly $G$-graded;
\end{item}
\begin{item}
$R \star_\alpha G$ is symmetrically $G$-graded.
\end{item}
\end{enumerate}
\end{prop}
\begin{proof}
(a)$\Rightarrow$(b): Suppose that $R$ is idempotent, i.\,e. $R^2=R$.
Then for all $x,y\in G$ we have $(R \delta_x)(R \delta_y) = R \alpha_x(R) \delta_{xy} = R \delta_{xy}$.
In other words, $R \star_\alpha G$ is strongly $G$-graded.

(b)$\Rightarrow$(c): This holds in general for strongly $G$-graded rings (see \cite[Prop.~4.45]{lannstrom2019structure}).

(c)$\Rightarrow$(a): This holds in general for symmetrically $G$-graded rings (see \cite[Prop.~4.47]{lannstrom2019structure}).
\end{proof}

\begin{prop}\label{prop:patrik2}
Suppose that $R$ is a ring and that $\alpha : G \to \Aut(R)$ is a group homomorphism. 
The following assertions are equivalent:
\begin{enumerate}[{\rm (a)}]
\begin{item}
$R$ is $s$-unital;
\end{item}
\begin{item}
$R \star_\alpha G$ is $s$-unital strongly $G$-graded;
\end{item}
\begin{item}
$R \star_\alpha G$ is nearly epsilon-strongly $G$-graded.
\end{item}
\end{enumerate}
\end{prop}
\begin{proof}
(a)$\Rightarrow$(b):
Suppose that $R$ is $s$-unital. In particular, $R$ is idempotent. Hence, $R \star_\alpha G$ is strongly $G$-graded by Proposition~\ref{prop:patrik1}. 
It is easy to see that $R \star_\alpha G$ is $s$-unital.

(b)$\Rightarrow$(c):
This follows from Lemma~\ref{lem:strongSunital}.

(c)$\Rightarrow$(a):
This holds for any nearly epsilon-strongly graded ring (see Proposition~\ref{prop:s-unital}).
\end{proof}

\begin{exa}
	In this example we consider the $s$-unital ring $M_\N(\R)$ of $\N \times \N$-matrices with only finitely many nonzero entries in $\mathbb{R}$. 
	Recall that the group $\text{SO}_3(\R)$ of rotations in $\R^3$ contains a subgroup $F$ isomorphic to free group of rank 2 (see e.\,g.~\cite{Haus14,Strom79}).
	For every $x \in F \subseteq \text{SO}_3(\mathbb{R})$ we may define a diagonal matrix $\text{diag}(x,x,x,\ldots)$ which is row-finite and column-finite but does not belong to $M_\N(\R)$.
  We thus obtain a group homomorphism $\alpha : F \to \Aut\bigl(M_\N(\R)\bigr)$ by putting
\begin{align*}
	\alpha_x(a):= \text{diag}(x,x,x,\ldots) \, a \, \text{diag}(x^{-1},x^{-1},x^{-1},\ldots)
\end{align*}
for $x \in F$ and $a \in M_\N(\R).$
    Since $M_\N(\R)$ is simple and $F$ is torsion-free, it follows from Corollary~\ref{cor:torsion_free} that the $s$-unital skew group ring $ M_\N(\R) \star_\alpha F$ is prime.
\end{exa}

We proceed to prove Theorem~\ref{thm:Connel}
by using our main theorem:

\begin{thm}\label{thm:connell2}
Suppose that $R$ is an $s$-unital ring. Then the group ring $R[G]$ is prime if and only if $R$ is prime and $G$ has no non-trivial finite normal subgroup. 
\end{thm}
\begin{proof}
We prove the converse statement: $R[G]$ is not prime if and only if $R$ is not prime or $G$ has a non-trivial finite normal subgroup.
By Proposition~\ref{prop:patrik2}, (a)$\Leftrightarrow$(c) in Theorem~\ref{thm:mainNew}  holds for $S=R[G]$. 
In other words, the group ring $R[G]$ is not prime if and only if it has a balanced NP-datum.
We prove that the $G$-graded ring $R[G]$ has a balanced NP-datum if and only if 
$R$ is not prime or $G$ has a non-trivial finite normal subgroup. 
First note that for any ideal $I$ of $R$ we have $I^x = R\delta_{x^{-1}} I R\delta_x = RIR \delta_e = I \delta_e$ for every $x \in G$. In particular, every ideal of $R$ is $G$-invariant. 

Suppose that
$(H,N,I,\tilde{A},\tilde{B})$ is a balanced NP-datum for $R[G]$.

\underline{Case 1: $H=G$.} 
Note that
$N \lhd H=G$ is a finite normal subgroup of $G$. Condition (NP4) proves that $R[N]$ is not prime. Then either $N=\{e\}$ and $R$ is not prime or there exists a non-trivial finite normal subgroup $N$ of $G$.

\underline{Case 2: $H \subsetneqq G$.} 
Note that condition (NP2) implies that there is a nonzero ideal $I$ of $R$ such that $I^2 = \{ 0 \}$.
Thus $R$ is not prime. 

Now we prove the converse statement.

\underline{Case I: $R$ is not prime.} 
There are nonzero ideals $\tilde A, \tilde B$ of $R$ such that $\tilde A \tilde B = \{ 0 \}$. 
This implies that $\tilde A R \delta_x \tilde B = R \delta_x \tilde A \tilde B = \{0\}$ for every $x \in G$. 
Therefore, $\tilde A \cdot R[G] \cdot \tilde B = \{ 0 \}$. 
We note that $(G,\{e\},R,\tilde{A},\tilde{B})$ is a balanced NP-datum. 

\underline{Case II: there exists a non-trivial finite normal subgroup $N$ of $G$.}

Consider $H := G$ and $I := R$. 
Pick a nonzero $a\in R$. 
Let $\tilde A$ be the ideal of $S_N$ generated by the element $\sum_{n \in N} a\delta_n$ and let $\tilde B$ be the ideal of $S_N$ generated by the set $ \{ r\delta_n - r\delta_e \mid n \in N, r\in R \}$. Since $N$ is non-trivial, it follows that $\tilde A$ and $\tilde B$ are nonzero ideals of $S_N$. 
Next, let $t \in R$, $x \in G$ and $n_1 \in N$. Then, since $N$ is finite and normal in $G$, 
\begin{align*}
    \left(\sum_{n \in N} a\delta_n \right)  t \delta_x  (r\delta_{n_1} - r\delta_e)
		&= \left(\sum_{n\in N} at \delta_{nx} \right)  (r\delta_{n_1} - r\delta_e) = \left(\sum_{n \in N} at\delta_{xn} \right)  (r\delta_{n_1} - r\delta_e)  
		\\
		&= a\delta_x \left(\sum_{n \in N} t\delta_n \right)  (r\delta_{n_1} - r\delta_e) 
		= a\delta_x \left(\sum_{n \in N} tr\delta_{n n_1} - \sum_{n \in N} tr\delta_{ne} \right) \\
		&= a\delta_x \left(\sum_{n \in N} tr\delta_n - \sum_{n \in N} tr\delta_n\right)  =  0.
\end{align*}
This shows that $\tilde A \cdot R[G] \cdot \tilde B = \{ 0 \}.$ Hence, 
$(H,N,I,\tilde{A},\tilde{B})$
is a balanced NP-datum.
\end{proof}

\begin{rem}
Note that Theorem~\ref{thm:connell2} applies to $s$-unital group rings $R[G]$ which are not necessarily unital. Hence, this application shows that our results indeed reach farther than Passman's results \cite{passman1983semiprime,PassmanCancellative,passman1984infinite} which are only concerned with unital rings.
\end{rem}

\begin{rem}\label{rem:not_prime_principal}
The above result can not be generalized to $s$-unital (unital) skew group rings.
Indeed, neither primeness of $R$ nor the non-existence of non-trivial finite normal subroups of $G$
are necessary conditions for primeness of an $s$-unital skew group ring $R \star_\alpha G$.
To see this, consider the matrix algebra
$M_4(\R) \cong \R^4 \star_\alpha \Z/4\Z$
as a unital skew group ring.
It is well-known that $M_4(\R)$ is prime, but $\R^4$ is not prime and $\Z/4\Z$ contains a non-trivial finite normal subgroup.
Note, however, that in this case $\R^4$ is actually $\Z/4\Z$-prime.
\end{rem}

\begin{exa}
Suppose that $G$ is torsion-free and let 
$F(G,\C)$ be the algebra of all complex-valued
functions on $G$ with finite support, under pointwise addition and multiplication.
Note that $F(G,\C)$ is $s$-unital.
We define a map $\alpha : G \to \Aut\big(F(G,\C)\big)$
by putting $\alpha_x(f)(y) := f(x^{-1}y)$ for all $x,y\in G$ and $f \in F(G,\C)$. 
Clearly, $F(G,\C)$ is $G$-prime.
Using Theorem~\ref{thm:StrongPrimeSuffTorsionFree}, we get that $F(G,\C) \star_\alpha G$ is prime.
\end{exa}

\begin{rem}\label{rem:12_8}
Consider the non-unital group ring $R[G]$ where $R:= 2\Z$ and $G:= \Z$. Note that $R$ is not $s$-unital and hence $R[G]$ is not nearly epsilon-strongly $G$-graded (see Proposition~\ref{prop:s-unital}). It is, however, non-degenerately $G$-graded. In fact, it is not difficult to see that $R[G]$ is a domain and hence prime. From this
we easily see that the equivalences (a)$\Leftrightarrow$(b)$\Leftrightarrow$(c)$\Leftrightarrow$(d)$\Leftrightarrow$(e) in Theorem~\ref{thm:mainNew} hold for $R[G]$. This example suggests that it might be possible to generalize Theorem~\ref{thm:mainNew}.
\end{rem}

\section{Applications to crossed products defined by partial actions}
\label{Sec:Partial}

A significant development in the study of $C^*$-algebras was the introduction of the notion of a \emph{partial action} by Exel \cite{exel1994circle}. Various algebraic analogues of this notion were developed and studied during the last two decades (see e.\,g. \cite{bagio2010crossed,dokuchaev2005associativity,dokuchaev2008crossed}).

In this section, we apply our main theorem to obtain results  
on primeness of 
$s$-unital partial skew group rings (see Section~\ref{SubSec:partialSGR}) and of unital partial crossed products (see Section~\ref{SubSec:UPCP}).
We also apply our results to
some particular examples of partial skew group rings associated with partial dynamical systems (see Section~\ref{SubSec:PartialDynSys}).

\subsection{Partial skew group rings}\label{SubSec:partialSGR}

Recall that a \emph{partial action of $G$ on an $s$-unital ring $R$} (see \cite[p.~1932]{dokuchaev2005associativity}) is a pair 
$
( \{\alpha_g \}_{g \in G}, \{ D_g \}_{g \in G} )
$,
where for all $g,h \in G$, $D_g$ is a (possibly zero) $s$-unital ideal of $R$, $\alpha_g \colon D_{g^{-1}} \to D_g$ is a  ring isomorphism. We require that the following conditions hold for all $g, h \in G$:
\begin{enumerate}[{\rm (P1)}]
\item $\alpha_e = {\rm id}_R$;
\item $\alpha_g(D_{g^{-1}} D_h) = D_g D_{gh}$;
\item if $r \in D_{h^{-1}} D_{(gh)^{-1}}$, 
then $\alpha_g ( \alpha_h (r) ) =
 \alpha_{gh}(r)$.
\end{enumerate}
Given a partial action of $G$ on $R$, we can form the \emph{$s$-unital partial skew group ring} $R \star_\alpha G := \bigoplus_{g \in G} D_g \delta_g$ where the $\delta_g$'s are formal symbols. For $g,h \in G, r \in D_g$ and $r' \in D_h$ the multiplication is defined by the rule:
\begin{displaymath}
(r \delta_g) (r' \delta_h) = \alpha_g(\alpha_{g^{-1}}(r) r') \delta_{gh}
\end{displaymath}
It can be shown 
that $R \star_\alpha G$ is an associative ring (see e.\,g. \cite[Cor.~3.2]{dokuchaev2005associativity}).
Moreover, $S := R \star_\alpha G$ is canonically $G$-graded by putting $S_g := D_g \delta_g$ for every $g\in G$.

\begin{prop}\label{prop:partialSGRgrading}
The canonical $G$-grading on $R \star_\alpha G$ is 
nearly epsilon-strong.
\end{prop}

\begin{proof}
Take $g\in G$. 
Note that
\begin{align*}
S_g S_{g^{-1}}
&=
D_g \delta_g  D_{g^{-1}} \delta_{g^{-1}}
=
\alpha_g( \alpha_{g^{-1}}(D_g) D_{g^{-1}}) \delta_e
=
\alpha_g( D_{g^{-1}} D_{g^{-1}}) \delta_e
=
\alpha_g( D_{g^{-1}}) \delta_e
=
D_g \delta_e
\end{align*}
and hence
\begin{align*}
S_g S_{g^{-1}} S_{g}
&=
(S_g S_{g^{-1}}) S_{g}
=
D_g \delta_e D_g \delta_g
= D_g^2 \delta_g = D_g \delta_g = S_g.
\end{align*}
This shows that the $G$-grading is symmetrical and that $S_g S_{g^{-1}}$ is $s$-unital for every $g\in G$.
By Proposition~\ref{prop:6} the desired conclusion follows.
\end{proof}

\begin{rem}
We will identify $R$ with $R \delta_e$ via the canonical isomorphism.
\end{rem}

\begin{defi}\label{def:Gaction}
Let $H$ be a subgroup of $G$.
An ideal $I$ of $R$ is called \emph{$H$-invariant} if 
$\alpha_{h}( I D_{h^{-1}}) \subseteq I$ for every $h \in H$.
The ring $R$ is called \emph{$G$-prime} if for all $G$-invariant ideals $I,J$ of $R$, we have $I = \{ 0 \}$ or $J = \{ 0 \}$, whenever 
$IJ = \{ 0 \}$.
\end{defi}

\begin{rem}\label{rem:Ginvariance}
Consider $S:=R\star_\alpha G$ with its canonical $G$-grading.

\noindent (a)
Let $H$ be a subgroup of $G$.
Note that, for $h\in H$, we have
\begin{align*}
I^h \subseteq I
& \Longleftrightarrow
D_{h^{-1}} \delta_{h^{-1}} \cdot I \cdot D_h \delta_h  \subseteq I \delta_e 
 \Longleftrightarrow
\alpha_{h^{-1}}(\alpha_h(D_{h^{-1}}) I D_h) \delta_e \subseteq I \delta_e \\
& \Longleftrightarrow
\alpha_{h^{-1}}(D_h I D_h) \delta_e \subseteq I \delta_e  \Longleftrightarrow
\alpha_{h^{-1}}(D_h I D_h) \subseteq I 
 \Longleftrightarrow
\alpha_{h^{-1}}(I D_h) \subseteq I. 
\end{align*}
This shows that $G$-invariance in the sense of Definition~\ref{def:Gaction} is equivalent to $G$-invariance defined by the $G$-grading (see Definition~\ref{def:weakly_invariant}).

\noindent (b) By (a) we note that $R$ is $G$-prime if and only if $S_e$ is $G$-prime.
\end{rem}

\begin{thm}\label{thm:primepartialskewgroupring}
Suppose that 
$G$ is torsion-free 
and that $R \star_\alpha G$
is an $s$-unital partial skew group ring.
Then 
$R \star_\alpha G$ is prime
if and only if
$R$ is $G$-prime.
\end{thm}

\begin{proof}
This follows from 
Proposition~\ref{prop:partialSGRgrading},
Theorem~\ref{thm:StrongPrimeSuffTorsionFree} and
Remark~\ref{rem:Ginvariance}(b).
\end{proof}

We proceed to characterize prime $s$-unital partial skew group rings for general groups.

\begin{lem}\label{lem:InvarianceEquality}
Suppose that $( \{\alpha_g \}_{g \in G}, \{ D_g \}_{g \in G} )$
is a partial action of $G$ on $R$, and that $I$ is an ideal of $R$.
For any subgroup $H$ of $G$, the following holds:
\begin{displaymath}
	\alpha_h(I D_{h^{-1}} ) \subseteq I, \quad \forall h\in H
	\quad
	\Longleftrightarrow
	\quad
	\alpha_h(I D_{h^{-1}}) = I D_h, \quad \forall h\in H
\end{displaymath}
\end{lem}

\begin{proof}
Take $h\in G$.

($\Leftarrow$): Clear, since $I D_h \subseteq I$.

($\Rightarrow$):
Note that
$I \cap D_h = I \cdot D_h$, by $s$-unitality of $D_h$.
Thus, 
$\alpha_h(I D_{h^{-1}}) \subseteq I$ 
implies 
$\alpha_h(I D_{h^{-1}}) \subseteq I \cap D_h = I D_h$.
By applying $\alpha_{h^{-1}}$ to both sides, and using that $h$ is arbitrary, we get
$I D_{h^{-1}} \subseteq \alpha_{h^{-1}}(I D_h) \subseteq I D_{h^{-1}}$.
Hence,
$\alpha_{h^{-1}}(I D_h) = I D_{h^{-1}}$.
\end{proof}

\begin{thm}\label{thm:partialSGR}
The $s$-unital partial skew group ring $R \star_\alpha G$ is not prime if and only if there are:
\begin{enumerate}[{\rm (i)}]
\item 
subgroups $N \lhd H \subseteq G$ with $N$ finite,
\item
an ideal $I$ of $R$ such that
\begin{itemize}
	\item $\alpha_{h}(I D_{h^{-1}}) = I D_h$ for every $h\in H$,
	\item $I D_g \cdot \alpha_g(I D_{g^{-1}})  = \{0\}$ for every $g\in G\setminus H$, and
\end{itemize}

\item
nonzero ideals $\tilde{A}, \tilde{B}$ of $R \star_\alpha N$ such that
$\tilde{A}, \tilde{B} \subseteq I\delta_e ( R \star_\alpha N)$ and $\tilde A \cdot D_h \delta_h \cdot \tilde B = \{0 \}$
for every $h\in H$.
\end{enumerate}
\end{thm}

\begin{proof}
By Proposition~\ref{prop:partialSGRgrading}, we may apply Theorem~\ref{thm:mainNew} to $S:=R \star_\alpha G$.
\noindent For $g\in G$, we get
\begin{align*}
I^g \cdot I = \{0\}
&\Longleftrightarrow
D_{g^{-1}} \delta_{g^{-1}} \cdot I \cdot D_g \delta_g \cdot I \delta_e = \{0\} 
\Longleftrightarrow
\alpha_{g^{-1}}(\alpha_g(D_{g^{-1}}) \cdot I D_g) \delta_e  \cdot I \delta_e = \{0\} \\
&\Longleftrightarrow
(\alpha_{g^{-1}}(D_g \cdot I D_g) \cdot I) \delta_e  = \{0\} 
\Longleftrightarrow
\alpha_{g^{-1}}(D_g I D_g) \cdot I  = \{0\} \\
&\Longleftrightarrow
\alpha_{g^{-1}}(D_g I D_g) \cdot I D_{g^{-1}}  = \{0\}
\Longleftrightarrow
D_g I D_g \cdot \alpha_{g}(I D_{g^{-1}})  = \{0\}.
\end{align*}
Using that $I, D_g$ are ideals of $R$ and that $D_g$ is $s$-unital, we get that $D_g I D_g \subseteq I D_g \subseteq D_g (I D_g)$.
Hence, $D_g I D_g = I D_g$.
We conclude that
$I^g \cdot I = \{0\}$
if and only if
$I D_g \cdot \alpha_g(I D_{g^{-1}})  = \{0\}$.
The desired conclusion now follows by
Remark~\ref{rem:Ginvariance}(a)
and
Lemma~\ref{lem:InvarianceEquality}.
\end{proof}

\subsection{Unital partial crossed products}\label{SubSec:UPCP}

Recall that a \emph{unital twisted partial action of $G$ on a unital ring $R$} (see \cite[p.~2]{nystedt2016epsilon}) is a triple 
$
( \{\alpha_g \}_{g \in G}, \{ D_g \}_{g \in G}, \{ w_{g,h} \}_{(g,h) \in G \times G})
$,
where for all $g,h \in G$, $D_g$ is a unital ideal of $R$, $\alpha_g \colon D_{g^{-1}} \to D_g$ is a  ring isomorphism and 
$w_{g,h}$ is an invertible element in $D_g D_{g h}$. Let $1_g \in Z(R)$ denote the (not necessarily nonzero) multiplicative identity element of the ideal $D_g$. We require that the following conditions hold for all $g, h \in G$:
\begin{enumerate}
\item[{\rm (UP1)}] $\alpha_e = {\rm id}_R$;
\item[{\rm (UP2)}] $\alpha_g(D_{g^{-1}} D_h) = D_g D_{gh}$;
\item[{\rm (UP3)}] if $r \in D_{h^{-1}} D_{(gh)^{-1}}$, 
then $\alpha_g ( \alpha_h (r) ) =
w_{g,h} \alpha_{gh}(r) w_{g,h}^{-1}$;
\item[{\rm (UP4)}] $w_{e,g} = w_{g,e} = 1_g$;
\item[{\rm (UP5)}] if $r \in D_{g^{-1}} D_h D_{hl}$, then
$\alpha_g(r w_{h,l}) w_{g,hl} =
\alpha_g(r) w_{g,h} w_{gh,l}$.
\end{enumerate}

Given a unital twisted partial action of $G$ on $R$, we can form the \emph{unital partial crossed product} $R \star_\alpha^w G := \bigoplus_{g \in G} D_g \delta_g$ where the $\delta_g$'s are formal symbols. For $g,h \in G, r \in D_g$ and $r' \in D_h$ the multiplication is defined by the rule:
\begin{displaymath}
(r \delta_g) (r' \delta_h) = r \alpha_g(r' 1_{g^{-1}}) w_{g,h} \delta_{gh}
\end{displaymath}
It can be shown that $R \star_\alpha^w G$
is an associative ring (see e.\,g. \cite[Thm.~2.4]{dokuchaev2008crossed}). 
Moreover, Nystedt, \"{O}inert and Pinedo established in~\cite[Thm.~35]{nystedt2016epsilon} that its natural $G$-grading is epsilon-strong, and in particular nearly epsilon-strong. 
Thus, Theorem~\ref{thm:mainNew} is applicable.

\begin{rem}\label{rem:UPCprime}
(a) Let $H$ be a subgroup of $G$.
Note that
an ideal $I$ of $R$ is $H$-invariant (in the sense of Definition~\ref{def:Gaction}) if and only if $\alpha_h(I 1_{h^{-1}}) \subseteq I$ for every $h\in H$.

\noindent (b)
We also define $G$-primeness of $R$ according to Definition~\ref{def:Gaction}.
By a computation, similar to the one in 
Remark~\ref{rem:Ginvariance}, we note that $G$-primeness of $R$ is equivalent to $G$-primeness of $S_e$.
\end{rem}

The next result partially generalizes Theorem~\ref{thm:primepartialskewgroupring}.

\begin{thm}\label{thm:primepartialcrossedproduct}
Suppose that 
$G$ is torsion-free 
and that $R \star_\alpha^w G$
is a unital partial crossed product.
Then 
$R \star_\alpha^w G$ is prime
if and only if
$R$ is $G$-prime.
\end{thm}

\begin{proof}
Using the fact that unital partial crossed products are epsilon-strongly graded (see \cite[Thm.~35]{nystedt2016epsilon}),
the desired conclusion follows from 
Theorem~\ref{thm:StrongPrimeSuffTorsionFree} and
Remark~\ref{rem:UPCprime}(b).
\end{proof}

The proof of the following result is similar to the
proof of Theorem~\ref{thm:primepartialskewgroupring} and is therefore omitted.

\begin{thm}\label{thm:partial}
The unital partial crossed product $R \star_\alpha^w G$ is not prime if and only if there are:
\begin{enumerate}[{\rm (i)}]
\item 
subgroups $N \lhd H \subseteq G$ with $N$ finite,
\item
an ideal $I$ of $R$ such that
\begin{itemize}
	\item $\alpha_{h}(I 1_{h^{-1}}) = I 1_h$ for every $h\in H$,
	\item $I \cdot \alpha_g(I 1_{g^{-1}})  = \{0\}$ for every $g\in G\setminus H$, and
\end{itemize}

\item
nonzero ideals $\tilde{A}, \tilde{B}$ of $R \star_\alpha^w N$ such that
$\tilde{A}, \tilde{B} \subseteq I \cdot ( R \star_\alpha^w N)$ and $\tilde A \cdot 1_h \delta_h \cdot \tilde B = \{0 \}$
for every $h\in H$.
\end{enumerate}
\end{thm}

\subsection{Partial dynamical systems}\label{SubSec:PartialDynSys}

In this section we consider several examples 
of partial skew group rings
coming from
a particular type of partial dynamical system (cf.~\cite{Ruy13}).

Let $X$ be a topological space
and let $A_1, A_2, B_1, B_2$ be subspaces of $X$.
Furthermore, let $h_1 : A_1 \to B_1$ and $h_2 : A_2 \to B_2$ be homeomorphisms.
For the remainder of this section, 
$G$ denotes the free group $\mathbb{F}_2=\langle g_1,g_2 \rangle$. For $g\in G$ we define,
\begin{displaymath}
\theta_g = \left\{
\begin{array}{ll}
h_j & \text{if } g=g_j \\
h_j^{-1} & \text{if } g=g_j^{-1} \\
\theta^{\pm 1}_{g_{k_1}}
\circ \cdots \circ
\theta^{\pm 1}_{g_{k_m}} & \text{if } g=g^\pm_{k_1} \cdots g^\pm_{k_m} \text{is in reduced form},
\end{array}
\right .
\end{displaymath}
where $\circ$ denotes partial function composition. 
Moreover, we let $X_g$ denote the domain of the function $\theta_{g^{-1}}$.
We thus obtain a partial action 
of $G$ on the space $X$ which we denote by $(\{ \theta_g \}_{g \in G}, \{ X_g \}_{g \in G})$.
This induces a partial action of $G$ on the $s$-unital ring
$R := C_c(X)$, of continuous compactly supported complex-valued functions on $X$,
 by putting
$D_g := C_c(X_g)$ and defining $\alpha_g : D_{g^{-1}} \to D_g$ by  $\alpha_g(f) := f \circ \theta_{g^{-1}}$ for every $g\in G$. 
Therefore, we may define the $s$-unital partial skew group ring $S:=R\star_\alpha G$.

\begin{exa}\label{exp:par.free,grp}
(a) First, we consider $X=\R$ with
\begin{itemize}
    \item $h_1 : [0,\infty) \to (-\infty,0], \quad t \mapsto -t$, and
    \item $h_2 : \R \to \R, \quad t \mapsto t+1$.
\end{itemize}
It is not difficult to see that $R=C_c(\R)$ is not $G$-prime.
Hence, by Theorem~\ref{thm:partial},  $C_c(\R) \star_\alpha G$ is not prime.

\noindent (b)
Now we consider $X=\R$ with
\begin{itemize}
    \item $h_1 : [0,\infty) \to  [0,\infty), \quad t \mapsto 2t$, and
    \item $h_2 : \R \to \R, \quad t \mapsto t+1$.
\end{itemize}
It is not difficult to see that $R=C_c(\R)$ is $G$-prime.
Hence, by Theorem~\ref{thm:partial},  $C_c(\R) \star_\alpha G$ is  prime.
\end{exa}

\begin{exa}
Now we consider $X$ with its discrete topology.

\noindent (a)
Consider 
$X=\{x_1,x_2,x_3,x_4\}$ with
\begin{itemize}
    \item $h_1 : \{x_1,x_2\} \to  \{x_3,x_4\}$ given by
    $h_1(x_1)=x_3$ and $h_1(x_2)=x_4$, and
    \item $h_2 : \{x_1,x_3\} \to \{x_2,x_4\}$ given by
    $h_2(x_1)=x_2$
    and $h_2(x_3)=x_4$.
\end{itemize}
Note that the ideals of $C_c(X) \cong \mathbb{C}^4$ correspond bijectively to the $2^4$ subsets of $X$. 
For arbitrary elements $x,y\in X$
there is $g\in G$ such that $\theta_g(x)=y$. 
From this we conclude that $R=C_c(X)$ is $G$-prime.
Hence, by Theorem~\ref{thm:partial},  $C_c(X) \star_\alpha G$ is  prime.

\noindent (b)
Consider $X = \{ x_1, x_2, x_3, x_4, x_5, x_6\}$ with
\begin{itemize}
    \item $h_1 : \{x_1,x_2\} \to  \{x_3,x_4\}$ given by
    $h_1(x_1)=x_3$ and $h_1(x_2)=x_4$, and
    \item $h_2 : \{x_1,x_3\} \to \{x_2,x_4\}$ given by
    $h_2(x_1)=x_2$
    and $h_2(x_3)=x_4$.
\end{itemize}
The nonzero ideals $J_1 := C_c(\{x_5 \})$ and $J_2 := C_c(\{x_6 \})$ of
$C_c(X)$ are $G$-invariant. 
Clearly, $J_1 J_2 = \{0\}$ and hence $R=C_c(X)$ is not  
$G$-prime.
By Theorem~\ref{thm:partial},  $C_c(X) \star_\alpha G$ is not prime.
\end{exa}

\section{Applications to Leavitt path algebras}
\label{Sec:LPA}

In this section, we use our main theorem to obtain a characterization of prime Leavitt path algebras with coefficients in an arbitrary, possibly non-commutative, unital ring (see Theorem~\ref{thm:lpa}).
Our result generalizes previous results by Abrams, Bell and Rangaswamy  \cite[Thm.~1.4]{abrams2014prime}, and
Larki~\cite[Prop.~4.5]{larki2015ideal}.

The \emph{Leavitt path algebra} $L_K(E)$ over a field $K$ associated with a directed graph $E$ was introduced by Ara, Moreno and Pardo in \cite{ara2007nonstable} and independently by Abrams and Aranda Pino in \cite{abrams2005leavitt}. 
These algebras are algebraic analogues of graph $C^*$-algebras. For a thorough account of the history and theory of Leavitt path algebras, we refer the reader to the excellent monograph  
\cite{abrams2017leavitt}. Recall that a \emph{directed graph} $E$ is a tuple $(E^0, E^1, s, r)$ where $E^0$ is the set of vertices, $E^1$ is the set of edges and $s \colon E^1 \to E^0$ and $r \colon E^1 \to E^0$ are maps specifying the \emph{source} respectively \emph{range} of each edge. For an arbitrary $v \in E^0$, the set $s^{-1}(v)=\{ e \in E^1 \mid s(e)=v \}$ is the set of edges emitted from $v$. If $s^{-1}(v) = \emptyset$, then $v$ is called a \emph{sink}. If $s^{-1}(v)$ is an infinite set, then $v$ is called an \emph{infinite emitter}. A vertex that is neither a sink nor an infinite emitter is called  \emph{regular}. A \emph{path} in $E$ is a series of edges $\alpha := f_1 f_2 \dots f_n$ such that $r(f_{i}) = s(f_{i+1})$ for $i \in \{ 1,\ldots,n-1 \}$, and such a path has {\it length} $n$ which we denote by $|\alpha|$. By convention, we consider a vertex to be a path of length zero. The set of all paths in $E$ is denoted by $E^*$. 

Leavitt path algebras with coefficients in a commutative unital ring was introduced by Tomforde~\cite{tomforde2009leavitt} and further studied in~\cite{katsov2017simpleness}. A further generalization was studied by Hazrat~\cite{hazrat2013graded}, and Nystedt and \"{O}inert~\cite{nystedt2017epsilon}. Following their lead, we consider Leavitt path algebras with coefficients in a general (possibly non-commutative) unital ring:

\begin{defi}
Let $E$ be a directed graph and let $R$ be a unital ring. The \emph{Leavitt path algebra of the graph $E$ with coefficients in $R$}, denoted by $L_R(E)$, is the free associative $R$-algebra generated by the symbols $\{ v \mid v \in E^0 \} \cup \{ f \mid  f \in E^1 \} \cup \{ f^* \mid f \in E^1 \}$ subject to the following relations:
\begin{enumerate}[{\rm (a)}]
\begin{item}
$ v w = \delta_{v,w} v$ for all $v, w \in E^0$;
\end{item}
\begin{item}
$s(f) f = f r(f)=f$  for every $f \in E^1$;
\end{item}
\begin{item}
$r(f)f^* = f^*s(f)=f^*$  for every $f \in E^1$;
\end{item}
\begin{item}
$f^* f' = \delta_{f, f'} r(f)$ for all $f, f' \in E^1$;
\end{item}
\begin{item}
$\sum_{f \in E^1, s(f)=v}f f^* =v $ for every $v \in E^0$ for which $0 < |s^{-1}(v)| < \infty$. 
\end{item}
\end{enumerate}
We let every element of $R$ commute with the generators.
\label{def:lpa}
\end{defi}
\begin{rem}
By (a), $\{ v \mid v \in E^0 \}$ is a set of pairwise orthogonal idempotents in $L_R(E)$. 
\end{rem}

In the following example, we use Theorem~\ref{thm:bijection1} to describe the  
$\mathbb{Z}$-invariant ideals of
$L_K(E)$.

\begin{exa}\label{ex:lpa_invariant}
Let $K$ be a field and let $E$ be a row-finite directed graph. Recall that there exists a bijection between graded ideals of the Leavitt path algebra $L_K(E)$ and hereditary subsets of $E^0$ (see \cite[Thm.~2.5.9]{abrams2017leavitt}). 
Furthermore, since $L_K(E)$ is naturally nearly epsilon-strongly $\mathbb{Z}$-graded (see \cite[Thm.~30]{nystedt2017epsilon}), we can apply Theorem~\ref{thm:bijection1} to infer that
there is a bijection between hereditary subsets of $E^0$ and  
$\mathbb{Z}$-invariant ideals of $(L_K(E))_0$. 
More precisely, we obtain the following explicit description of the 
$\mathbb{Z}$-invariant ideals of $(L_K(E))_0$: $$ \{ I(H) \mid  H \subseteq E^0 \text{ hereditary vertex set} \}, $$ where $I(H)$ denotes the ideal of $(L_K(E))_0$ generated by the elements $v \in H$.
\end{exa}

Recall that $L_R(E)$ comes equipped with a canonical $\Z$-grading defined by ${\rm deg}(f):=1$ and ${\rm deg}(f^*):=-1$ for every $f \in E^1$, and ${\rm deg}(v):=0$ for every $v \in E^0$ (cf. \cite[Cor.~2.1.5]{abrams2017leavitt}).
In the sequel, the following result will become useful:

\begin{prop}[Nystedt and \"{O}inert~\cite{nystedt2017epsilon}]\label{prop:LPAepsilon}
Suppose that $E$ is a directed graph and that $R$
is a unital ring. Consider $L_R(E)$ with its
canonical $\mathbb{Z}$-grading.
\begin{enumerate}[{\rm (a)}]
     \item $L_R(E)$ is nearly epsilon-strongly $\Z$-graded.
    \item If $E$ is finite, then $L_R(E)$ is
    epsilon-strongly $\Z$-graded.
\end{enumerate}
\end{prop}

In order to begin understanding when a Leavitt path algebra is prime,
we consider a few examples.

\begin{exa}
Let $R$ be a unital ring and let $E_1$ be the directed graph below:
\begin{displaymath}
E_1 : \qquad
	\xymatrix{
	\bullet_{v} 
	}
\end{displaymath}
In this case, $L_R(E_1)=vR \cong R$ is prime if and only if $R$ is prime.
\end{exa}

\begin{exa}\label{ex:13_4}
Let $R$ be a unital ring and let $E_2$ be the directed graph below:
\begin{displaymath}
E_2 : \qquad
	\xymatrix{
	\bullet_{v_1} & \bullet_{v_2} 
	}
\end{displaymath}
We have $L_R(E_2) = v_1 R + v_2 R \cong R \oplus R.$ Note that $v_1 R, v_2 R$ are nonzero ideals of $L_R(E_2)$ such that $(v_1 R)(v_2R) = \{ 0 \}$. 
Thus, 
$L_R(E_2)$ is never prime, 
for any ring $R$.
\end{exa}

From the above examples, it is clear that 
a criterion for primeness of $L_R(E)$ 
must depend on properties of 
both the coefficient ring $R$ and the graph $E$. 
To describe such a criterion, we need to 
introduce a preorder $\geq$ on the set of 
vertices $E^0$ of the directed graph $E$
in the following way: 
We write $u \geq v$ if there is a path (possibly of length zero) from $u$ to $v$. Note that $v \geq v$ for every $v \in E^0$, i.\,e. the preorder is reflexive. Transitivity of the preorder follows by  concatenating the paths. 

\begin{defi}
A directed graph $E$ is said to satisfy \emph{condition~(MT-3)} if the above defined preorder $\geq$ is \emph{downward directed}, i.\,e. if for every pair of vertices $u, v \in E^0$, there is some $w \in E^0$ such that $u \geq w$ and $v \geq w$.
\end{defi}

The graph $E_2$ in Example~\ref{ex:13_4} does not satisfy condition~(MT-3) since there is no vertex $u$ such that $v_1 \geq u$ and $v_2 \geq u$. The next example shows a graph satisfying condition~{(MT-3)}:

\begin{exa}
Let $R$ be a unital ring and let $E_3$ be the directed graph below:
\begin{displaymath}
E_3 : \qquad
	\xymatrix{
	\bullet_{v_1} \ar[r] & \bullet_{v_2} 
	}
\end{displaymath}
$E_3$ satisfies condition~(MT-3). Indeed for $v_1, v_2$ we have $v_1 \geq v_2$ and $v_2 \geq v_2.$ A computation yields $L_R(E_3) \cong M_2(R)$ (see \cite[Expl.~2.6]{lannstrom2019graded}). 
By Corollary~\ref{cor:finmatrixprime}, 
it follows that $L_R(E_3)$ is prime if and only if $R$ is prime.
\end{exa}

In the case when $K$ is a field,
Abrams, Bell and Rangaswamy have shown that $L_K(E)$ is prime if and only if $E$ satisfies condition~(MT-3) (see~\cite[Thm.~1.4]{abrams2014prime}). 
For Leavitt path algebras with coefficients in a commutative unital ring, the following generalization was proved by Larki~\cite[Prop.~4.5]{larki2015ideal}:

\begin{prop}
\label{prop:larki}
Suppose that $E$ is a directed graph and that $R$ is a unital commutative ring. 
Then $L_R(E)$ is prime if and only if $R$ is an integral domain and $E$ satisfies condition~(MT-3).
\end{prop}

We aim to generalize Proposition~\ref{prop:larki} to Leavitt path algebras with coefficients in a general (possibly non-commutative) unital ring. Since Leavitt path algebras are nearly epsilon-strongly $\mathbb{Z}$-graded (see Proposition~\ref{prop:LPAepsilon}), we will be able to obtain this generalization as a corollary to Theorem~\ref{thm:mainNew}. 
We begin with the following result:

\begin{prop}\label{twoideals}
Suppose that $E$ is a directed graph and that $R$ is a unital ring.
Consider the Leavitt path algebra $S = L_R(E)$.
The following assertions hold:
\begin{enumerate}[{\rm (a)}]

\item 
There exist $v,w \in E^0$
such that $S v S w S = \{ 0 \} $ if and only if $E$ does not
satisfy condition~(MT-3).

\item If $R$ is prime and there exist nonzero
$r,s \in R$ and $v,w \in E^0$ such that 
$S r v S s w S = \{ 0 \}$, 
then $E$ does not satisfy condition~(MT-3).

\end{enumerate}
\end{prop}

\begin{proof}
(a): Suppose that $E$ does not satisfy condition~(MT-3). 
There exist $v,w \in E^0$  such that 
for every $y \in E^0$ we have $v \ngeq y$ or $w \ngeq y$.
Take a monomial $r \alpha \beta^*$ in $L_R(E)$.
From the properties of $v$ and $w$, it follows that 
$v r \alpha \beta^* w = 0$.
Therefore 
$ S v S w S = \{ 0\}$.

Now suppose that $E$ satisfies condition~(MT-3).
Take $v,w \in E^0$.
There exist $y \in E^0$ 
and paths $\alpha, \beta$
from $v$ to $y$ and from $w$ to $y$, respectively. 
Then 
$S v S w S \ni v \cdot v \cdot \alpha \beta^* \cdot w \cdot w =
 \alpha \beta^* \neq 0$.

(b): Suppose that $R$ is prime and that there exist nonzero
$r,s \in R$ and $v,w \in E^0$ such that 
$S r v S s w S = \{ 0 \}$. 
Let $P = RrR$ and $Q = RsR$. Then $P$ and $Q$ are nonzero
ideals of $R$. 
Hence, from primeness of $R$ it follows
that $PQ$ is a nonzero ideal of $R$.
Take $p_i \in P$ and $q_i \in Q$, for $i \in\{1,\ldots,n\}$,
such that $\sum_{i=1}^n p_i q_i \neq 0$. 
Seeking a contradiction, suppose that $E$ satisfies condition~(MT-3).
There exist $y \in E^0$ 
and paths $\alpha, \beta$
from $v$ to $y$ and from $w$ to $y$, respectively.
We get 
$	\{ 0 \} =  S r v S s w S 
	\supseteq
	S \cdot (R r R) v \cdot S \cdot (R s R) w \cdot S
	\ni \sum_{i=1}^n v \cdot p_i v \cdot \alpha \beta^* \cdot q_i w \cdot w =
\sum_{i=1}^n p_i q_i \alpha \beta^* \neq 0,
$ which is a contradiction.
\end{proof}

The following result is a special case of \cite[Thm.~2.2.11]{abrams2017leavitt}. Tomforde \cite{tomforde2009leavitt} established this result for Leavitt path algebras with coefficients in a commutative unital ring. His proof generalizes verbatim 
to Leavitt path algebras with coefficients in a general unital ring. 
For the convenience of the reader,
we include a full proof:

\begin{prop}[{cf.~\cite[Lem.~5.2]{tomforde2009leavitt}}]
\label{tomforde}
Suppose that $E$ is a directed graph and that $R$ is a unital ring.
If $a \in (L_R(E))_0$ is nonzero, then there exist
$\alpha,\beta \in E^*$, 
$v \in E^0$
and a nonzero
$t \in R$ such that $\alpha^* a \beta = tv$.
\end{prop}
\begin{proof}
If we for every $N \in \mathbb{N}$ put
$\mathcal{G}_N := \text{Span}_R \{ \alpha \beta^* \mid \alpha, \beta \in E^*, \, |\alpha| = |\beta| \leq N \}$,
then $(L_R(E))_0 = \bigcup_{N=0}^\infty \mathcal{G}_N$. 
The proof proceeds by induction over $N$. 
Base case: $N=0$. 
Take a nonzero $a \in \mathcal{G}_0$.
Then $0 \ne a = \sum_{i=1}^n r_i v_i$ for some nonzero $ r_i \in R$ and distinct vertices $v_i \in E^0$. 
If we put $\alpha= \beta := v_1$, then $\alpha^* a \beta = r_1 v_1$.

Inductive step: suppose that $N > 0$ and that 
the statement of the proposition holds for every nonzero element in $\mathcal{G}_{N-1}$. 
Take a nonzero $a \in \mathcal{G}_N$. 
Then $a = \sum_{i=1}^M r_i \alpha_i \beta_i^* + \sum_{j=1}^{M'} s_j v_j,$ where $\alpha_i, \beta_i \in E^*$ with $|\alpha_i| = |\beta_i| \geq 1$ and $v_j \ne v_{j'}$ for all $j \ne j'$. 
We consider two mutually exclusive cases.

\underline{Case 1: some $v_j$ is not regular.} 
If $v_j$ is an infinite emitter, then there is some edge $f \in E^1$ with $s(f)=v_j$ such that $f$ is not included in any path $\alpha_i, \beta_i$. 
Put $\alpha = \beta := f $. Then $\alpha^* a \beta = 0 + f^* s_j v_j f = s_j v_j$. If $v_j$ is a sink, then put $\alpha=\beta :=v_j$ and note that $\alpha^* a \beta = s_j v_j$. 

\underline{Case 2: every $v_j$ is regular.} 
Then $v_j = \sum_{s(f)=v_j} ff^*$ for every $j$.
Hence, we may write 
$a = \sum_{i=1}^{M''} r_i \gamma_i \delta_i^*$ where $\gamma_i, \delta_i \in E^*$ with $|\gamma_i| = |\delta_i| \geq 1$.
By regrouping the elements of the sum, we may rewrite it as
$a = \sum_{i=1}^P \sum_{j=1}^Q e_i x_{i,j} f_j^*,$ 
where 
\begin{itemize}
    \item 
        $e_i, f_i \in E^1$ with $e_i \ne e_{i'}$ for $i \ne i'$ and $f_j \ne f_{j'}$ for $j \ne j'$, and
    \item
        $x_{i,j} \in \mathcal{G}_{N-1}$ with $e_{i} x_{i,j} f_j^* \ne 0$ for all $i,j$.
\end{itemize}
Note that $e_1 x_{1,1} f_1^* \ne 0$ implies $r(e_1) x_{1,1} r(f_1) \ne 0$. 
 By the induction hypothesis, there are $\alpha', \beta' \in E^*$ such that $(\alpha')^* r(e_1) x_{1,1} r(f_1) \beta' = tv$ for some $v \in E^0$ and $t \in R$. Put $\alpha := e_1 \alpha'$ and $\beta := f_1 \beta'$. Then $\alpha^* a \beta = (\alpha')^* e_1^* a f_1 \beta' = (\alpha')^* e_1^* e_1 x_{1,1} f_1^* f_1 \beta' = (\alpha')^* r(e_1) x_{1,1} r(f_1) \beta' = tv.$
\end{proof}

We can now establish Theorem~\ref{thm:LPA}:

\begin{thm}\label{thm:lpa}
Suppose that $E$ is a directed graph and that $R$ is a unital ring. 
The Leavitt path algebra $L_R(E)$ is prime
if and only if $R$ is prime and $E$ satisfies condition~(MT-3). 
\end{thm}

\begin{proof}
Put $S := L_R(E)$.
Suppose that $R$ is not prime. There exist
nonzero ideals $I,J$ of $R$ such that $IJ = \{ 0 \}$.
Let $A$ and $B$ be the nonzero ideals in $S$ consisting of
sums of monomials with coefficients in $I$ and $J$,
respectively. 
Then $AB = \{ 0 \}$ which
implies that $S$ is not prime.
Suppose now that $E^0$ does not satisfy condition~(MT-3).
By Proposition~\ref{twoideals}(a), 
there exist $v,w \in E^0$
such that $S v S w S = \{ 0 \}$. 
Consider the nonzero ideals 
$C := S v S$ and $D := S w S$ 
of $S$. 
Then
$CD = SvS S w S \subseteq S v S w S = \{ 0 \}$ 
which shows that 
$S$ is not prime.

Suppose that $S$ is not prime and that $R$ is prime. 
By Proposition~\ref{prop:LPAepsilon}, $S$ is nearly epsilon-strongly $\mathbb{Z}$-graded. 
Theorem~\ref{thm:mainNew} implies that there
exist nonzero ideals $\tilde{A},\tilde{B}$
of $S_0$ such that $\tilde{A} S \tilde{B} = \{ 0 \}$.
Take $a \in \tilde{A} \setminus \{ 0 \}$ and
$b \in \tilde{B} \setminus \{ 0 \}$.
By Propostion~\ref{tomforde}, there exist
$v,w \in E^0$, $r,s \in R \setminus \{ 0 \}$ and
$\alpha,\beta,\gamma,\delta \in E^*$ such that
$\alpha^* a \beta = rv$ and $\gamma^* b \delta = sw$. 
Now, 
$S r v S s w S \subseteq 
S \tilde{A} S \tilde{B} S = \{ 0 \}$
and hence $S r v S s w S = \{ 0 \}$. 
Employing Proposition~\ref{twoideals}(b), we conclude that
$E$ does not satisfy condition~(MT-3).
\end{proof}

\bibliographystyle{abbrv}
\bibliography{short,prime}

\end{document}